\documentclass{m2an}
\newtheorem{remark}{Remark}[section] 
\newtheorem{lemma}{Lemma}[section] 
\newtheorem{example}{Example}[section] 
\numberwithin{equation}{section}
\usepackage{color} 
\usepackage{graphicx} 
\usepackage{algorithm}
\usepackage{algorithmic}
\usepackage{epsfig,mathrsfs} 
\usepackage[bf,SL,BF]{subfigure} 
\usepackage{amsfonts}
\usepackage{mathrsfs}
\usepackage{bm}
\usepackage{lineno,hyperref}
\begin{document}
\title{Multilevel and multiscale schemes for fractional partial differential equations}\thanks{This work was supported by the National Natural Science Foundation of China under
Grant No. 11271173.}
\author{Zhijiang Zhang}\address{School of Mathematics and Statistics,
Gansu Key Laboratory of Applied Mathematics and Complex Systems, Lanzhou University, Lanzhou 730000, P.R. China. zhjzhang14@lzu.edu.cn; dengwh@lzu.edu.cn}
\author{Weihua Deng$^{1}$}
%
%
\begin{abstract}
The  wavelet numerical methods for  the classic PDEs have been well developed, but they  are still not discussed for the fractional PDEs.
 This paper focuses on investigating the applications of wavelet bases to numerically solving fractional PDEs and digging out the potential benefits of wavelet methods comparing with other numerical methods, especially in the aspects of realizing preconditioning, adaptivity, and keeping the Toeplitz structure. More specifically, the contributions of this paper are as follows: 1. the techniques of efficiently generating stiffness matrix with computational cost $\mathcal{O}(2^J)$ are provided for first, second, and any order bases; 2. theoretically and numerically discuss the effective multilevel preconditioner for time-independent equation and multiresolution multigrid method for time-dependent equation, respectively; 3. the wavelet multiscale adaptivity is experimentally discussed and numerically applied to solve the time-dependent (independent) equations. In fact, having reliable, simple, and local regularity indicators is the striking benefit of the wavelet in adaptively solving fractional PDEs.

\end{abstract}
%
%
\maketitle
\section{Introduction}
The continuous time random walk (CTRW), a fundamental model in statistic physics, is a stochastic process with arbitrary distributions of jump lengths and waiting times. When the jump length and/or waiting time distribution(s) are/is power law and the second order moment of jump lengths and/or the first order moment of waiting times are/is divergent, the CTRW describes the anomalous diffusion, i.e., the super and sub diffusive cases; and its Fokker-Planck equation has space and/or time fractional derivative(s) \cite{Metzler:00}. It can be noted that the corresponding fractional PDEs are essentially dealing with the  multiscale phenomena; and generally the fractional PDEs have weaker regularity at the area close to boundary and initial time. Besides anomalous diffusion, the fractional models are also used to characterize the memory and hereditary properties inherent in various materials and processes and, recently, much more scientific applications are found in a variety of fields; see, e.g.,  \cite{Mainardi:10, Meerschaert:11,Raberto:02, Zaslavsky:02} and the references therein.

The obtained analytical solutions of fractional PDEs are usually in the form of transcendental functions or infinite series; and in much more cases, the analytical solutions are not available. Then the approximation and numerical techniques for solving the fractional PDEs become essential and have been developed very fast recently, such as, the finite difference method \cite{Chen:15, Meerschaert:04,Tian:15,Yang:10,Zhang:14}, the finite element method \cite{Deng:08, Ervin:05, Ervin:07,Wang:14,Xu:14}, and the spectral method \cite{Li:09, Li:10, Zayernouri:13}. But the computational expenses and nonuniform regularity are still the big challenges that one faces in numerically solving the fractional PDEs, owing to the nonlocality and potential multiscale characteristics of the fractional derivatives; and basing on the preconditioning, adaptivity, and fast transform techniques to develop high efficient methods seems to be a new trend.  The preconditioning techniques are discussed in \cite{ Moroney:13,Yang:11}, where the Krylov subspace projection is their common theme. Fast transform method and multigrid method are provided in  \cite{Wang:10} and \cite{Pang:12}, respectively.

So far,  there seems to be very limited works \cite{Jafari:11,Saadatmandia:12,Saeed:13, Wang:114} to  solve the fractional  PDEs  or ODEs  by wavelet,  although the wavelet numerical methods for  classical PDEs  or ODEs have been well developed \cite{Cohen:00,Urban:09}.
The goal of this paper is to dig out the potential advantages of wavelets in treating the fractional operators, including preconditioning, multigrid, adaptivity, and keeping the quasi-Toeplitz structure for arbitrary order wavelet bases. More concretely, the clearly  obtained benefits of wavelets for fractional operator consist of the following: 1) stiffness matrix of fractional operator is Toeplitz for scaling bases because of their shift-invariant property (it is not always true for the familiar finite element bases, such as the quadratic or cubic element) and a simple diagonal scaling usually produces a good preconditioner; 2) multiscale coefficients indicate the local regularity, and they can be used as the indicator (local posteriori error estimate seems hard to be obtained for the adaptive finite element method because of the global property of the operator) in the adaptive mesh refinement for controlling the entire computational process and increasing the efficiency, i.e., one only needs to make the local refinement on the subdomain where the wavelet coefficients are larger compared with those of other places. For avoiding all non indispensable complications, we present the main ideas and techniques in their simplest form  and restrict ourselves to the following homogeneously space fractional PDE \cite{Ervin:05,Tian:15, Wang:14}:
\begin{equation}\label{eq:1.1}
	   qu_t+{\bf A}u=f  \qquad \mbox{on}\  \Omega,
	\end{equation}
	where  $ \Omega =(0,1)$, $q=0$ or $1$; and ${\bf A}$ is a $(2-\beta)$-th $(0\le\beta<1)$ order differential operator
	\begin{equation}
	{\bf A}u:=-\kappa_{\beta}D\left(p\ {}_0D_x^{-\beta} + (1-p)\ {}_xD_1^{-\beta}\right)Du
	\end{equation}
  with $\kappa_{\beta}>0$ being the generalized diffusivity; $0\le p\le 1$, $D$ represents a single spatial derivative; ${}_0D_x^{-\beta}$ and ${}_xD_1^{-\beta}$	are the left and right fractional integral operators \cite{Metzler:00},
    being, respectively, defined as
	\begin{eqnarray} \label{eq:1.3}
	{}_0D_x^{-\beta}u:&=&\frac{1}{\Gamma(\beta)}\int_0^x (x-s)^{\beta-1} u(s)\,\mathrm{d}s,\\
	{}_xD_1^{-\beta}u:&=&\frac{1}{\Gamma(\beta)}\int_x^1 (s-x)^{\beta-1} u(s)\,\mathrm{d}s.
	\end{eqnarray}
When $q=1$, one gets the fractional initial boundary value problem (IBVP) with an additional initial condition $ u(x,0)=g(x) $; and  when $q=0$, it is the fractional boundary value problem (BVP), which can also be regarded as the steady state equation of the associated IBVP. Considering the homogeneous boundary condition and using integration by parts, one can easily get
	$D{_0D_x^{-\beta}}Du={}_0D_x^{2-\beta}u$ and $D{}_xD_1^{-\beta}Du={}_xD_1^{2-\beta}u$.
Then one can reduce the model (\ref{eq:1.1}) to a more familiar form, and a basic theoretical framework for its variational solution has been presented in \cite{Ervin:05}; this will enable us to put focus on the wavelet numerical methods themselves.

This paper is organized as follows. In Section 2, we give a brief recall to  the spline scaling and wavelet functions.  They have the closed-form expression, which is of course attractive for the fractional operators. In Section 3, we study the computational formulation with respect to the uniform grids. We first discuss the effective way to construct the algebraic system, and then derive the multilevel wavelet preconditioning and the multiresolution multigrid schemes (MMG) for solving the  BVP and IBVP, respectively. In Section 4, we give some heuristically adaptive algorithms and show how singularities can be easily detected by wavelet, and  put our attention on its efficiency by proposing and testing the adaptive algorithms that concentrate the degrees of freedom in the neighborhood of near singularities. The numerical results are shown in Section 5 and we conclude the paper with some remarks  in the last section.

\section {Preliminaries}
In this section, we collect/present some essential properties  of the scaling functions
and wavelets to make the paper self-contained and more readable. For the details, refer to  \cite{Primbs:10,Primbs:09} and \cite{Cohen:00, Urban:09}. 
First, we give the definitions of the fractional Sobolev spaces used in this paper.
For any $s\ge0$, let  $\mathcal{H}^{s}(\mathbb{R})$  be the Sobolev space of order $s$ on $\mathbb{R}$, and $\mathcal{H}^{s}(\Omega)$ the space of the restriction of the functions from $\mathcal{H}^{s}(\mathbb{R})$.  More specifically,
\begin{equation}
\mathcal{H}^s(\mathbb{R})=\left\{u(x)\in L^2(R)\,\big|\left|u\right|^2_{\mathcal{H}^s(\mathbb{R})}<\infty\right\}
\end{equation}
endowed with the seminorm
\begin{equation}
\left|u\right|^2_{\mathcal{H}^s(\mathbb{R})}= \int_{\mathbb{R}}|\omega|^{2s}\left|\mathscr{F}[u](\omega)\right|^2d\omega
\end{equation}
and the norm
\begin{equation}
\left\|u\right\|^2_{\mathcal{H}^s(\mathbb{R})}=
\int_{\mathbb{R}}\left(1+|\omega|^{2s}\right)\left|\mathscr{F}[u](\omega)\right|^2d\omega\sim \int_{\mathbb{R}}\left(1+|\omega|^2\right)^{\mu}\left|\mathscr{F}[u](\omega)\right|^2d\omega;
\end{equation}
\begin{equation}
\mathcal{H}^{s}(\Omega)=\left\{u\in L^2(\Omega)\,\big| \exists \tilde{u}\in \mathcal{H}^s(\mathbb{R}) {~\rm such ~that~} \tilde{u}|_{\Omega}=u\right\}
\end{equation}
endowed with
\begin{eqnarray}
|u|_{\mathcal{H}^s(\Omega)}=\inf_{\tilde{u}|_{\Omega}=u}\left|\tilde{u}\right|_{\mathcal{H}^s(\mathbb{R})}
{~\rm and~}
 \left\|u\right\|^2_{\mathcal{H}^s(\Omega)}=\left\|u\right\|^2_{L^2(\Omega)}+|u|^2_{\mathcal{H}^s(\Omega)},
 \end{eqnarray}
 where $\mathscr{F}[u]$ denotes the Fourier transform of $u$.
And $\mathcal{H}_0^s(\Omega)$ is defined as the closure of $C_0^{\infty}(\Omega)$ w.r.t. $\|\cdot\|_{\mathcal{H}^s(\Omega)}$.

Let $[x_0,\ldots,x_d]f$ denote the $d$-th order {\em divided difference} of $f$ at the points $x_0,\ldots, x_d$, $t_+^l:=(\max\{0,t\})^l, d\ge 2$; and choose the Schoenberg sequence of knots 
  \begin{equation}
	{\boldsymbol{t}^j}:=\{\underbrace{0,\ldots,0}_{d}, 2^{-j},2\times 2^{-j},3\times 2^{-j},\ldots,1-2^{-j},\underbrace{1,\ldots,1}_{d}\},
	\end{equation}
to define the scaling function sets $\Phi_j=\Big\{ \phi_{j,k}, k\in\triangle_j=\left\{1,\ldots,2^j+d-3\right\}\Big\}$ with
	\begin{equation}\label{scale_base}
	\phi_{j,k}(x):=2^{\frac{j}{2}}(t_{k+d+1}^j-t_{k+1}^j)[t_{k+1}^j,\ldots,t_{k+d+1}^j](t-x)_+^{d-1},
	\end{equation}
which is the scaled B-Splines \cite{Primbs:10}.
Then the sequence $S_j={\rm span}\{\Phi_j\}$ forms a multiresolution analysis (MRA)  of $L_2(I)$, where $I=(0,1)$ . The system $\Phi_j$ is uniformly local and locally finite, i.e., ${\rm diam} ({\rm supp}\phi_{j,k})\stackrel{<}{\sim}2^{-j}$ and $ \#\{\phi_{j,k}: {\rm supp}\phi_{j,k}\cap {\rm supp}\phi_{j,i}\}\stackrel{<}{\sim}1$; it forms a stable Riesz basis of $S_j$, i.e.,
	\begin{equation}
 c_{\Phi}\|{\boldsymbol{c}_j}\|_{l_2(\triangle_j)}\le\Big\|\sum_{k\in\Delta_j}c_{j,k}\phi_{j,k}\Big\|_{L_2(\Omega)}
    \le C_{\Phi}\|{\boldsymbol{c}_j}\|_{l_2(\triangle_j)}; \label{eq:uniform stable}
    \end{equation}
and $S_j$ satisfies the Jackson and Bernstein estimates, i.e.,
	\begin{eqnarray} %
	  & &\inf_{v_j\in S_j}\left\|v-v_j\right\|_{L_2(\Omega)}\stackrel{<}{\sim}2^{-jd}\left\|v\right\|_{\mathcal{H}^d(I)} \quad \forall v\in \mathcal{H}_0^d(I),\\
	  & &\left\|v_j\right\|_{\mathcal{H}^s(\Omega)}\stackrel{<}{\sim}2^{js}\left\|v_j\right\|_{L_2(I)} \quad \forall  v_j\in S_j,\ 0\le s\le \gamma,\label{eq:Jackson}
	\end{eqnarray}
    where $\gamma:=\sup\{\nu\in \mathbb{R}:v_j\in \mathcal{H}^\nu(I)~~ \forall v_j\in S_j\}$ and by $A\stackrel{<}{\sim}B$ we mean that $A$ can be bounded by a multiple of $B$, independent of the parameters they may depend on.

  Since $\Phi_j$ is a Riesz basis of $S_j$, there exists a dual MRA sequence $\tilde{S}_j={\rm span}\{\widetilde{\Phi}_j\}$, which also forms a MRA of $L_2(\Omega)$. And one can define the biorthogonal projector:
 \begin{equation}
   P_j: L_2(\Omega)\rightarrow S_j, \qquad P_jv:=\sum_{k\in\triangle_{j}}\left(v,\tilde{\phi}_{j,k}\right)\phi_{j,k},
   \end{equation}
where $(\cdot,\cdot)$ denotes the $L^2$ inner product.
Then $P_{j+1}P_j=P_jP_{j+1}=P_j$, and for $0\le s\leq 1$,
\begin{equation}\label{project}
  \left\|v-P_jv\right\|_{\mathcal{H}^s(\Omega)}\stackrel{<}{\sim}2^{j(s-\gamma)}\left\|v\right\|_{\mathcal{H}^\gamma(\Omega)},\qquad 0\le s<\gamma < d.
\end{equation}
One can also construct the interval biorthogonal wavelet sets $\Psi_j=\{\psi_{j,k},k\in\nabla_j\}$ and $\tilde{\Psi}_j=\{\tilde{\psi}_{j,k},k\in\nabla_j\}$; it holds the
norm equivalence, i.e., there exist $\tilde{\sigma}, \sigma >0$ such that for the Sobolev space $ \mathcal{H}^s(\Omega)$,
\begin{equation} \label{normequiv}
	\quad\Big\|\sum_{j\ge J_0-1}\sum_{k\in\nabla_j}d_{j,k}\psi_{j,k}\Big\|_{\mathcal{H}^s(\Omega)}^2\sim
	                                  \sum_{j\ge J_0-1}\sum_{k\in \nabla_{j}}2^{2js}\big|d_{j,k}\big|^2 \qquad \forall \ s\in(-\tilde{\sigma},\sigma),
\end{equation}
where $\psi_{J_0-1,k}:=\phi_{J_0,k}, \nabla_{J_0-1}:=\triangle_{J_0}, d_{J_0-1,k}:=c_{J_0,k}$; $J_0$ denotes the lowest level.
It also means that $\bigcup_{j=J_0-1}^\infty2^{-js}\Psi_j$ is a Riesz basis of $\mathcal{H}_0^s(\Omega)$.

Moreover, denote $W_j={\rm span}\{\Psi_j\}$. Then the operator $Q_j:=P_{j+1}-P_j$ is a projection onto the space $W_j$, having the representation
   \begin{equation}
   Q_jv=\sum_{k\in\nabla_j}\left(f,\tilde{\psi}_{j,k}\right)\psi_{j,k}.
   \end{equation}
And for $0<\gamma<d$, there exists
   \begin{equation}\label{compression}
      \left|\left(f,\tilde{\psi}_{j,k}\right)\right|\stackrel{<}{\sim}\inf_{p\in P_{d-1}}\left\|f-p\right\|_{L_2({\rm supp}  \tilde{\psi}_{j,k})}\left\|\tilde{\psi}_{j,k}\right\|\stackrel{<}{\sim}2^{-j{\gamma}}\left\|f^{(\gamma)}\right\|_{L_2({\rm supp} \tilde{\psi}_{j,k})}.
   \end{equation}
   This shows that the wavelet coefficients are small provided that the function is locally smooth, which is the   foundation to design  wavelet adaptive algorithms.

   Note that based on the scaling function sets $\Phi_j$, some other special wavelets can also be constructed.  If one demands $W_j=\tilde{W}_j$ and $ S_j=\tilde{S}_j$, the semiorthogonal  wavelets  $\Psi_j$ can be obtained \cite{Chui:92};
  if one chooses $d=2$, $\psi(x)=\frac{1}{\sqrt{2}}\phi_{1,1}(x), \Psi_j=\{\psi_{j,k}=2^{\frac{j}{2}}\psi(2^jx-k),k\in\nabla_j\}$,
  then the interpolation wavelet $\{\Psi_j\}_{j\ge-1}$ is obtained, and it also satisfies the norm equivalence for $ s\in (1,\frac{3}{2})$ (Page 605 of \cite{Cohen:00});
  and if $d=4$, let $\Phi_j=\{\phi_{j,k},k\in\triangle_j/\{1,2^j+1\}\}$, and define $\phi(x), \phi_b(x),\psi(x)$ and $\psi_b(x)$ by
	\begin{eqnarray}
	\phi(x)&=&\frac{1}{6}\sum_{i=0}^4{4 \choose i}(-1)^i(x-i)_+^3, \label{Cubic:1}\\
	\phi_b(x)&=&\frac{3}{2}x_+^2-\frac{11}{12}x_+^3+\frac{3}{2}(x-1)_+^3-\frac{3}{4}(x-2)_+^3+\frac{(x-3)^3_+}{6},\label{Cubic:2}\\
	\psi(x)&=&-\frac{1}{4}\phi(2x)+\phi(2x-1)-\frac{1}{4}\phi(2x-2),\\
  \psi_b(x)&=&\phi_b(2x)-\frac{1}{4}\phi(2x).\label{Cubic:4}
	\end{eqnarray}
One has the semi-interpolation spline wavelet $\Psi_j=\{\psi_b(2^j x), \psi_{j,k}\left|_{0\le k\le (2^j-3)},\right. \psi_b(2^j(1-x))\}$,  which satisfies the so-called {\em point value vanishing property}; and the wavelet expansion coefficients also indicate  the regularity of  the approximation function \cite{Cai:96}.  In practice, the base functions $\phi_{j,1}(\cdot)$ and $\phi_{j,2^j+1}(\cdot)$ are usually added for removing the limitation that the first order derivative of the (to be approximated) function at the boundary needs to be zero.
	
 Finally,  we point out that there  exists the refinement relations
	\begin{equation}\label{Refine-raltion}
	\Phi_j^T=\Phi_{j+1}^TM_{j,0}, \qquad \Psi_j^T=\Phi_{j+1}^TM_{j,1}.
	\end{equation}
And the space $S_J$ can be written as $S_J=S_{J_0}\oplus W_{J_0}\oplus \cdots \oplus W_{J-1}$. For any $u_J\in S_J$, it follows that
	\begin{equation}\label{eq:2.15}
	u_J=\sum_{k\in \triangle_J}c_{J,k}\phi_{J,k}=\sum_{k\in\triangle_{J_0}}c_{J_0,k}\phi_{J_0,k}+\sum_{j=J_0}^{J-1}\sum_{k\in\nabla_j}d_{j,k}\psi_{j,k}.
	\end{equation}
 Denote ${\boldsymbol {\mathrm c}}_j=(c_{j,k})_{k\in \triangle_j}$, $ {\boldsymbol{\mathrm d}_j}=(d_{j,k})_{k\in\nabla_j}$
 and $ {\boldsymbol d}_J=({\boldsymbol{\mathrm c}}_{J_0},{\boldsymbol{\mathrm d}}_{J_0},\ldots,{\boldsymbol{\mathrm d}}_{J-1})$.
 Then there exists a fast wavelet transform (FWT) between the single-scale and the multiscale representations, i.e.,
\begin{equation}\label{eq:2.16}
{\boldsymbol c}_J=M{\boldsymbol d}_J,
\end{equation}
which can be performed with the cost $\mathcal{O}(2^J)$; here
	\begin{equation}\label{eq:2.21}
	M=\left(\begin{array}{cc}M_{J-1}&0\\ 0&I_{J-1}\end{array}\right)\left(\begin{array}{cc}M_{J-2}&0\\0&I_{J-2}\end{array}\right)\cdots
	                                                                 \left(\begin{array}{cc}M_{J_0}&0\\0&I_{J_0}\end{array}\right)
	\end{equation}
and $M_j=(M_{j,0},M_{j,1})$.

\section {Uniform Schemes}
The nonlocal property of fractional operator makes the matrix of its discretizations inevitably dense. We will show that the chosen bases being the dilation and translation of one single function render the matrix to have a special structure, which greatly reduces the cost of computing and storing the entires. In this sense, these kind of bases are superior to the other possible bases, such as the usually used finite element or spectral polynomial bases. Meanwhile, based on the benefits of these bases, a simple diagonal preconditioner and the fast transform  are presented to enhance the effectiveness of the widely used nonlinear or linear iterative schemes.

We first consider the BVP of (\ref{eq:1.1}) with $q=0$. It has the variational formulation:
  Find $u\in\mathcal{H}_0^{\alpha}(\Omega)$ with $\alpha=1-\beta/2( 0\le \beta<1)$, such that
	\begin{equation}\label{eq:2.2.1}
	a(u,v)=(f,v) \qquad \forall v\in \mathcal{H}_0^{\alpha}(\Omega).
	\end{equation}
 More precisely, using integration by parts and the adjoint property of fractional integral operator \cite{Deng:08} leads to
    \begin{eqnarray}\label{elliptic}
     \qquad\, a(u,v)&=&\left \langle -\kappa_{\beta}D(p\ {}_0D_x^{-\beta} + (1-p)\ {}_xD_1^{-\beta})Du,\ v \right \rangle\\ \nonumber
		   &=&\kappa_{\beta} \left \langle p\ {}_0D_x^{-\beta}Du+(1-p){}_xD_1^{-\beta}Du,\ Dv \right \rangle\\		\nonumber			
           &=&\kappa_{\beta}\left( p\ {}_0D_x^{-\beta/2}Du,{}_xD_1^{-\beta/2}Dv\right)+\kappa_{\beta}\left((1-p){}_xD_1^{-\beta/2}Du, {}_0D_x^{-\beta/2}Dv \right).
  \end{eqnarray}

The bilinear form  $a(\cdot,\cdot): \, \mathcal{H}_0^{\alpha}(\Omega)\times \mathcal{H}_0^{\alpha}(\Omega)\rightarrow \mathbb{R}$  is continuous and coercive \cite{Ervin:07}, i.e.,
	\begin{equation}\label{coervcice}
	   |a(u,v)|\stackrel{<}{\sim }\|u\|_{\alpha}\|v\|_{\alpha}, \qquad a(u,u)\stackrel{>}{\sim}\|u\|_{\alpha}^2.
	\end{equation}
For $f\in L_2(\Omega)$, Eq. (\ref{eq:2.2.1}) admits a unique solution.
Letting $S_J$ be a subspace of $\mathcal{H}_0^{\alpha}(\Omega)$ with order $d$, the Galerkin approximation $u_J$ belonging to $S_J$ satisfies
	\begin{equation}\label{eq:2.2.2}
	a(u_J, v_J)=(f,v_J) \qquad  \forall \ v_J\in S_J.
	\end{equation}
If $u$ is sufficiently smooth, following (\ref{project}) and the Ce$\acute{a}$'s lemma, one gets
	\begin{equation}
	\|u-u_J\|_{\alpha}\stackrel{<}{\sim}\inf_{v_J\in S_J}\|u-v_J\|_{\alpha}\stackrel{<}{\sim}2^{J(\alpha-d)}\|u\|_{\mathcal{H}^d(\Omega)}.
	\end{equation}

For space discretization, one can either use the scaling basis $\Phi_J $ or the multiscale basis $\Psi^J=\left\{\Psi_j\right\}_{j=J_0-1}^{J-1}$, 
generating the following linear systems, respectively,
\begin{eqnarray}
 A_J {\boldsymbol c}_J&=&F_J, \\
 \hat{A}_J  {\boldsymbol d}_J&=&\hat{F}_J, \label{eq:3.2}
\end{eqnarray}
where $A_J=a(\Phi_J,\Phi_J),\, F_J=(f,\Phi_J),\, \hat{A}_J=a(\Psi^J,\Psi^J),\, \hat{F}_J=(f,\Psi^J)$. There are the following Lemmas.

\begin{lemma}\label{shift:1}
Let $\phi(x)\in \mathcal{H}_0^{\alpha}(\Omega)$,  {\rm supp}${\phi(x)}=[0,d]$\,and $\phi_{J,k}(x):=2^{J/2}\phi(2^Jx-k), 0\le k\le 2^J-d,\,k\in \mathbb{N}$. Then $a\left(\phi_{J,k_1},\phi_{J,k_2 }\right)=a\left(\phi_{J,k_1^\prime},\phi_{J,k_2^\prime }\right)$ if and only if $k_2-k_1=k_2^\prime-k_1^\prime$.
 \end{lemma}
 \begin{proof}
For $\phi(x)\in \mathcal{H}_0^{\alpha}(\Omega)$, there holds
 \begin{eqnarray*}
 &{}&\left({}_0D_x^{-\beta/2}D\phi_{J,k_1}, {}_xD_1^{-\beta/2}D\phi_{J,k_2}\right) \\
  &{}& =\frac{1}{\Gamma(\beta)}\int_0^1\int_0^x \left(x-s\right)^{\beta-1} \phi_{J,k_1}^{\prime}(s)\,\mathrm{d}s\, \phi^{\prime}_{J,k_2}(x)\,\mathrm{d}x\\
 &{}&=\frac{2^{3J}}{\Gamma(\beta)}\int_{2^{-J}k_2}^{2^{-J}(d+k_2)}\int_0^x \left(x-s\right)^{\beta-1} \phi^{\prime}\left(2^Js-k_1\right)\,\mathrm{d}s\, \phi^{\prime}\left(2^Jx-k_2\right)\,\mathrm{d}x\\
 &{}&=\frac{2^{2J}}{\Gamma(\beta)}\int_{0}^{d}\int_0^{2^{-J}(x+k_2)} \left(2^{-J}\left(k_2+x\right)-s\right)^{\beta-1} \phi^{\prime}\left(2^Js-k_1\right)\,\mathrm{d}s\, \phi^{\prime}(x)\,\mathrm{d}x\\
 &{}&=\frac{2^{2J\alpha}}{\Gamma(\beta)}\int_{0}^{d}\int_{-k_1}^{x+k_2-k_1}(k_2+x-s-k_1)^{\beta-1} \phi^{\prime}\left(s\right)\,\mathrm{d}s\, \phi^{\prime}(x)\,\mathrm{d}x\\
 &{}&=\frac{2^{2J\alpha}}{\Gamma(\beta)}\int_{0}^{d}\int_{0}^{x+k_2-k_1}(x-s+k_2-k_1)^{\beta-1} \phi^{\prime}\left(s\right)\,\mathrm{d}s\, \phi^{\prime}(x)\,\mathrm{d}x,
 \end{eqnarray*}
 which just depends on the value of $k_2-k_1$. The second part of (\ref{elliptic}) can be expressed by its first part, i.e.,
 \begin{equation}
  \left({}_xD_1^{-\beta/2}D\phi_{j,k_2}, {}_0D_x^{-\beta/2}D\phi_{j,k_1} \right)=\left({}_0D_x^{-\beta/2}D\phi_{j,k_1}, {}_xD_1^{-\beta/2}D\phi_{j,k_2}\right).
 \end{equation}
Then the desired result is obtained.
\end{proof}

 \begin{lemma}\label{shift:2}
 Let $\phi(x)$ and $\phi_{J,k}(x)$ be given as above, and $\phi(d/2-x)=\phi(d/2+x)$. Define $\theta_{J,i}(x):=2^{J/2}\theta_i(2^Jx)$ and $\tilde{\theta}_{J,i}(x):=2^{J/2}\theta_i(2^J(1-x))$ with $\theta_i(x)\in \mathcal{H}_0^1(\Omega)$ and {\rm supp}${\theta_i(x)}=[0,d_i]$, where $0<d_i<d$ and $i=1,2$. Then
 \begin{equation}
 \left({}_0D_x^{-\beta/2}D\theta_{J,i},\, {}_{x}D_1^{-\beta/2}D\phi_{J,k}\right)=\left({}_0D_x^{-\beta/2}D\phi_{J,2^J-d-k},\, {}_{x}D_1^{-\beta/2}D\tilde{\theta}_{J,i}\right).
 \end{equation}
 \end{lemma}
 \begin{proof} Similar to Lemma \ref{shift:1}, it follows that
 \begin{eqnarray*}
 && \left({}_0D_x^{-\beta/2}D\theta_{J,i},\, {}_{x}D_1^{-\beta/2}D\phi_{J,k}\right) \\
 && =\frac{1}{\Gamma(\beta)}\int_0^1\int_0^x \left(x-s\right)^{\beta-1} \theta_{J,i}^{\prime}(s)\,\mathrm{d}s\, \phi^{\prime}_{J,k}(x)\,\mathrm{d}x \\
 && =\frac{2^{2J\alpha}}{\Gamma(\beta)}\int_{0}^{d}\int_{0}^{x+k}(x+k-s)^{\beta-1} \theta_i^{\prime}\left(s\right)\,\mathrm{d}s\, \phi^{\prime}(x)\,\mathrm{d}x.
 \end{eqnarray*}
By the properties of symmetry and compact support, there exists
 \begin{eqnarray*}
 &&\left({}_0D_x^{-\beta/2}D\phi_{J,2^J-d-k},\, {}_{x}D_1^{-\beta/2}D\tilde{\theta}_{J,i}\right)\\
          &&=\frac{1}{\Gamma(\beta)}\int_0^1\int_0^x \left(x-s\right)^{\beta-1} \phi_{J,2^J-d-k}^{\prime}(s)\,\mathrm{d}s\,\tilde{\theta}^{\prime}_{J,i}(x)\,\mathrm{d}x\\
       && =\frac{2^{2J}}{\Gamma(\beta)}\int_{d_i}^{0}\int_0^{1-2^{-J}x} \left(1-s-2^{-J}x\right)^{\beta-1} \phi^{\prime}\left(2^Js-2^J+d+k\right)\,\mathrm{d}s\, \theta_i^{\prime}(x)\,\mathrm{d}x\\
       && =\frac{2^{2J\alpha}}{\Gamma(\beta)}\int_{0}^{d_i}\int_{\max{\{0,x-k}\}}^{\min{\{2^J-k,d\}}} \left(s+k-x\right)^{\beta-1} \phi^{\prime}\left(s\right)\,\mathrm{d}s\, \theta_i^{\prime}(x)\,\mathrm{d}x\\
       && =\frac{2^{2J\alpha}}{\Gamma(\beta)}\int_{0}^{d}\int_{0}^{x+k}(x+k-s)^{\beta-1} \theta_i^{\prime}\left(s\right)\,\mathrm{d}s\, \phi^{\prime}(x)\,\mathrm{d}x,
 \end{eqnarray*}
 where the Fubini-Tonelli theorem and $\min{\left\{2^J-k,d\right\}}=d$ are used.
\end{proof}

It is also easy to check that for any $ i_1,i_2\in \{1,2\}$,
   \begin{eqnarray}
   &&\left({}_0D_x^{-\beta/2}D\theta_{J,i_1},\, {}_{x}D_1^{-\beta/2}D\theta_{J,i_2}\right)=\left({}_0D_x^{-\beta/2}D\tilde{\theta}_{J,i_2},\, {}_{x}D_1^{-\beta/2}D\tilde{\theta}_{J,i_1}\right),\\
   && \left({}_0D_x^{-\beta/2}D\phi_{J,k},\, {}_{x}D_1^{-\beta/2}D\theta_{J,i}\right)=\left({}_0D_x^{-\beta/2}D\tilde{\theta}_{J,i},\, {}_{x}D_1^{-\beta/2}D\phi_{J,2^J-d-k}\right).
   \end{eqnarray}
Now, from the structure of $\Phi_J$ \cite{Primbs:10} and the above lemmas, one knows that the matrix $A_l:=\left({}_0D_x^{-\beta/2}D\Phi_J, {}_xD_1^{-\beta/2}D\Phi_J\right)$ has a quasi-Toeplitz structure, that is, it is a Toeplitz matrix after removing very few rows and columns near the boundaries.
More precisely, for $d=2$, it is a full Toeplitz matrix, but for $d=3$ and $d=4$, they have the following structures, respectively,
\begin{eqnarray}
&& \left(\begin{array}{ccc}a_1&r(\boldsymbol {\mathrm{a_2}})^T&0\\ \boldsymbol {\mathrm{a_1}}&H_{(2^J-2)\times (2^J-2)}&\boldsymbol {\mathrm{a_2}}\\a_2&r(\boldsymbol {\mathrm{a_1}})^T&a_1\end{array}\right)_{2^J\times 2^J},\\
&& \left(\begin{array}{ccccc}a_1&a_2&r(\boldsymbol {\mathrm{a_1}})^T&0&0\\a_3&a_4& r(\boldsymbol{\mathrm{a_2}})^T&0&0\\\boldsymbol {\mathrm{a_3}}&\boldsymbol {\mathrm{a_4}}&H_{(2^J-3)\times(2^J-3)}&\boldsymbol{\mathrm{a_2}}&\boldsymbol{\mathrm{a_1}}\\a_5&a_6&r(\boldsymbol {\mathrm{a_4}})^T&a_4&a_2\\a_7&a_5&r(\boldsymbol {\mathrm{a_3}})^T&a_3&a_1\end{array}\right)_{(2^J+1)\times(2^J+1)}, \label{struction2}
\end{eqnarray}
where $a_i$ are real numbers; $\boldsymbol{\mathrm{a_i}}$ are vectors, $r(\boldsymbol{\mathrm{a_i}})$  the reverse order of $\boldsymbol{\mathrm{a_i}}$; and $H_{N\times N}$ is Toeplitz matrix.

The fact that the bases are obtained by dilating and translating of a single function and the symmetry of the bases are essential for obtaining the above results, recalling that they do not hold for the general finite element (except linear element) and spectral methods. For the high order finite difference methods, the similar results can be got after modifying the approximation near the boundary for recovering the desired accuracy \cite{Chen:15,Zhao:15}, but it seems that the general theoretical results (stability, convergence and so on) are hard to obtain.
Further results for the generated matrix are
\begin{eqnarray}
 && A_r:=\left({}_xD_1^{-\beta/2}D\Phi_J, {}_0D_x^{-\beta/2}D\Phi_J \right)=\left({}_0D_x^{-\beta/2}D\Phi_J, {}_xD_1^{-\beta/2}D\Phi_J\right)^T=A_l^T,\label{eq:3.14e}\\
&&\qquad~ \left({}_0D_x^{-\beta/2}D\phi_{J,k_1}, {}_xD_1^{-\beta/2}D\phi_{J,k_2}\right)=0 \quad \forall k_2-k_1\leq -d.\label{eq:3.15e}
\end{eqnarray}

\begin{remark}
The structure of $\Phi_J$  also allows one to compute  its Riemann-Liouville fractional derivative easily, which can greatly reduce the computational complexity of generating the differential matrix in Galerkin and collocation methods (see Section 5). As an example, we present the techniques for $d=4$ in Appendix, being similar for other values of $d$. Then combining with (\ref{struction2}) and (\ref{eq:3.15e}), the left differential matrix $A_l=\left({}_0D_x^{1-\beta}\Phi_J, D\Phi_J\right)$ can be calculated exactly or numerically with the cost $\mathcal{O}(N)$,  being superior to the traditional finite element and spectral approximation with the cost $\mathcal{O}(N^2)$ \cite{Ervin:07,Li:10}.
\end{remark}

\subsection{Multilevel Preconditioning}
For an algebraic system with dense matrix, a well convergent iterative method generally has the computational cost $\mathcal{O}(N\log(N))$ or $\mathcal{O}(N^2)$, which is much less than the cost $\mathcal{O}(N^3)$ of the direct method. Moreover, a well conditional  number and `bunching of eigenvalues' usually bring good numerical stability and fast convergence speed \cite{Benzi:02,Campos:95}. In general, for a linear system $Ax=b$, a satisfactory preconditioned system $Bx'=b'$ should have the property
  \[
  \|B\|\leq C, \quad \|B^{-1}\|\leq C, \quad \mbox{$C$ is a moderate-sized constant independent of $N$};
  \]
and the computational cost for the preconditioning step is cheap. Here, both the matrix $A_J$ and $\hat{A}_J$ are dense, and their condition numbers are of order $\mathcal{O}(2^{2J\alpha})$; see Table \ref{tab:2_1} for Example 5.2. But with the aid of   wavelet bases, by the norm equivalence, a simple diagonal scaling  can lead to a good preconditioned system.
In fact, define
\begin{equation}
\nabla^J=\triangle_{J_0}\cup\nabla_{J_0}\cup\cdots \cup\nabla_{J-1},
\end{equation}
\begin{equation}
~~~~K={\rm diag}\Big(\underbrace{2^{-J_0\alpha},\ldots,2^{-J_0\alpha}}_{\#\triangle_{J_0}},\underbrace{2^{-J_0\alpha},\ldots,2^{-J_0\alpha}}_{\#\nabla_{J_0}},
        \ldots,\underbrace{2^{-(J-1)\alpha},\ldots,2^{-(J-1)\alpha}}_{\#\nabla_{J-1}}\Big).
\end{equation}
Combining the ellipticity (\ref{coervcice}), norm equivalence (\ref{normequiv}), and the Riesz representation theorem, one gets that for all ${\boldsymbol x}\in l_2(\nabla^J)$,
\begin{eqnarray}
&&\|K\hat{A}_JK{\boldsymbol x}\|_{l_2(\nabla^J)}=\sup_{{\boldsymbol y}\in l_2(\nabla^J)}\frac{\langle K\hat{A}_JK{\boldsymbol x},\, {\boldsymbol y}\rangle_{l_2(\nabla^J)}}{\|{\boldsymbol y}\|_{l_2(\nabla^J)}}\nonumber \\
                   &&=\sup_{{\boldsymbol y}\in l_2(\nabla^J)}\frac{a({\boldsymbol x}^TK\Phi^J,\, {\boldsymbol y}^TK\Phi^J)}{\|{\boldsymbol y}\|_{l_2(\nabla^J)}}
					\stackrel{<}{\sim}\frac{\|{\boldsymbol x}^TK\Phi^J\|_{\alpha}\|{\boldsymbol y}^TK\Phi^J\|_{\alpha}}{\|{\boldsymbol
                   y}\|_{l_2(\nabla^J)}}\stackrel{<}{\sim}\|{\boldsymbol x}\|_{l_2(\nabla^J)},\nonumber
\end{eqnarray}
\begin{equation}
 \|K\hat{A}_JK{\boldsymbol x}\|_{l_2(\nabla^J)}\stackrel{>}{\sim}\frac{\|{\boldsymbol x}^TK\Phi^J\|_{\alpha}\|{\boldsymbol x}^TK\Phi^J\|_{\alpha}}{\|{\boldsymbol x}\|_{l_2(\nabla^J)}}
														\stackrel{>}{\sim}\|{\boldsymbol x}\|_{l_2(\nabla^J)}.
\end{equation}
Therefore, there exist $C_1,C_2$ not depending on $J$ such that
\begin{equation} \label{cond1}
 \qquad \quad C_1\|{\boldsymbol x}\|_{l_2(\nabla^J)}\le\|K\hat{A}_JK{\boldsymbol x}\|_{l_2(\nabla^J)}\le C_2 \|{\boldsymbol x}\|_{l_2(\nabla^J)}.
\end{equation}
Now, one arrives at
\begin{eqnarray}
 &&\|K\hat{A}_JK\|\stackrel{<}{\sim}C_2,\ \, \|(K\hat{A}_JK)^{-1}\|\stackrel{<}{\sim}(1/C_1),\\
 && {\rm cond}_2(K\hat{A}_JKu_J)=\|K\hat{A}_JK\|\|(K\hat{A}_JK)^{-1}\|\stackrel{<}{\sim} ({C_2}/{C_1}).\label{cond2}
\end{eqnarray}

The norm equivalence implies $ a(\psi_{j,k},\psi_{j,k})\sim 2^{2js}$, so one can also define matrix $K$ by the inverse  square root of the diagonal of $\hat{A}_J$, and (\ref{cond1}) and (\ref{cond2}) still hold. Usually the current $K$ performs better since it uses the information directly from the stiffness matrix, and we will use it in Section 5. Moreover, the cost of generating $K$ is only $\mathcal{O}(J)$; this is because that by using the translation property of the inner wavelet on the same level, one just needs to calculate the entries $a(\psi_{j,k},\psi_{j,k})$ near the boundaries and one in the inner part without the necessity to assemble $\hat{A}_J$.

Now, one can rewrite (\ref{eq:3.2}) as the two-sided preconditioned form	
\begin{equation}\label{eq:3.9}
 \underbrace{K\hat{A}_JK}K^{-1}{\boldsymbol d}_J=K\hat{F}_J.
\end{equation}
Further using (\ref{eq:2.16}), one gets that
\begin{equation}\label{eq:3.10}
 \underbrace{KM^TA_JMK}K^{-1}M^{-1}{\boldsymbol c}_J=KM^T F.
\end{equation}
 A straightforward product of $A_J$ or $\hat{A}_J$ to a given vector needs a computational cost $\mathcal{O}(2^{2J})$. But if one uses the quasi-Toeplitz structure of the matrix, the computational cost can be reduced to $\mathcal{O}(J2^J)$.
 In fact,  one can rewrite $A_J$ as
\begin{equation}
 A_J={\rm diag}(K_1)A_l+{\rm diag}(K_2)A_r,
\end{equation}
where $K_1$ and $K_2$ denote the coefficient vectors, formed by the coefficient of space fractional derivative taking values at the discretized intervals, and $A_l$ and $A_r$ are  quasi-Toeplitz matrices.  Using the FFT to the matrix-vector product makes the computational cost $\mathcal{O}(J2^J)$ \cite{Wang:10}.
 Finally, because the FWT (having the matrix representation  $M$ or $M^T$, which denotes the primal reconstruction or the dual decomposition \cite{Urban:09}) can be implemented with the cost $\mathcal{O}(2^J)$, if the CG scheme (symmetric) is applied to (\ref{eq:3.10}) or to the corresponding normal equation (asymmetric), the well conditioned number of the matrix implies that the convergence rate is independent of the level $J$; then we can solve it with the total operations $\mathcal{O}(J2^J)$. For the general iterative schemes, such as GMRES or Bi-CGSTAB, usually one can show that the system with clustered spectrum and well conditioned number after preconditioning has an accelerated convergence. What's more, compared with the most existing preconditioners which require the solving of a linear system (see, e.g., the ILU \cite{Lin:14} and the Strang \cite{Lei:13}), the wavelet preconditioning operation reduces to the matrix-vector product, where FWT can be used.

\subsection{Multiresolution Multigrid Method}

The multigrid method based on the finite difference discretization for solving the fractional IBVPs have been developed in \cite{Chenming:15, Pang:12}, where the transition operators (restriction and prolongation operators) between the grids are chosen as the full weight and interpolation operators. In this Subsection, we investigate the MMG method for solving fractional IBVPs. We will show that the transition operators in the MRA background can be more straightforwardly defined. And using the techniques presented in the above content, the MMG scheme can also be fast implemented.

Denoting $S_j$ as the subspace, $\mathcal{A}_j: S_j\rightarrow S_j$ with $(\mathcal{A}_j\omega_j,v_j):=a(\omega_j,v_j)~ \forall v_j\in S_j$ and $\mathcal{Q}_j: L^2\rightarrow S_j$ with $(\mathcal{Q}_j\rho,v_j)=(\rho,v_j) ~ \forall v_j \in S_j$, one arrives at the semidiscrete form: Find $u_J(t)\in S_J, t\ge 0$ such that
\begin{equation}\label{semi:1}
 \left\{
 \begin{array}{l}
 \frac{\partial{u_J}}{\partial{t}}+\mathcal{ A_J} u_J=f_J(t):=\mathcal{ Q_J}f(t) \\
 u_J(0)=u_J^0\in S_J.
 \end{array} \right.
 \end{equation}
Taking the time mesh as $0\equiv t_0<t_1<\cdots<t_{N-1}<t_N\equiv T$ and the stepsizes $\Delta t_n=t_{n+1}-t_n,\, n=0,\ldots,N-1$, one gets the backward Euler multiresolution Galerkin method (B-MGM)
\begin{equation}
(U_J^{n+1},v_J)+\triangle t_n a(U^{n+1}_J,v_J)=(U^n_J+\triangle t_n f(t_{n+1}),v_J) \quad \forall v_J\in S_J;
\end{equation}
and the Crank-Nicolson multiresolution Galerkin method (CN-MGM)
\begin{equation}
(\overline{\partial}U^{n+1}_J,v_J)+a\left(\frac{U^{n+1}_J+U^n_J}{2},v_J\right)=(f(t_{n+1/2}),v_J)\quad \forall v_J\in S_J,
\end{equation}
where $\overline{\partial}U_J^{n+1}=(U_J^{n+1}-U_J^n)/\triangle t_n$.
Introduce the bilinear form $B_{n+1}(u,v):=(u,v)+\lambda\triangle t_n a(u,v)$, where $\lambda=1$ for the B-MGM and $\lambda=1/2$ for the CN-MGM. Define $\mathcal{B}_j^{n+1}: S_j\to S_j$ with $(\mathcal{B}_j^{n+1}\rho_j,v_j)=B_{n+1}(\rho_j,v_j)$ $\forall v_j\in S_j$, and the operator $P^{n+1}_j:\mathcal{H}^{\alpha}_0(\Omega)\to S_j$ with   $B_{n+1}(P^{n+1}_j\rho,v_j)=B_{n+1}(\rho,v_j)$ $\forall v_j\in S_j$. Then the MGM schemes can be rewritten uniformly as the form
\begin{equation}\label{MultiGrid}
\mathcal{B}_J^{n+1}U_J^{n+1}=g_J^{n+1},
\end{equation}
where $g_J^{n+1}:=U^n_J+\triangle t_nQ_Jf(t_{n+1})$ for the  B-MGM and $g_J^{n+1}:=-\frac{\triangle t_n}{2}\mathcal{A}_JU^n_J+U^n_J+\triangle t_nQ_Jf(t_{n+1/2})$ for the  CN-MGM, respectively. Suppose that $U_J=\sum_{k\in \triangle_J}c_{J,k}\phi_{J,k}\in S_J$, and define ${\boldsymbol c_J},\tilde{g}_J\in \mathbb{R}^{\#(\triangle_J)}$, $({\boldsymbol c_J})_k:=c_{J,k},\, (\tilde{g}_J)_k:=(g_J,\phi_{J,k}),\,k\in \triangle_J$. Denoting $B^{n+1}_J= \left( B_{n+1}(\phi_{J,k},\phi_{J,i}) \right)_{k,i\in \triangle_J}$, one also gets the algebraic representation of (\ref{MultiGrid}) given by $B_J^{n+1}{\boldsymbol c_J^{n+1}}=\tilde{g}_J^{n+1}$.

The basic iteration algorithm for the operator equation (\ref{MultiGrid}) is \cite{Xu:97}
\begin{equation}
U_J^{n+1,l+1}=U_J^{n+1,l}+R_J^{n+1}\left(g_J^{n+1}-\mathcal{B}_J^{n+1}U_J^{n+1,l}\right),\quad l=0,1,2,\dots
\end{equation}
with the error propagation operator $\mathcal{K}_J^{n+1}:=I-R_J^{n+1}\mathcal{B}_J^{n+1}$ and the iterator $R_J^{n+1}:S_J\to S_J$. For the damped Richardson and Jacobi methods, the iterators are, respectively, given by
\begin{eqnarray}
R_J^{n+1}g&=&\omega\sigma(B_J^{n+1})^{-1}\sum_{k\in \triangle_J}\left(g,\phi_{J,k}\right)\phi_{J,k}\quad \forall g\in S_J;\\
R_J^{n+1}g&=&\omega\sum_{k\in \triangle_J}B_{n+1}\left(\phi_{J,k},\phi_{J,k}\right)^{-1}\left(g,\phi_{J,k}\right)\phi_{J,k} \quad \forall g\in S_J;
\end{eqnarray}
 they can also be regarded as the correction with subspaces decomposition $S_j=\sum_{k \in \triangle_j}V_j^k$ with $V_j^k={\rm span}\{\phi_{j,k}\}$; and one can also rewrite the Jacobi iterator as
\begin{equation}
 R_j^{n+1}=\omega \sum_{k \in \triangle_j}P_j^{n+1,k}(\mathcal{B}_j^{n+1})^{-1},
\end{equation}
where $P_j^{n+1,k}:S_j\to V_j^k$ with $B_{n+1}(P_j^{n+1,k}v_j,\phi_{j,k})=B_{n+1}(v_j,\phi_{j,k})~ \forall v_j\in S_j$. Now, the MMG $V$-cycle algorithm for (\ref{MultiGrid}) reads
 \begin{equation}
 U_{J}^{n+1,l+1}=U_{J}^{n+1,l}+\mathcal{M}_J^{n+1}\left(g_J^{n+1}-\mathcal{B}_J^{n+1}U_{J}^{n+1,l}\right),\quad l=0,1,2,\dots
 \end{equation}
 and the multigrid iterator $\mathcal{M}_J^{n+1}$ is defined in Algorithm \ref{V-CYCLE} by induction, where $R_j^{n+1}$ and  $\mathcal{K}_{j}^{n+1}: S_j \to S_j$ are defined in the same way as $R_J^{n+1}$ and $\mathcal{K}_{J}^{n+1}$.
\begin{algorithm}[h t b p]
\caption{MMG V-CYCLE ITERATOR}
\label{V-CYCLE}
\begin{algorithmic}[1]
\STATE  Fix $t=t_{n+1}$; for $j=J_0$, define $\mathcal{M}_{J_0}^{n+1}=(\mathcal{B}_{J_0}^{n+1})^{-1}$. Assume that $\mathcal{M}_{j-1}^{n+1}:S_{j-1}\to S_{j-1}$ is defined. For $g\in S_j$, define the iterator $\mathcal{M}_{j}^{n+1}:S_{j}\to S_{j}$ through the following steps:
\renewcommand{\baselinestretch}{1.5}
\large\normalsize
\STATE $(1)$ Pre-smoothing: For $x^{n+1}_0=0\in S_j$ and $l=1,\dots,m_1(j)$,

  $\qquad \qquad \qquad\quad x^{n+1}_l=x^{n+1}_{l-1}+R_j^{n+1}\left(g-\mathcal{B}_j^{n+1}x_{l-1}^{n+1}\right)$

\STATE $(2)$ Coarse grid correction:

$\qquad\qquad \qquad \quad x_{m_1(j)+1}^{n+1}=x_{m_1(j)}^{n+1}+\mathcal{M}_{j-1}^{n+1}Q_{j-1}\left(g-\mathcal{B}_j^{n+1}x_{m_1(j)}^{n+1}\right)$

\STATE $(3)$ Post-smoothing: For $l=m_1(j)+2,\dots,m_1(j)+m_2(j)+1$,

$\qquad\qquad\qquad \quad x^{n+1}_l=x^{n+1}_{l-1}+R_j^{n+1}\left(g-\mathcal{B}_j^{n+1}x_{l-1}^{n+1}\right)$

\STATE Define $\mathcal{M}_j^{n+1}g=x^{n+1}_{m_1(j)+m_2(j)+1}$
\renewcommand{\baselinestretch}{1}
\large\normalsize
\end{algorithmic}
\end{algorithm}
Obviously, the  MMG error propagation operator satisfies
\begin{eqnarray}\label{MMG-Err}
~~~I-\mathcal{M}_{j+1}^{n+1}\mathcal{B}_{j+1}^{n+1}&=&\left(\mathcal{K}_{j+1}^{n+1}\right)^{m_2(j+1)}\left(I-\mathcal{M}_{j}^{n+1}Q_{j}\mathcal{B}_{j+1}^{n+1}\right)\left(\mathcal{K}_{j+1}^{n+1}\right)^{m_1(j+1)}\\\nonumber
&=&\underbrace{\left(\mathcal{K}_{j+1}^{n+1}\right)^{m_2(j+1)}\left(I-P_j^{n+1}\right)\left(\mathcal{K}_{j+1}^{n+1}\right)^{m_1(j+1)}}_{\mathbb{ I}}\\ \nonumber &&+\underbrace{\left(\mathcal{K}_{j+1}^{n+1}\right)^{m_2(j+1)}\left(I-\mathcal{M}_j^{n+1}\mathcal{B}_{j}^{n+1}\right)P_j^{n+1}\left(\mathcal{K}_{j+1}^{n+1}\right)^{m_1(j+1)}}_{\mathbb{ II}}, \nonumber
\end{eqnarray}
where the relation $Q_{j}\mathcal{B}_{j+1}^{n+1}=\mathcal{B}_{j}^{n+1}P_j^{n+1}$ has been used;  ${\mathbb{I}}$  just denotes the usual two-grid error propagation operator; and $m_1(j+1)$ (or $m_2(j+1)$) means that the iterative times $m_1$ (or $m_2$) may depend on the level $j+1$.

Recall that the refinement relation (\ref{Refine-raltion}), one can get the prolongation matrix $M_{j,0}$ straightforward, and it holds
\begin{equation}
\left\{\begin{array}{l}
{\boldsymbol c^{n+1}_{j+1}}=M_{j,0}{\boldsymbol c^{n+1}_j} \quad \forall U^{n+1}_{j+1}=U^{n+1}_j\in S_j\subset S_{j+1};\\[5pt]
 \widetilde{Q_jr^{n+1}_{j+1}}=M_{j,0}^T\widetilde{r^{n+1}_{j+1}} \quad \forall r^{n+1}_{j+1}\in S_{j+1},
 \end{array}
 \right.
\end{equation}
where the meaning of $\widetilde{r^{n+1}_{j+1}}$ is defined after Eq. (\ref{MultiGrid}).
This means that the  transpose of $M_{j,0}$ is just the restriction matrix. Noticing that $\mathcal{B}_j^{n+1}U^{n+1}_j=Q_j\mathcal{B}_{j+1}^{n+1}U^{n+1}_j~ \forall U^{n+1}_j\in S_j$, there holds
\begin{equation}
B_j^{n+1}{\boldsymbol c_j^{n+1}}=\widetilde{\mathcal{B}_j^{n+1}U^{n+1}_j}=M_{j,0}^T\widetilde{\mathcal{B}_{j+1}^{n+1}U^{n+1}_j}=M_{j,0}^TB_{j+1}^{n+1}M_{j,0}{\boldsymbol c_j^{n+1}},
\end{equation}
i.e.,
\begin{equation}
 B_j^{n+1}=M_{j,0}^TB_{j+1}^{n+1}M_{j,0},
 \end{equation}
which actually is the Galerkin identity, facilitating the convergence analysis, but this is not true for the difference method.
Note that the quasi-Topelitz structure of $B_j^{n+1}$ makes it feasible to be generated directly with the cost $\mathcal{O}(2^j)$. Using the fast algorithms (FFT and FWT), the matrix-vector product can preformed with the cost $\mathcal{O}(j2^j)$. So the total computational count per MMG step is $\mathcal{O}(J2^J)$ and the storage cost is $\mathcal{O}(2^J)$.

In the following, we present the convergence analysis of the MMG when $\bf{A}$ is a Riesz derivative and $m_1(j)=m_2(j)=m_0$; see Algorithm 1.  It is easy to check that $\mathcal{B}_j^{n+1}$ is symmetric, $P^{n+1,k}_{j}$ and $P_j^{n+1}$ are $A$-orthogonal projectors, and  $\mathcal{K}_j^{n+1}$ and $I-\mathcal{M}_j^{n+1}\mathcal{B}_j^{n+1}$ are $A$-selfadjoint; all of them are considered with respect to $B_{n+1}(\cdot,\cdot)$.
\begin{lemma} [see \cite{Chen:10}] \label{MMG-lemma}
Assume that $R_j^{n+1}:S_j\to S_j$ is symmetric with respect to $(\cdot,\cdot)$, positive semi-definite, and satisfies
\begin{equation}\label{MMG-cond}
\left\{
\begin{array}{l}
 B_{n+1}\left(\mathcal{K}_j^{n+1}v_j,v_j\right)\ge 0 \quad \forall v_j\in S_j, \\[5pt]
 \left(\left(R_j^{n+1}\right)^{-1}v_j,v_j\right)\le {\epsilon}B_{n+1}\left(v_j,v_j\right) \quad \forall v_j\in \left(I-P_{j-1}^{n+1}\right) S_j.
 \end{array} \right.
 \end{equation}
 Then we have
 \begin{equation}\label{MMG-coveg}
    0\le B_{n+1}\left(\left(I-\mathcal{M}_j^{n+1}\mathcal{B}_j^{n+1}\right)v_j,v_j\right)\le \delta B_{n+1}\left(v_j,v_j\right)\quad \forall v_j\in S_j,
 \end{equation}
 where $\delta=\epsilon/(\epsilon+2m_0)$.
\end{lemma}

Since $I-\mathcal{M}_J^{n+1}\mathcal{B}_J^{n+1}$ is $A$-selfadjoint, (\ref{MMG-coveg}) actually means that its spectral radius
\begin{eqnarray}
\sigma\left(I-\mathcal{M}_J^{n+1}\mathcal{B}_J^{n+1}\right)&=&\left\|I-\mathcal{M}_J^{n+1}\mathcal{B}_J^{n+1}\right\|_{A}\\\nonumber
&=&\sup_{0\ne v\in S_J}\frac{B_{n+1}\left(\left(I-\mathcal{M}_J^{n+1}\mathcal{B}_J^{n+1}\right)v,v\right)}{B_{n+1}\left(v,v\right)}\le \delta<1.
\end{eqnarray}
 When $\omega\in [c_0,1],\, 0<c_0\le 1$, the Richardson method obviously satisfies the requirements of Lemma \ref{MMG-lemma}. Since the damped Jacobi iteration converges under the condition $0<\omega<2/\sigma(R_j^{n+1}\mathcal{B}_j^{n+1})$, for any $v_j\in S_j$, there exists
\begin{eqnarray*}
   &\,&B_{n+1}\left(\mathcal{K}_j^{n+1}v_j,\mathcal{K}_j^{n+1}v_j\right)\\
   &\,&\ =B_{n+1}\left(v_j,v_j\right)-2\omega B_{n+1}\left(R_j^{n+1}\mathcal{B}_j^{n+1}v_j,v_j\right)+
   \omega^2B_{n+1}\left(R_j^{n+1}\mathcal{B}_j^{n+1}v_j,R_j^{n+1}\mathcal{B}_j^{n+1}v_j\right)\\
   &\,&\ =B_{n+1}\left(v_j,v_j\right)-2\omega\left((R_j^{n+1})^{\frac{1}{2}}\mathcal{B}_j^{n+1}v_j,(R_j^{n+1})^{\frac{1}{2}}\mathcal{B}_j^{n+1}v_j\right)\\
   &\,&\ \quad\qquad\qquad+\,\omega^2\left(\left[(R_j^{n+1})^{\frac{1}{2}}\mathcal{B}_j^{n+1}(R_j^{n+1})^{\frac{1}{2}}\right]
   (R_j^{n+1})^{\frac{1}{2}}\mathcal{B}_j^{n+1}v_j,(R_j^{n+1})^{\frac{1}{2}}\mathcal{B}_j^{n+1}v_j\right)\\
   &\,&\ \le B_{n+1}\left(v_j,v_j\right)-\omega\left(2-\omega\sigma(R_j^{n+1}\mathcal{B}_j^{n+1})\right)\left((R_j^{n+1})^{\frac{1}{2}}\mathcal{B}_j^{n+1}v_j,(R_j^{n+1})^{\frac{1}{2}}\mathcal{B}_j^{n+1}v_j\right).
\end{eqnarray*}
Then it is sufficient to take $0<\omega<1/\sigma(R_j^{n+1}\mathcal{B}_j^{n+1})$ for getting the first condition in (\ref{MMG-cond}).
 \begin{eqnarray*}
     \left((R_j^{n+1})^{-1}v_j,v_j\right)&=&\sum_{k\in \triangle_j}\left(\left(R_j^{n+1}\right)^{-1}v_j,c_{j,k}\phi_{j,k}\right)
                                                               =\sum_{k\in \triangle_j}B_{n+1}\left(\mathcal{E}_j^{n+1}v_j,c_{j,k}\phi_{j,k}\right)\\
     &\le&\sqrt{\sum_{k\in \triangle_j}B_{n+1}\left(\mathcal{E}_j^{n+1}v_j,\mathcal{E}_j^{n+1}v_j\right)}\sqrt{\sum_{k\in \triangle_j}B_{n+1}\left(c_{j,k}\phi_{j,k},c_{j,k}\phi_{j,k}\right)}\\
     &=&\sqrt{\sum_{k\in \triangle_j}\left(\mathcal{E}_j^{n+1}v_j,(R_j^{n+1})^{-1}v_j\right)}\sqrt{\sum_{k\in \triangle_j}B_{n+1}\left(c_{j,k}\phi_{j,k},c_{j,k}\phi_{j,k}\right)}\\
     &=&\sqrt{\frac{1}{\omega}\left((R_j^{n+1})^{-1}v_j,v_j\right)}\sqrt{\sum_{k\in \triangle_j}B_{n+1}\left(c_{j,k}\phi_{j,k},c_{j,k}\phi_{j,k}\right)},
 \end{eqnarray*}
where $ v_j=\sum\limits_{k\in\triangle_j}c_{j,k}\phi_{j,k}\in S_j$ and $\mathcal{E}_j^{n+1}:=P_j^{n+1,k}(\mathcal{B}_j^{n+1})^{-1}(R_j^{n+1})^{-1}$.

By Bernstein estimate and  uniform stability given in (\ref{eq:Jackson}) and (\ref{eq:uniform stable}), one gets
\begin{equation}
\qquad\sum_{k\in \triangle_j}B_{n+1}\left(c_{j,k}\phi_{j,k},c_{j,k}\phi_{j,k}\right)\stackrel{<}{\sim}2^{2j\alpha}\sum_{k\in \triangle_j}\|c_{j,k}\phi_{j,k}\|^2_{L_2(\Omega)}\sim 2^{2j\alpha}\|v_j\|^2_{L_2(\Omega)}.
\end{equation}
By the Aubin-Nitscale trick \cite{Ervin:05,Deng:08,Huang:14}, there exists
\begin{eqnarray}
\left\|(I-P_{j-1}^{n+1})v_j\right\|^2_{L_2(\Omega)}\stackrel{<}{\sim}2^{-2j\alpha}\left\|(I-P_{j-1}^{n+1})v_j\right\|^2_{\mathcal{H}^{\alpha}{(\Omega)}}
\sim 2^{-2j\alpha}B_{n+1}\left((I-P_{j-1}^{n+1})v_j,(I-P_{j-1}^{n+1})v_j\right).
\end{eqnarray}
Noting that $(I-P_{j-1}^{n+1})v_j\in S_j$, then the second requirement in (\ref{MMG-cond}) holds. The proof of the convergence for MMG with Jacobi iterator is completed.

\section{Multiscale Adaptive Schemes}

Although there are works  to discuss the low regularity (especially the weaker regularity at the area close to boundary) of the solutions for fractional PDEs, it seems few of them are for designing the adaptive algorithms, which have been well developed for classical PDEs. In  the classical finite elements approximation,  adaptivity is usually driven by so-called local a posteriori error estimate (an efficient and reliable error indicator consisting of local terms and being easy to compute).  In the following  we first show the challenge when using the traditional finite element method to adaptively solve the fractional BVPs, then  numerically demonstrate the performance of the wavelet adaptation algorithm.


 For the linear element and  $\beta\in (1/2,1)$, with a given positive integer $K$, we define the mesh
 $$ 0\equiv x_0<x_1<\cdots x_{K-1}<x_K\equiv 1,\qquad I_i=(x_{i-1},x_{i}).$$
Then the posteriori error for the BVP is bounded by
\begin{equation}\label{Posteriori}
 a(u-u_h,u-u_h)\stackrel{<}{\sim}\sum_{i=1}^{K}h_i^{2 \alpha}\left\|f+\kappa_{\beta}\left(p{}_0D_x^{2\alpha}u_h +(1-p){}_xD_1^{2\alpha}u_h\right) \right\|^2_{L_2(I_i)},
 \end{equation}
 where $h_i=x_i-x_{i-1}$, and $\beta\in (1/2,1)$ ensures that ${}_0D_x^{2\alpha}u_h$ and ${}_xD_1^{2\alpha}u_h$ belong to $L_2$.

Now we prove (\ref{Posteriori}). Denote by $\Pi_h$ the operator for the piecewise linear interpolation associated with $\{x_i\}$ and $e_h=u-u_h$. For any $v\in H^{\alpha}_0(\Omega)$, there exists
 \[a\left(e_h,v\right)=a\left(e_h,v-\Pi_hv\right)=\left(f,v-\Pi_hv\right)-a\left(u_h,v-\Pi_hv\right).\]
 Combining (\ref{elliptic}), $(v-\Pi_hv)(x_{i-1})=(v-\Pi_hv)(x_i)=0$, and the regularity of $u_h$ leads to
 \begin{eqnarray*}
  a\left(e_h,v\right)&=&\sum_{i=1}^K\int_{I_i}\left(f+\kappa_{\beta}\left(p\,{}_0D_x^{2\alpha}u_h +(1-p){}_xD_1^{2\alpha}u_h\right)\right)\left(v-\Pi_h v\right)\,\mathrm{d}x\\\nonumber
  &\le&\sum_{i=1}^K\left\|f+\kappa_{\beta}\left(p\,{}_0D_x^{2\alpha}u_h +(1-p){}_xD_1^{2\alpha}u_h\right)\right\|_{L_2(I_i)}\left\|v-\Pi_hv\right\|_{L_2(I_i)}\\\nonumber
  &\stackrel{<}{\sim}&\sum_{i=1}^Kh_i^{\alpha}\left\|f+\kappa_{\beta}\left(p\,{}_0D_x^{2\alpha}u_h +(1-p){}_xD_1^{2\alpha}u_h\right)\right\|_{L_2(I_i)}\left\|v\right\|_{\mathcal{H}^{\alpha}(I_i)}\\\nonumber
  &\stackrel{<}{\sim}&\sqrt{\sum_{i=1}^Kh_i^{2\alpha}\left\|f+\kappa_{\beta}\left(p\, _0D_x^{2\alpha}u_h +(1-p){}_xD_1^{2\alpha}u_h\right)\right\|^2_{L_2(I_i)}}\left\|v\right\|_{\mathcal{H}^{\alpha}(\Omega)}.
  \end{eqnarray*}
Then (\ref{Posteriori}) follows by taking $v=e_h$ in the above inequality and using $a(v,v){\sim} \|v\|_{\mathcal{H}^{\alpha}(\Omega)}$.
The term $\left\|\,\cdot\,\right\|^2_{L_2(I_i)}$ involves nonlocal calculations; so it can not be used directly as a local error indicator. The following in Algorithm \ref{AD-STATIC}, we will show that for the wavelet methods of fractional BVPs, the local regularity indicator can be  the wavelet coefficient when the (to be determined) solution is represented by the multiscale bases: small coefficient implies good local regularity while big one indicates the opposite.
\begin{algorithm}[!h t b p]
\caption{ADAPTIVE WAVELET SOLVER FOR THE BVP}
\label{AD-STATIC}
\begin{algorithmic}[1]
\STATE Given $\epsilon(j), It_{max}$, $J_0$
\STATE $m=0$
\STATE Solve the equation in space $V_{J_0+1}$ to get the initial approximation coefficients $(\bf{c_{J_0}^0, d_{J_0}^0})$
		   and the index $\Lambda^m=(J_0,\lambda), \lambda\in\nabla_{J_0}$
\REPEAT
\STATE Determine the significant index set $\Lambda$ by $\epsilon(j)$

\STATE Check the adjacent zone index set $\mathcal{N}_{l,\lambda}$  of each $(l,\lambda)\in\Lambda$;
       denote $\Lambda_{\mathcal{N}}=\cup_{(l,\lambda)\in\Lambda}\mathcal{N}_{l,\lambda}$ and establish
			 $\Lambda^{m+1}=\Lambda\cup\Lambda_{\mathcal{N}}$
			
 \FOR{$(l,\lambda)\in\Lambda^{m+1}$}
 \STATE  ${}^*d_{l,\lambda}^{m+1}=\left\{
                                \begin{array}{ll}
								d_{l,\lambda}^m, &\quad (l,\lambda)\in\Lambda^m\\
								0,               &\quad\mbox{otherwise}
								\end{array}\right. $
\ENDFOR
\STATE  $\bf{{}^*c_{J_0}^{m+1}=c_{J_0}^m}$
\STATE Solve the algebraic matrix equation, resulted from the discretization in the nonlinear approximation space $\hat{V}_{J_0+m+1}(\Omega) \subset V_{J_0+m+1}(\Omega)$, by appropriate iterative scheme with the initial guess $(\bf{{}^*c_{J_0}^{m+1},{}^*d_{J_0}^{m+1},\ldots,{}^*d_{J_0+m+1}^{m+1}})$
\STATE Determine ${\Lambda}=\{(l,\lambda): (l,\lambda)\in \Lambda^{m+1}, |d_{l,\lambda}^m|\ge\epsilon(l)\}$
\STATE $m=m+1$
\UNTIL{$m>It_{max}$ or $\Lambda=\Phi$ (empty set)}
\end{algorithmic}
\end{algorithm}
One of the main features of  Algorithm \ref{AD-STATIC}  is that the finest grid resolution can be automatically determined by the given tolerance $\epsilon(j)$. In order to gain a more robust and faster algebra solver, $\psi_{j,k}$ has been  scaled by the inverse square root of  $a\left(\psi_{j,k},\psi_{j,k}\right)$, and the multiscale  approximation of the solution at the current scale has been as an initial guess for the iteration in the finer scale obtained after adding the wavelets. For constructing the refined index set, thanks to the tree structure of wavelet singularity detection, we first include a coarsening step by thresholding the latest available wavelet coefficients to get a significant index set, then add all their children. If $j=l+1$ and $k\in \left\{2\lambda,\,2\lambda+1\right\}$, then the wavelet indexed by $\left(j, k\right)$ is called a child of the wavelet indexed by $(l,\lambda)$. One can further extend the index set by including the horizontal neighbors of the wavelet indices already included. Such an extended index set associated with the index  $\left(l,\lambda\right)$ is called an adjacent zone, which is denoted by $\mathcal{N}_{l,\lambda}$. In Algorithm \ref{AD-STATIC}, the index set is continuously updated to resolve the local structures that appear in the solution. One can dynamically adjust the number and locations of the wavelets used in the wavelet expansion, reducing significantly the cost of the scheme while providing enough resolution in the regions where the solution varies significantly.
The $m$-th approximation of the solution is given by
\begin{equation}
\widehat{u}_{J_0+m}=\sum_{(l,\lambda)\in{\triangle_{J_0}\cup\Lambda^m}} d_{l,\lambda}\hat{\psi}_{l,\lambda},
\end{equation}
where $\triangle_{J_0}\cup\Lambda^m$ is the irregular index set, and $\hat{\psi}_{l,\lambda}$ represents the normalization of $\psi_{l,\lambda}$.
 Finally, after getting the sufficiently accurate approximation, the corresponding single scaling representation can be got by the FWT.

 The following we  develop  adaptive algorithm for   time-dependent problem. With the time partition $0\equiv t_0<t_1<\cdots <t_N\equiv T$ and stepsizes $\Delta t_n=t_{n+1}-t_n,\, n=0,\ldots,N-1$, for some time $\tau\in [t_{n-1},t_n]$, one arrives at a system of the form
\begin{equation}
 \partial_tu\left|_{\tau}\right.+{\bf{A}}u(\cdot,\tau)=f(\cdot,\tau).
\end{equation}
The numerical solution $\overline{U}_J^n$ can be uniquely represented (a unique decomposition) in one of the subspaces of $S_J:=S_{J_0} \cup W_{J_0}\cup\cdots W_{J-1}$:
\begin{equation}
 \overline{U}_J^n(x)=\sum_{k\in\triangle_{J_0}}\overline{c}_{J_0,\lambda}^n\phi_{J_0,\lambda}+\sum_{j=J_0}^{J-1}\sum_{\lambda\in\nabla_j\cap \mathcal{G}^n}\overline{d}_{j,\lambda}^n\psi_{j,\lambda},
\end{equation}
which is equivalent to the unique coefficient vector
\begin{equation}
\wp^n:=
\left({\overline{{\mathrm c}}_{J_0}^n,\, \overline{{\mathrm d}}_{J_0}^n,\,\ldots, \,\overline{{\mathrm d}}_{J-1}^n}\right).
\end{equation}
For establishing the algebraic system of $\wp^n$, one can still use the Galerkin scheme; as an extension, the collocation method based on the semi-interpolation wavelet (see eq. (\ref{Cubic:1})--(\ref{Cubic:4})) can also be considered, which replaces the test function $\left\{\phi_{J_0,\,\lambda}\right\}\bigcup\left\{\psi_{j,\lambda}\right\}, \left(j,\,\lambda\right)\in \mathcal{G}^{n}$ of the Galerkin scheme by the Dirac distribution $\delta$ centered at $x_i$, being the collocation point corresponding to the index set $\mathcal{G}^{n}$, a subset of $\left\{k/2^{J_0}\right\}_{k=1}^{2^{J_0}-1}\bigcup \left\{(2k+1)/2^{j+1}\right\}_{k=0,j=J_0}^{2^j-1,J-1}$. Generally, the collocation method is more convenient and efficient for problems with variable coefficients and/or nonlinear terms. Details of such a scheme are provided in Algorithm \ref{AD-TIME},
\begin{algorithm}[! h t b p]
\caption{ADAPTIVE WAVELET SOLVER FOR THE IBVP}
\label{AD-TIME}
\begin{algorithmic}[1]
\STATE Given time partition $\{t_n\}, J_{max}$ and  threshold $\epsilon(j)$
\STATE Construct the initial irregular index set $\mathcal{G}^0$ and the  multiscale coefficients $\wp^0$ by $u_0$
\FOR {$n=1,\ldots,N$}
\STATE Based on $\wp^{n-1}$ to solve  $\wp^n$ on the index set 
\STATE Threshold $\wp^n$ to obtain the significant index set 

    \[ \tilde{\mathcal{G}}^n:=\{(j,\lambda)\in \mathcal{G}^{n-1}:|d_{j,\lambda}|\ge \epsilon(j) \} \]

\STATE Add the adjacent indices to it, and denote the result by $\mathcal{G}^n$
\IF {$\mathcal{G}^{n-1}$ and $\mathcal{G}^n$ are different}
\STATE For every index $(j,\lambda)\in\mathcal{G}^n$  not in $\mathcal{G}^{n-1}$,
   the corresponding wavelet coefficients are initialized with $0$, and denote the result by $\wp^n$
\ENDIF
\ENDFOR
\end{algorithmic}
\end{algorithm}
the steps are very similar to Algorithm \ref{AD-STATIC}, except that for treating the structures appearing in the solutions as they evolve, the computational index needs to dynamically  adapt to the local change of the regularity of the solution. For an implicit or explicit time integration, we use wavelet amplitudes of the approximate solution at the current time level to construct the irregular index for the approximate solution of the next time level. The initial irregular index $\mathcal{G}^0$ can be constructed by adding the adjacent zone to the significant index set of the initial solution $u(x,0)=g(x)$.

\section{Numerical results}
In order to illustrate the accuracy and efficiency of the proposed numerical schemes, we apply them to solve the BVP and/or IBVP (\ref{eq:1.1}). Example \ref{example1} is used to discuss the implementations of the MGM for the BVP and the collocation method for the IBVP, and in particular the convergence orders are carefully verified. We use Example \ref{example2} to show the powerfulness of the provided multilevel preconditioner and MMG. And Example \ref{example3} is used to illustrate the effectiveness  of the presented wavelet adaptive schemes.

\begin{example}\label{example1}\end{example} Consider the MGM for the  BVP (\ref{eq:1.1}) with $q=0$ and $p=\kappa_{\beta}=1$, and the source term \[f(x)=\frac{2x^{\beta}}{\Gamma(\beta+1)}-\frac{\Gamma(\nu+1)}{\Gamma(\nu+\beta-1)}x^{\nu+\beta-2}.\]
The exact solution of the problem is $u(x)=x^{\nu}-x^2$. It is well known that if $\nu>0$ and $\nu\notin\mathbb{N}$, then $u\in \mathcal{H}^{\nu+1/2-\epsilon}(\Omega)$.
For $\beta=4/5$, the numerical results are listed in Tables \ref{tab:1_1} and \ref{tab:1_2},
\begin{table}[!h t b]\fontsize{7.0pt}{12pt}\selectfont
\begin{center}
 \caption{Numerical results of the BVP (\ref{eq:1.1}), solved by MGM, with $q=0$, $p=\kappa_{\beta}=1$, and $\beta=4/5$.}
\begin{tabular} {|cc|cccc|cc|}  \hline
 $d$ & $J$    &\multicolumn{4}{c|}{$\nu=4$}                                     &\multicolumn{2}{c|}{$\nu=17/10$} \\\cline{3-8}
      &    &  $L_2$-Err    & $L_2$-Rate  &   $a(u-u_J,u-u_J)^{1/2}$  &$\mathcal{H}^{\alpha}$-Rate              &  $L_2$-Err   & $L_2$-Rate  \\
      \hline
      & $6$   &   17589e-04    &  ---        & 8.5677e-04 &  ---                                &  1.4535e-05      & ---        \\
 $d=2$& $7$   &  4.3968e-05   & 2.0001     & 3.1635e-04 &1.4374                                &   3.6314e-06     & 2.0009     \\
      & $8$   &   1.0993e-05   & 1.9999     &1.1829e-04  &1.4192                                &  9.0287e-07     & 2.0079    \\
      \hline
      &$6$   &   6.2317e-07   &  ---        & 6.2807e-06 &  ---                                 &  1.0342e-06      & ---        \\
 $d=3$& $7$  & 7.7779e-08   &3.0021       & 1.1896e-06   &2.4004                                 &  2.2509e-07      &2.2000   \\
     & $8$   &   9.7152e-09   &3.0011       & 2.2531e-07 & 2.4005                                & 4.8988e-08      &2.2000     \\
    \hline
   \end{tabular}\label{tab:1_1}
\end{center}
\end{table}
\begin{table}[!h t b p]\fontsize{6.0pt}{11pt}\selectfont
\begin{center}
 \caption{Numerical results of the BVP (\ref{eq:1.1}), solved by MGM, with $q=0$, $p=\kappa_{\beta}=1$, and $\beta=4/5$.}
\begin{tabular} {|c|cc|cc|cc|cc|}  \hline
    $J$    &\multicolumn{4}{c|}{$\nu=11/10$} &\multicolumn{4}{c|}{$\nu=21/10$}  \\ \cline{2-9}
           &\multicolumn{2}{c|}{$d=3$}&\multicolumn{2}{c|}{$d=4$}&\multicolumn{2}{c|}{$d=3$}&\multicolumn{2}{c|}{$d=4$}\\
          &  $L2$-Err    & $L_2$-Rate     &  $L_2$-Err    & $L_2$-Rate   &  $L2$-Err    & $L_2$-Rate     &  $L_2$-Err   & $L_2$-Rate  \\ \hline
    $6$   &   1.4385e-05 &  ---        & 8.0390e-06         &---         & 1.2656e-07& ---                     &   3.2703e-08  &   ---       \\
    $7$   &   4.7453e-06  & 1.6000      & 2.6516e-06         &1.6002       &2.0865e-08& 2.6007                 &  5.3930e-09   &   2.6002    \\
    $8$   &   1.5654e-06  & 1.6000      & 8.7469e-07         &1.6002       &3.4407e-09& 2.6003                 & 8.8950e-10   &   2.6000     \\
    \hline
   \end{tabular}\label{tab:1_2}
\end{center}
\end{table}
 which confirm that if the analytical  solution is smooth enough, the convergence order is $d$ and $d-\alpha$ in the $L_2$ and $\mathcal{H}^{\alpha}$-norm,  respectively. Otherwise the convergence order is limited by the regularity of the solution, but the approximation accuracy is improved when the high order bases are used. Moreover, If the modified Galerkin method, e.g., the one proposed in \cite{Jin:1}, is used, for this type problem  one can have a  convergence rate $d-\beta$ for the sufficient smooth source term $f$; when $f=1$, the numerical results are listed in Table \ref{tab:1-3}.
 \begin{table}[!h t b]\fontsize{7.0pt}{12pt}\selectfont
\begin{center}
 \caption{Numerical results of the BVP (\ref{eq:1.1}), solved by MGM, with $q=0$, $p=\kappa_{\beta}=1$, $f=1$, $d=4$, and $\mu=4$.}
\begin{tabular} {|c|cc|cc|cc|}  \hline
    $J$    &\multicolumn{2}{c|}{$\beta=4/5$}   &\multicolumn{2}{c|}{$\beta=1/2$}      &\multicolumn{2}{c|}{$\beta=1/5$} \\\cline{2-7}
          &  $L2$-Err    & $L_2$-Rate          &  $L_2$-Err    & $L_2$-Rate                   &  $L_2$-Err   & $L_2$-Rate  \\ \hline

    $6$   &    2.4209e-07  &  ---      & 9.0406e-09         &---                               &  3.0206e-10    &   ---       \\
    $7$   &    2.6360e-08  & 3.1991   &  8.0061e-10         &3.4973                             & 2.1733e-11    & 3.7969    \\
    $8$   &    2.8694e-09  & 3.1996    & 7.0772e-11        &3.4999                              & 1.5568e-12    & 3.8032   \\
    \hline
   \end{tabular}\label{tab:1-3}
\end{center}
\end{table}
 We want to emphasize that including the boundary bases is very important to ensure the polynomial exactness  (known as the Strang-fix condition), which is the foundation to have the desired convergence results.  The numerical results in Table \ref{tab:2} are for the cases that the boundary base functions are absent (one base is removed for $d=3$ and two for $d=4$).
At this moment the exact solution after zero extension is required to have sufficient regularity to recover the desired convergence order. The similar observations are also detected for the finite difference methods, and the ways of recovering the optimal convergence orders are presented in \cite{Chen:15, Zhao:15}.
\begin{table}[!h t b]\fontsize{7.0pt}{12pt}\selectfont
\begin{center}
 \caption{Numerical results of the BVP (\ref{eq:1.1}), solved by inner MGM, with $q=0$, $p=\kappa_{\beta}=1$, and $\beta=4/5$.}
\begin{tabular}{|c c|c c |c c|c c|}
  \hline
$d$  &   $J$    &\multicolumn{2}{c|}{$u(x)=x^{4}-x^2$}  &\multicolumn{2}{c}{$u(x)=x^2(x-1)^2$}     \vline& \multicolumn{2}{c|}{$u(x)=x^3(x-1)^3$} \\
  \cline{3-8}
    &      &  $L_2$-Err    & $L_2$-Rate                                &  $L_2$-Err    & $L_2$-Rate       &  $L_2$-Err    & $L_2$-Rate \\
\hline
    &$6$   &   9.8488e-3    & ---                                      & 1.2187e-06   &  ---              &5.9382e-07     & ---           \\
3   &$7$   &   4.9332e-3    & 0.9974                                   & 1.5229e-07   & 3.0004            & 7.5027e-08    &2.9845         \\
    &$8$   &   2.4688e-3    & 0.9987                                    & 1.9027e-08   & 3.0007            & 9.4230e-09    &2.9931         \\
 \hline
    &$6$   &    1.9092e-2    & ---                                      & 8.9636e-05   &  ---               & 5.9852e-08   & ---          \\
4   &$7$   &    9.6958e-3    &0.9775                                    & 2.2395e-05   &  2.0009            & 3.7671e-09   &  3.9898        \\
    &$8$   &    4.8857e-3    &0.9888                                    & 5.5941e-06   &  2.0012            & 2.3616e-10   &  3.9956       \\
 \hline
\end{tabular}\label{tab:2}
\end{center}
\end{table}

We further consider the collocation method for the variable-coefficient version of the IBVP (\ref{eq:1.1}). The collocation points are chosen as $\big\{1/2^{J+1}, k/2^{J}\big|_{k=1}^{2^J-1}, 1-1/2^{J+1}\big\}$, and the approximation properties of the cubic spline collocation method are discussed in the space $S_J:={\rm span}\{\Phi_J\}$ with $\Phi_J=\{\phi_{J,k},k\in\triangle_J, d=4\}$.
The considered equation is
\begin{equation}\label{collocation:1}
  u_t-(k_1x^{2-\beta}{}_0D_x^{2-\beta}u + k_2(1-x)^{2-\beta}\ {}_xD_1^{2-\beta}u)=f \qquad t\in (0,T]
\end{equation}
with the right-hand term
\begin{eqnarray*}
f(x,t)&=&-12\exp(-t)\Big\{x^2(1-x)^2+\frac{1}{6}\left[k_1x^2+k_2(1-x)^2\right]\\
&&-\frac{1}{\beta+1}\left[k_1x^3+k_2(1-x)^3\right]+\frac{2}{(\beta+1)(\beta+2)}\left[k_1x^4+k_2(1-x)^4\right]\Big\}
\end{eqnarray*}
and the initial condition $u(x,0)=x^2(1-x)^2$. Then it can be checked that the analytical solution is  $u(x,t)=\exp(-t)x^2(1-x)^2$.

The Crank-Nicolson scheme is used to get the full discretization approximation of (\ref{collocation:1}) with the time stepsize $1/2^{2J}$ and $T=1/2$.
\begin{table}[!h t b]\fontsize{7.0pt}{12pt}\selectfont
\begin{center}
 \caption{Convergence performance of the cubic spline collocation method with $ k_1=k_2=1$.}
\begin{tabular}{| c|c c |c c|c c|}
    \hline
        $J$    &\multicolumn{2}{c|}{$\beta=2/10$}  &\multicolumn{2}{c|}{$\beta=8/10$}     & \multicolumn{2}{c|}{$\beta=0$} \\
  \cline{2-7}
& $L_{\infty}$-Err      & $L_{\infty}$-Rate                           &$L_{\infty}$-Err  & $L_{\infty}$-Rate  & $L_{\infty}$-Err  & $L_{\infty}$-Rate \\
   \hline
    $5$    &    5.8773e-05    & ---                                           & 1.8431e-06    &  ---       &1.6167e-04     & ---        \\
    $6$    &    1.2862e-05     & 2.1920                                       & 2.6580e-07   & 2.7937      & 4.0533e-05     &1.9958    \\
    $7$    &    2.8036e-06     & 2.1977                                       & 3.8223e-08  &  2.7978      & 1.0441e-05     &1.9990      \\
 \hline
\end{tabular}\label{tab:4}
\end{center}
\end{table}
Table \ref{tab:4} shows the expected convergence order $2+\beta$ of collocation method for fractional PDE, agreeing with the classical conclusion when $\beta=0$.  Though the superconvergence can be obtained for classical PDE by carefully averaging the derivative values gotten in collocation points, it seems that this result maynot be directly extended to fractional PDE.
For convenience, let's consider the frequently used Hermite spline collocation method. Take the collocation space
\[V_J:={\rm span}\left\{\pi_2(2^Jx+1)\big|_{\Omega}, \pi_1(2^Jx-k),\pi_2(2^Jx-k),\pi_2(2^Jx-2^J+1))\big|_{\Omega}\right\},\]
where $\big|_{\Omega}$ denotes the restriction in $\Omega$, $k=0,1,\cdots,2^J-2$, and $\pi_1,\pi_2$ are the cubic Hermite compactly supported functions given as
\begin{eqnarray*}
&& \pi_1(x)=\left\{\begin{array}{lc} -x^2(2x-3)&0\le x<1,\\(x-2)^2(2x-1)&\quad1\le x\le 2, \end{array}\right.\\[3pt]
&& \  \pi_2(x)=\left\{\begin{array}{lc} x^2(x-1)&0\le x<1,\\(x-2)^2(x-1)&\quad 1\le x\le 2. \end{array}\right.
\end{eqnarray*}
To determine the unknown coefficients, one need total $2^{J+1}$ collocation points. As well known for classic PDE ($\beta=0$), when the general collocation points such as the third-quarter points of every interval $[i/2^J, (i+1)/2^J], i=0,1,\cdots,2^J-1$ are used, the convergence order is $2$. But if the Gauss nodes are used, one arrives at the superconvergence result of order $4$. Unfortunately,  in both case the convergence order are  $2+\beta$ for the fractional PDE, except the approximation accuracy maybe improved.
The numerical results are presented in Table \ref{tab:5}, where the abbr `Equi' and `Gauss' denote the two types of collocation points mentioned above.
\begin{table}[!h t b]\fontsize{5.0pt}{11pt}\selectfont
\begin{center}
 \caption{Convergence performance of the cubic Hermite collocation method  with $ k_1=1$ and $k_2=0$.}
\begin{tabular}{| c|c c |c c|c c|c c|}
  \hline
  $J$    &\multicolumn{2}{c|}{$\beta=5/10$, Equi}  &\multicolumn{2}{c|}{$\beta=5/10$, Gauss}   & \multicolumn{2}{c|}{$\beta=0$, Equi} & \multicolumn{2}{c|}{$\beta=0$, Gauss} \\
  \cline{2-7}\cline{8-9}
  & $L_{\infty}$-Err      & $L_{\infty}$-Rate  &$L_{\infty}$-Err  & $L_{\infty}$-Rate  & $L_{\infty}$-Err  & $L_{\infty}$-Rate & $L_{\infty}$-Err & $L_{\infty}$-Rate \\
   \hline
    $5$    & 8.2952e-06     & ---               & 1.9362e-07    &  ---        &  3.5693e-05    & ---          &  3.5875e-08  & --- \\
    $6$    & 1.4663e-06     & 2.5001            &  3.2737e-08   & 2.5642     & 8.9270e-06      &1.9994        & 2.2506e-09  &3.9946\\
    $7$    & 2.5912e-07     & 2.5005            &  5.6891e-09   & 2.5247      & 2.2316e-06     &2.0001        &  1.4098e-10  &3.9968\\
 \hline
\end{tabular} \label{tab:5}
\end{center}
\end{table}
\begin{remark}
 Unlike the Galerkin method, the differential matrix of right derivative in collocation method is not the transpose of its left twin. Instead, combining the symmetry of the scaling spline bases and collocation points, it is easy to prove that \[A_r=A_l({\rm end}:-1:1,{\rm end}:-1:1).\]
 \end{remark}

\begin{example}\label{example2} \end{example} Now, we focus on the wavelet multilevel schemes for solving the fractional PDEs. The presented numerical results with $d=2$, and in this case the coefficient matrix has a full Toeplitz structure. The matrix-vector product is performed by FFT. For the other bases, the computational procedure is almost the same after a slight modification, e.g., when $d=3$, $A_1{\boldsymbol {\mathrm c}}_j:=\left({}_0D_x^{-\beta/2}D\Phi_J, {}_xD_1^{-\beta/2}D\Phi_J\right){\boldsymbol {\mathrm c}}_j$ can be decomposed into several blocks with $H$ being the Toeplitz matrix:
 \begin{eqnarray*}
&&\big(A_1{\boldsymbol {\mathrm c}}_j\big)(1)=\big[a_1,r(\boldsymbol {\mathrm{a_2}})^T,0\big]{\boldsymbol {\mathrm c}}_j,\\
&&\big(A_1{\boldsymbol {\mathrm c}}_j\big)(2: {\rm end}-1)={\boldsymbol {\mathrm c}}_j(1)\,\boldsymbol {\mathrm{a_1}}+H{\boldsymbol {\mathrm c}}_j(2:{\rm end}-1)+{\boldsymbol {\mathrm c}}_j({\rm end})\,\boldsymbol {\mathrm{a_2}},\\
&&\big(A_1{\boldsymbol {\mathrm c}}_j\big)({\rm end})=\big[a_2,r(\boldsymbol {\mathrm{a_1}})^T,a_1\big]{\boldsymbol {\mathrm c}}_j.
\end{eqnarray*}

For the BVP, we first reveal that the multilevel preconditioning brings a uniform  matrix condition number and an improved spectral distribution. Considering the BVP (\ref{eq:1.1}) with $\kappa_{\beta}=1, p=1$ and $\kappa_{\beta}=1,p=1/2$, the condition numbers for different $\beta$ are presented in Table \ref{tab:2_1};
one can see that without preconditioning, the condition number of the stiffness matrix behaves like $\mathcal{O}(2^{J(2-\beta)})$, which means the conditional number increases fast with the refinement especially when $\beta$ is small.
 \begin{table}[!h t b p]\fontsize{6.0pt}{11pt}\selectfont
\begin{center}
 \caption{Primal condition numbers of the BVP (\ref{eq:1.1}) with $q=0,\,\kappa_{\beta}=1$, and $d=2$.}
\begin{tabular} {|c|cc|cc|cc|cc|}  \hline
$J$ &\multicolumn{2}{c|}{$p=1, \beta=1/2$} &\multicolumn{2}{c|}{$p=1/2,\beta=1/2$}  &\multicolumn{2}{c|}{$p=1,\beta=1/5$}    &\multicolumn{2}{c|}{$p=1/2,\beta=1/5$} \\\cline{2-9}
          &  Con-Num   & Rate     &  Con-Num  & Rate         & Con-Num  & Rate   &Con-Num & Rate \\ \hline

    $8$   &     1.4763e+03  &  ---    & 1.8304e+03           &   ---    & 8.7494e+03         &---             &9.1119e+03 &---  \\
    $9$   &    4.1754e+03  & 1.5000  & 5.1784e+03            & 1.5003     & 3.0467e+04       &1.8000           &3.1732e+04 &1.8001 \\
    $10$   &    1.1810e+04  & 1.5000   & 1.4648e+04           & 1.5002      &1.0609e+05       &1.8000           &1.1050e+05 &1.8000\\
    \hline
   \end{tabular}\label{tab:2_1}
\end{center}
\end{table}
 After preconditioning, the uniformly bounded condition numbers with different wavelet preconditioners are obtained; see Table \ref{tab:2_2}, where  `inte-', `Semi-', and `Bior- ($\tilde{d}$)' denote the interpolation wavelet, semiorthogonal wavelet, and biorthogonal wavelet ${}^{2,2}\psi$, respectively, having been introduced in Section 2. Note that when performing the decomposition by semiorthogonal and biorthogonal wavelets, the interpolation wavelet has been used for $S_1$ and $S_2$.
\begin{table}[!h t b]\fontsize{7.0pt}{12pt}\selectfont
\begin{center}
 \caption{ Preconditioned condition numbers of the BVP (\ref{eq:1.1}) with  $q=0,\, \kappa_{\beta}=1$, $d=2$, and $J_0=0$.}
\begin{tabular} {|cc|ccc|ccc|}  \hline
 $p$ & $J$    &\multicolumn{3}{c|}{$\beta=1/2$}                        & \multicolumn{3}{c|}{$\beta=1/5$}  \\\cline{3-8}
      &       & Inte-&    Semi-      & Bior- ($\tilde{d}=2$)        & Inte-  & Semi-     & Bior- ($\tilde{d}=2$) \\
      \hline
      & $8$   &   3.0970    & 5.8363    &  13.3957                    & 1.5953      & 10.2897      & 12.2315      \\
 $p=1$& $9$   &   3.2286    & 6.1158    & 14.4784                     &1.6269        &10.5561      &12.9767    \\
      & $10$   &  3.3457    & 6.3622     &15.4103                     &1.6540       &10.7779      & 13.5788   \\
      \hline
      &$8$   &   3.2614   &  8.0344      & 12.7702                      & 1.5935       & 11.1094      &12.2830      \\
$p=1/2$& $9$  & 3.4745   &  8.1648       & 13.6624                       &1.6296        &11.4094      &13.0063      \\
     & $10$   &  3.6686   &  8.2634      & 14.4026                      & 1.6608       & 11.6511      &13.5854       \\
    \hline
   \end{tabular}\label{tab:2_2}
\end{center}
\end{table}
 We also display the matrix eigenvalue distribution for $\beta=1/5$ in Figures  \ref{Fig-Precondtion1} and \ref{Fig-Precondtion2}; they show the preconditioning benefits of  a more concentrated eigenvalue distribution.
\begin{figure}[!h t b p]
\begin{center}
\includegraphics[width=1.1in,height=1.1in,angle=0]{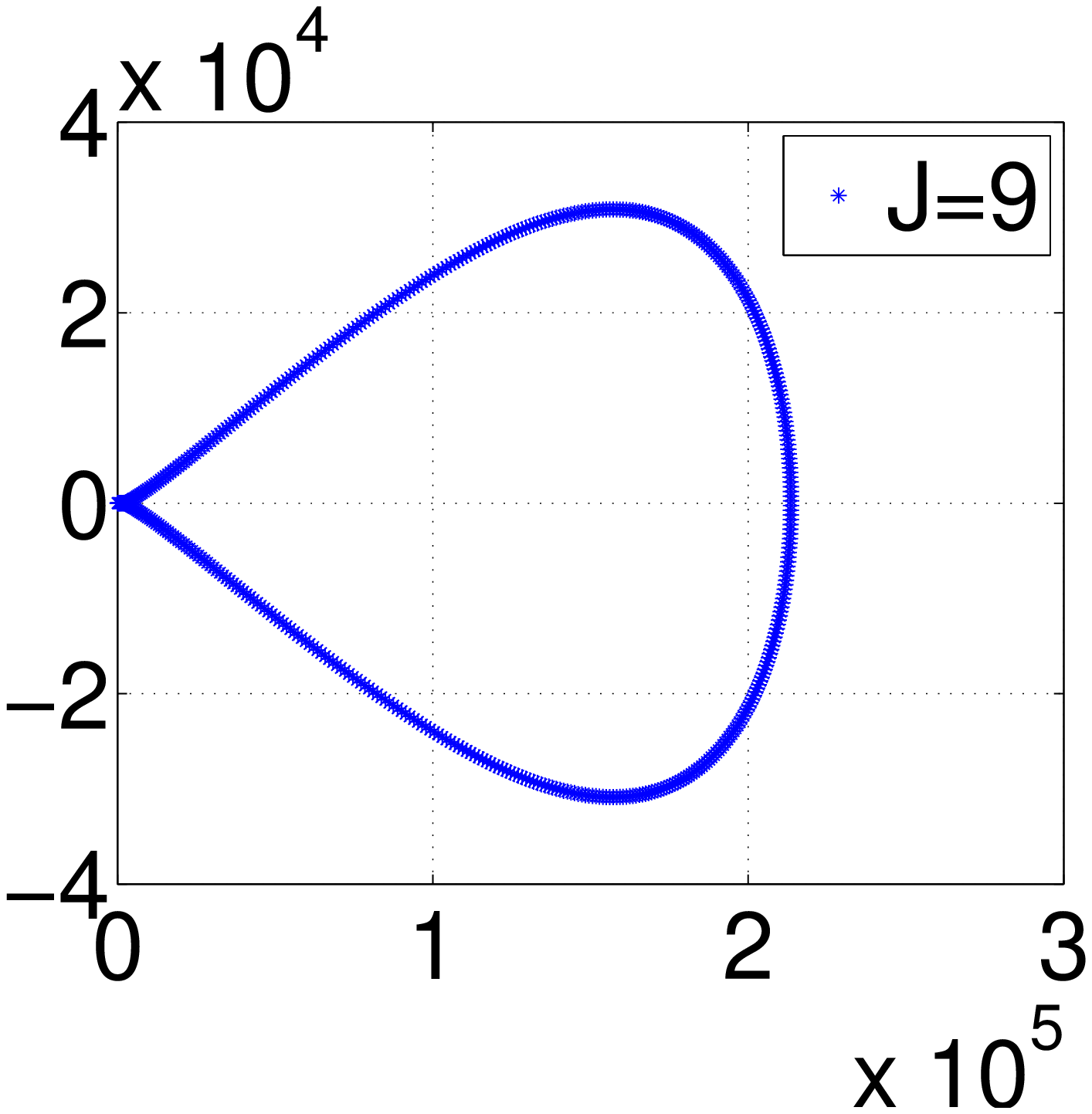}
\includegraphics[width=1.1in,height=1.1in,angle=0]{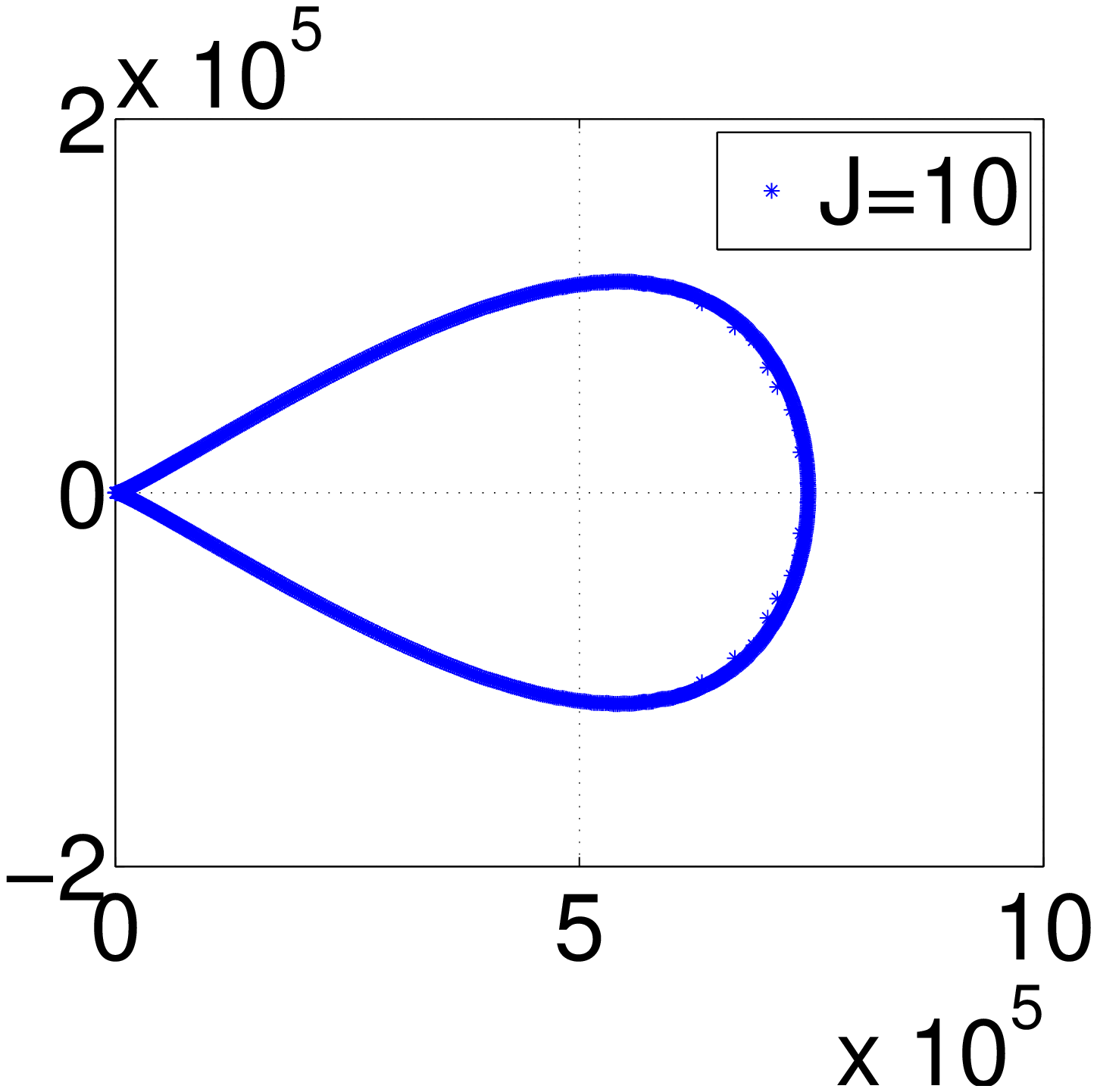}
\includegraphics[width=1.1in,height=1.1in,angle=0]{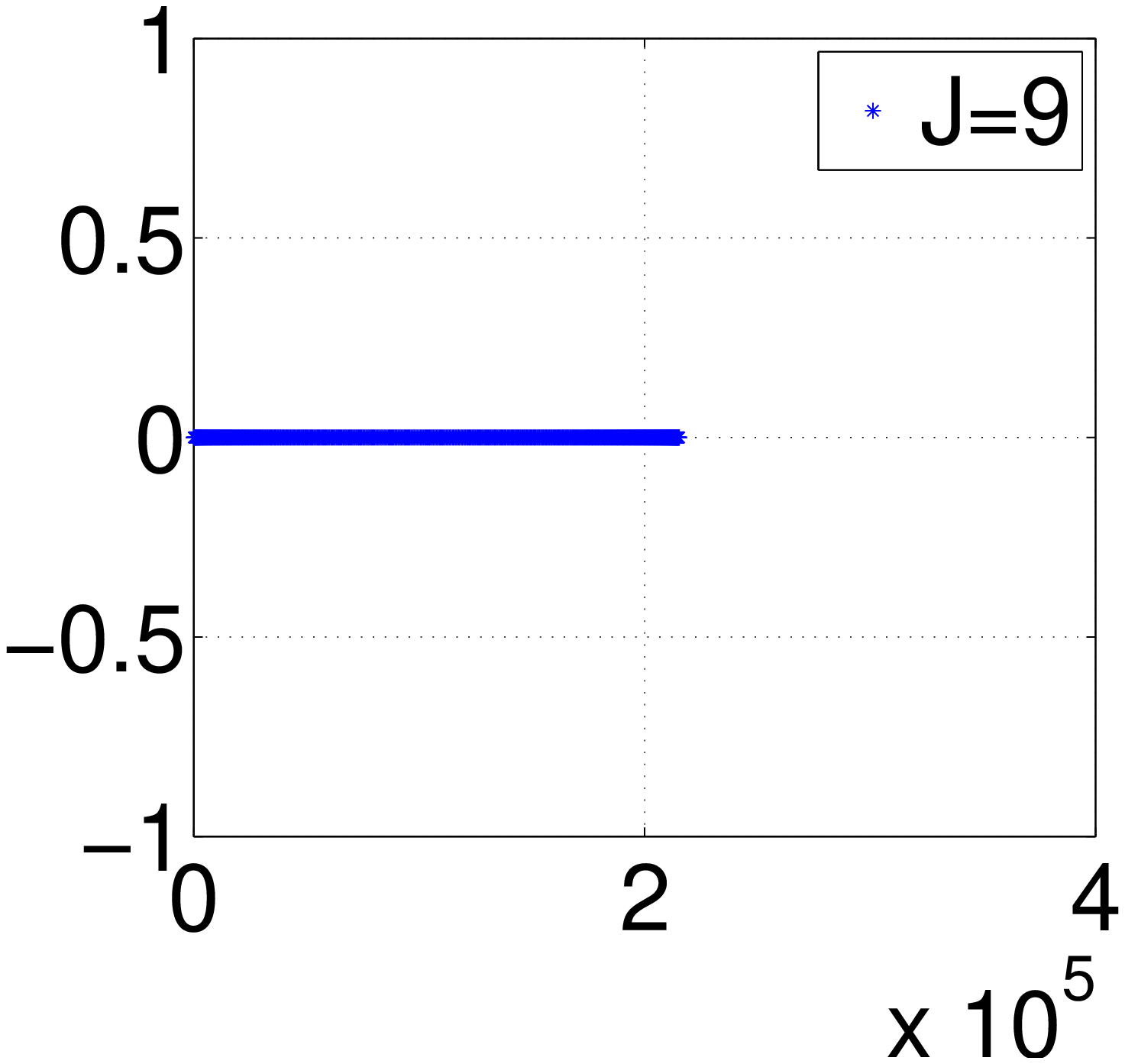}
\includegraphics[width=1.1in,height=1.1in,angle=0]{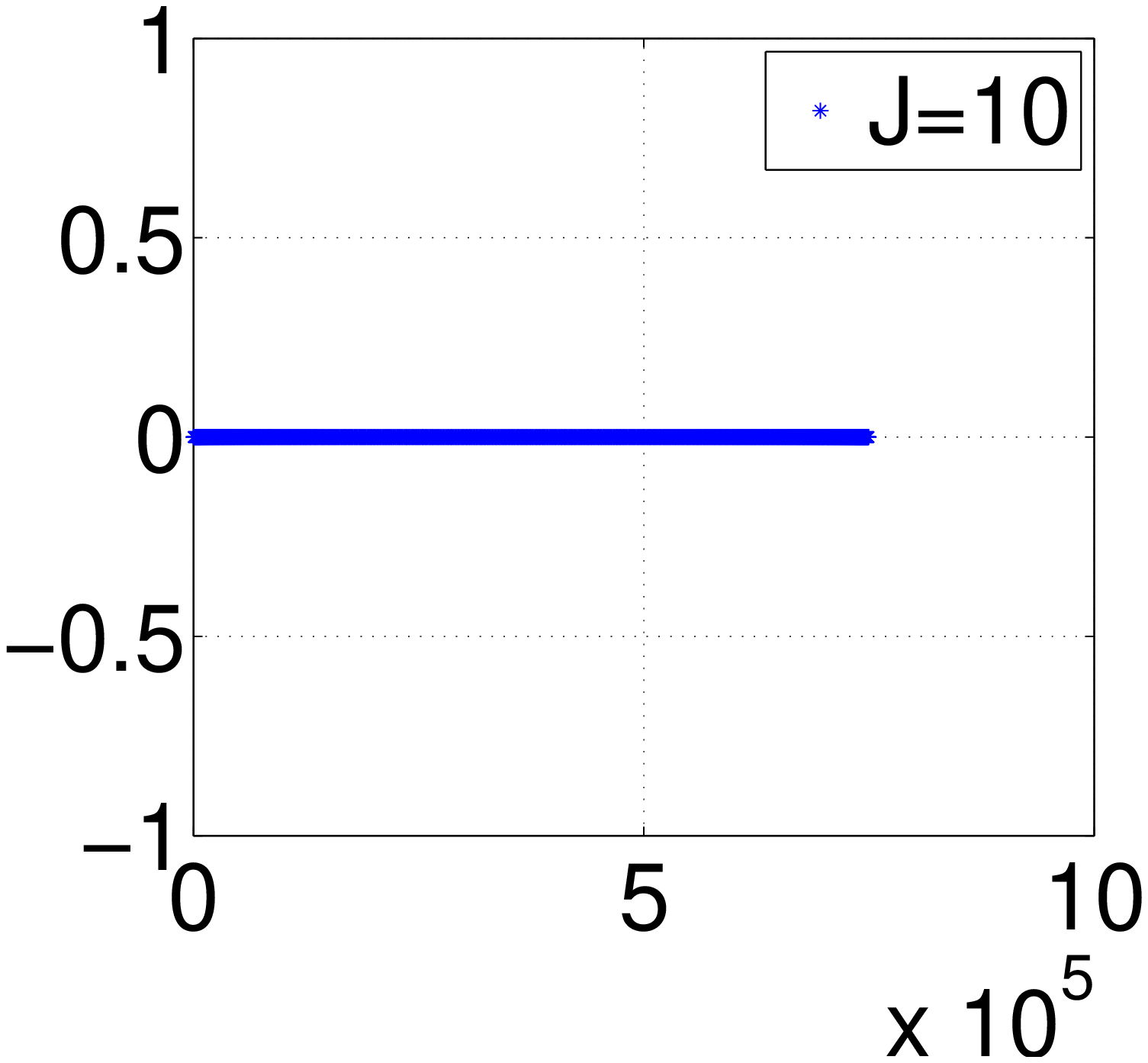}
\caption{Eigenvalue distribution of the BVP matrix with $p=1$ (first two) and $p=1/2$ (last two).}\label{Fig-Precondtion1}
\end{center}
\end{figure}
\begin{figure}[!h t b p]
\begin{center}
\includegraphics[width=1.1in,height=1.1in,angle=0]{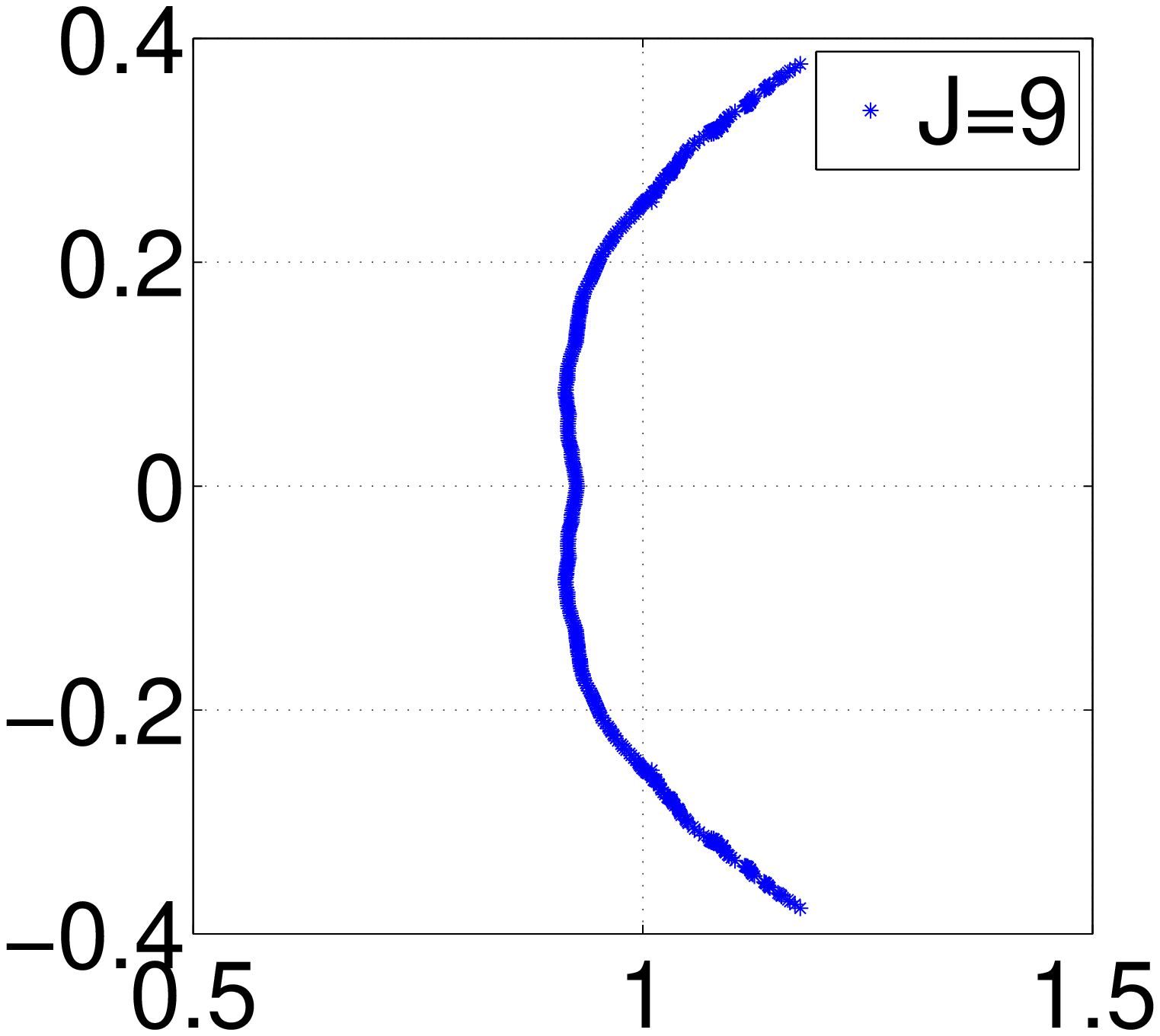}
\includegraphics[width=1.1in,height=1.1in,angle=0]{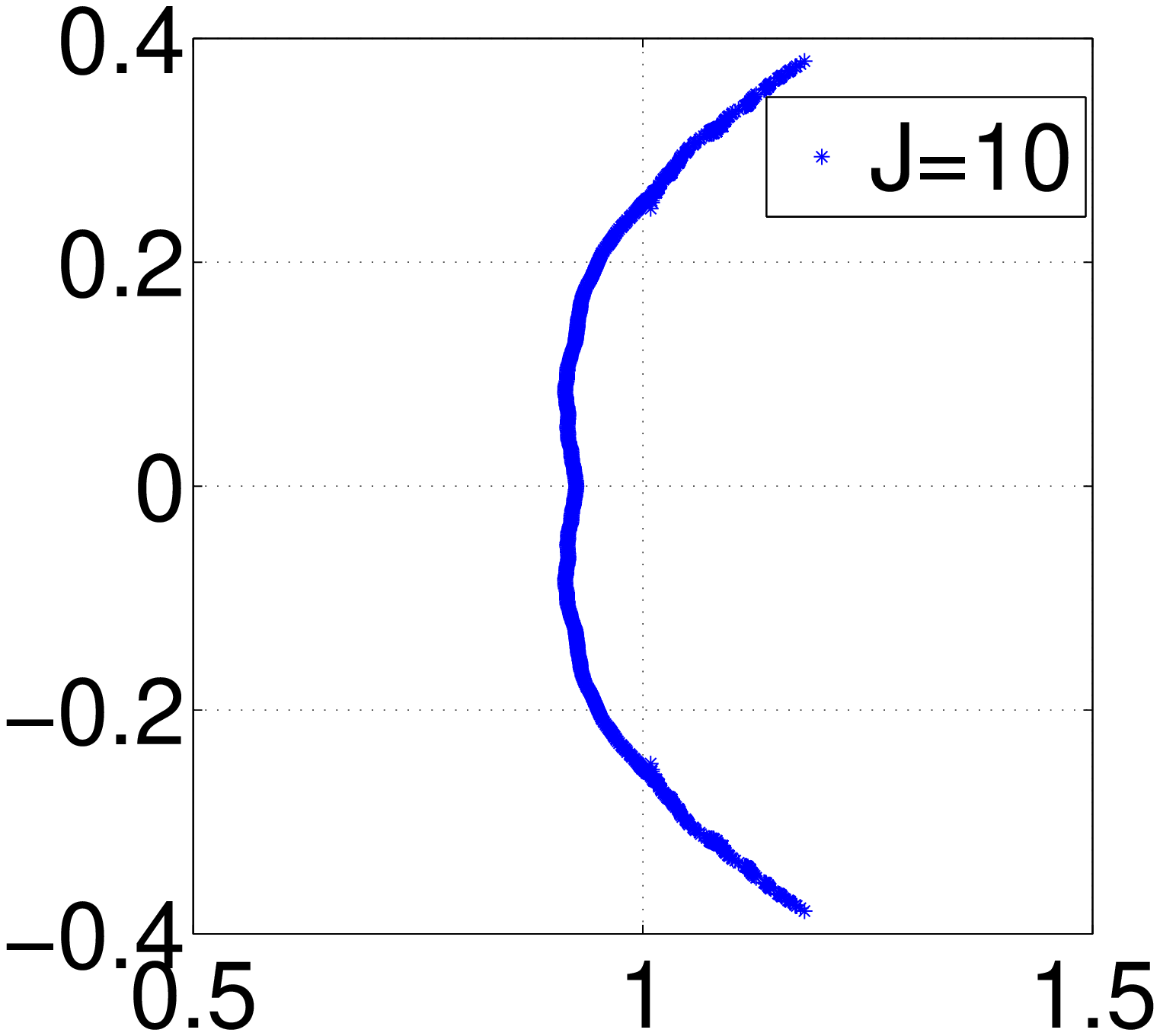}
\includegraphics[width=1.1in,height=1.1in,angle=0]{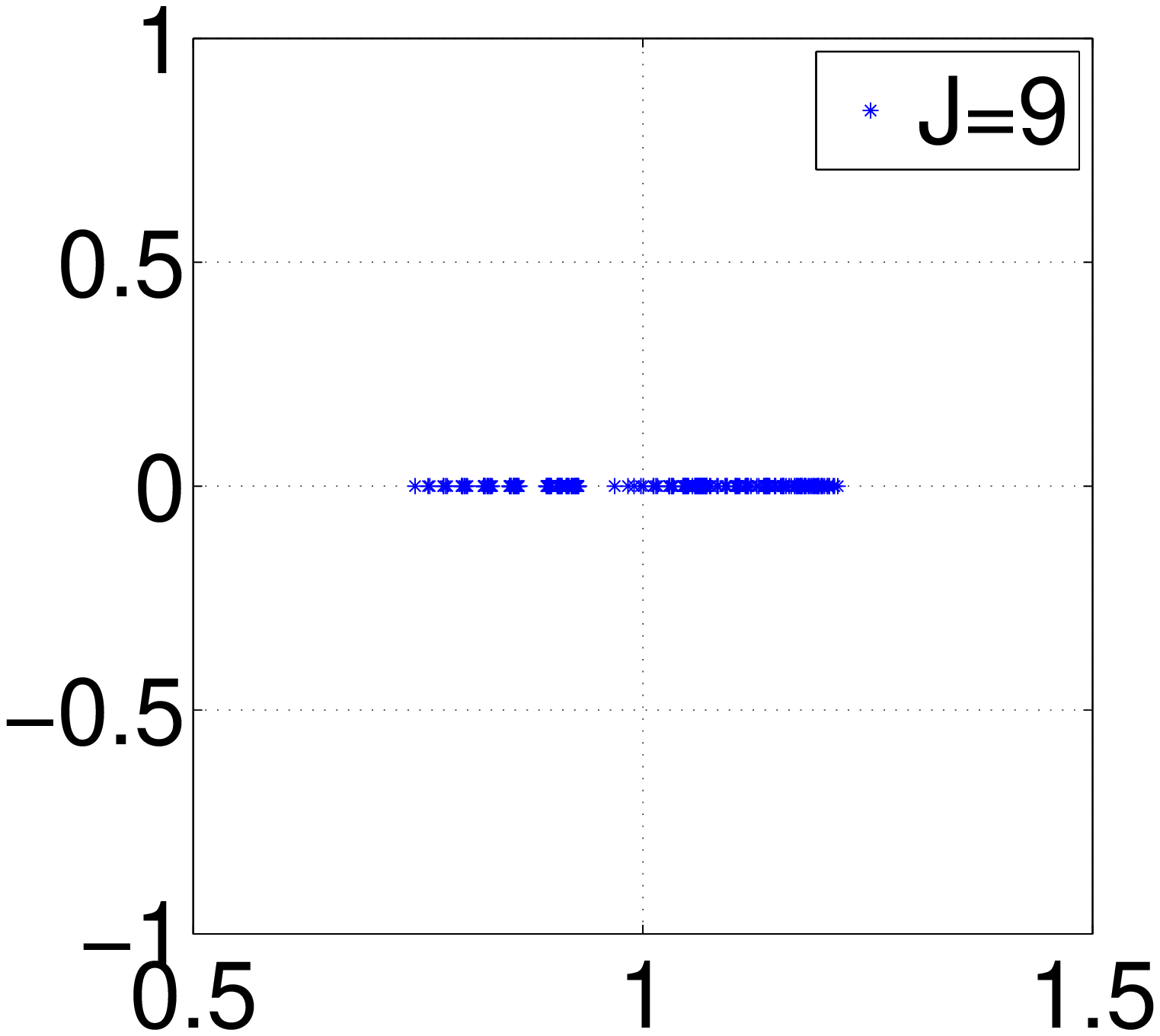}
\includegraphics[width=1.1in,height=1.1in,angle=0]{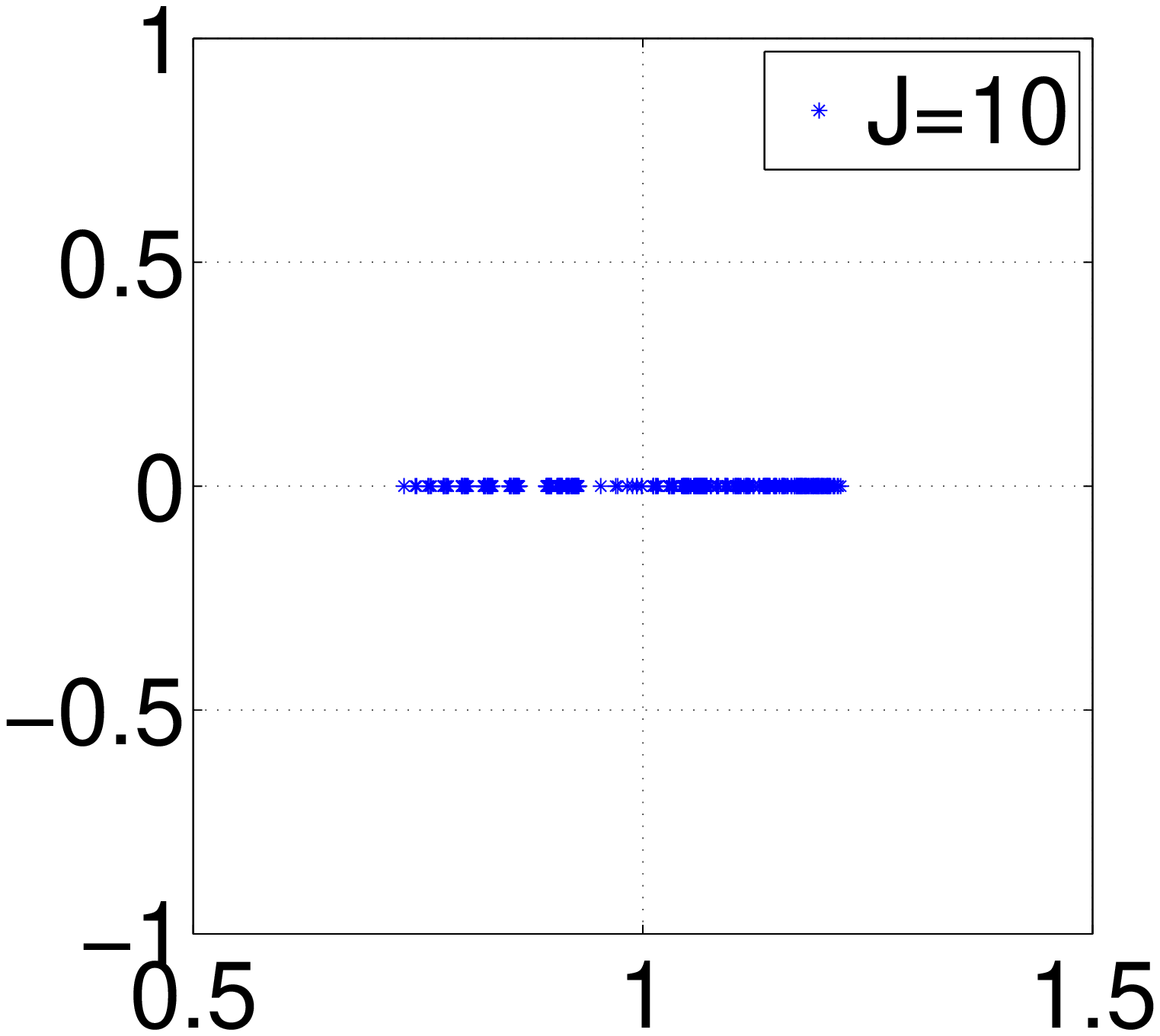}

\includegraphics[width=1.1in,height=1.1in,angle=0]{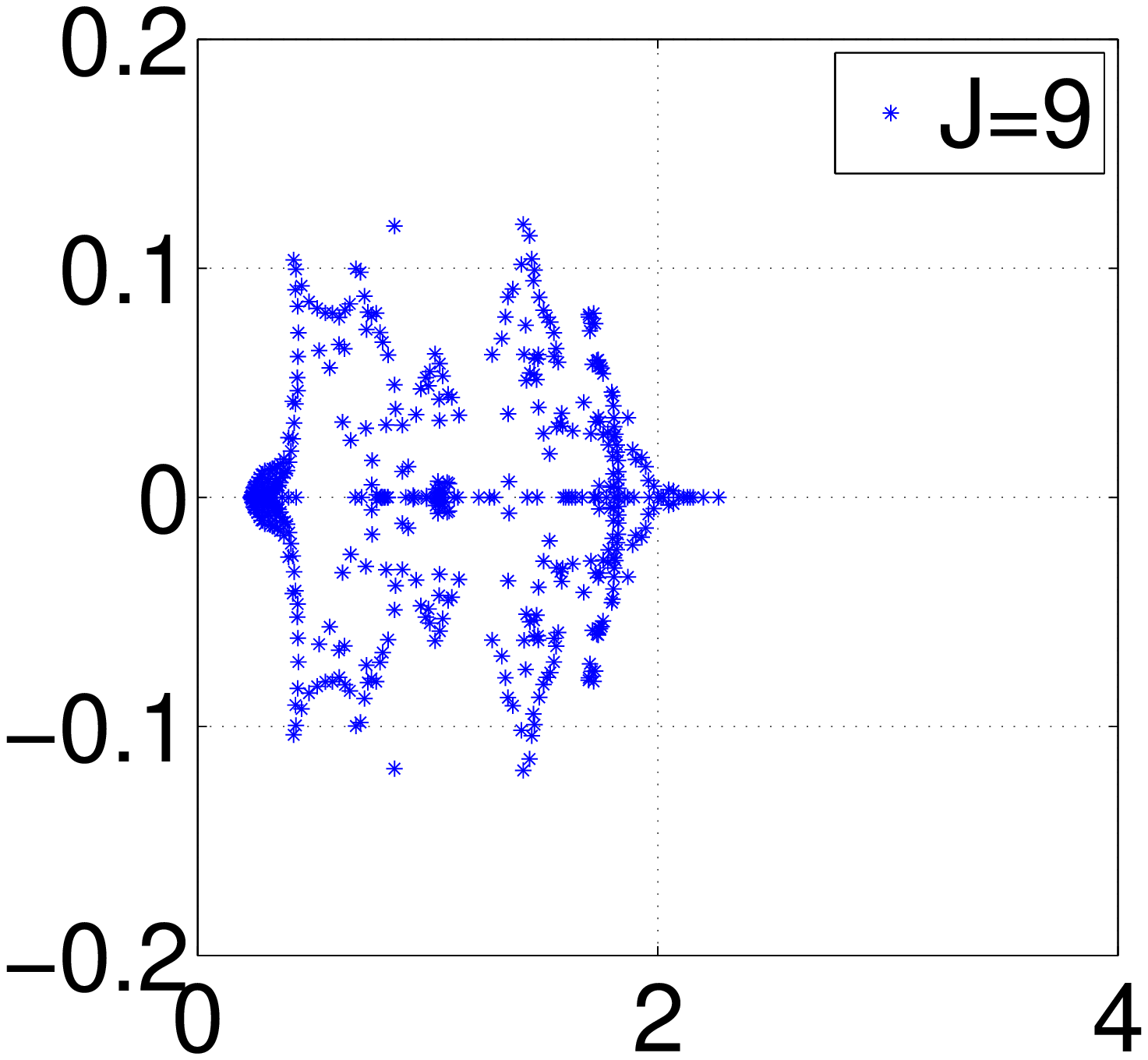}
\includegraphics[width=1.1in,height=1.1in,angle=0]{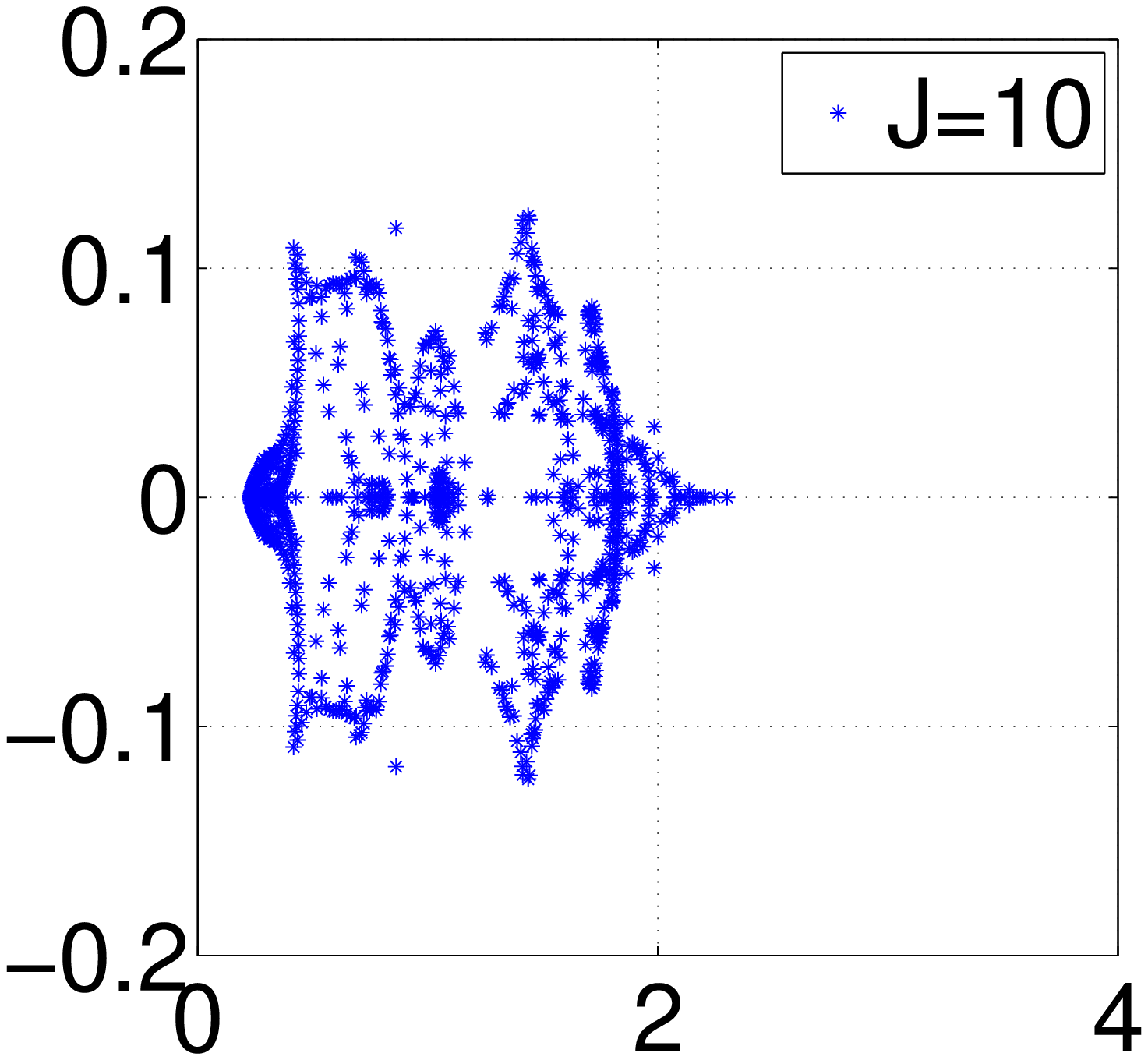}
\includegraphics[width=1.1in,height=1.1in,angle=0]{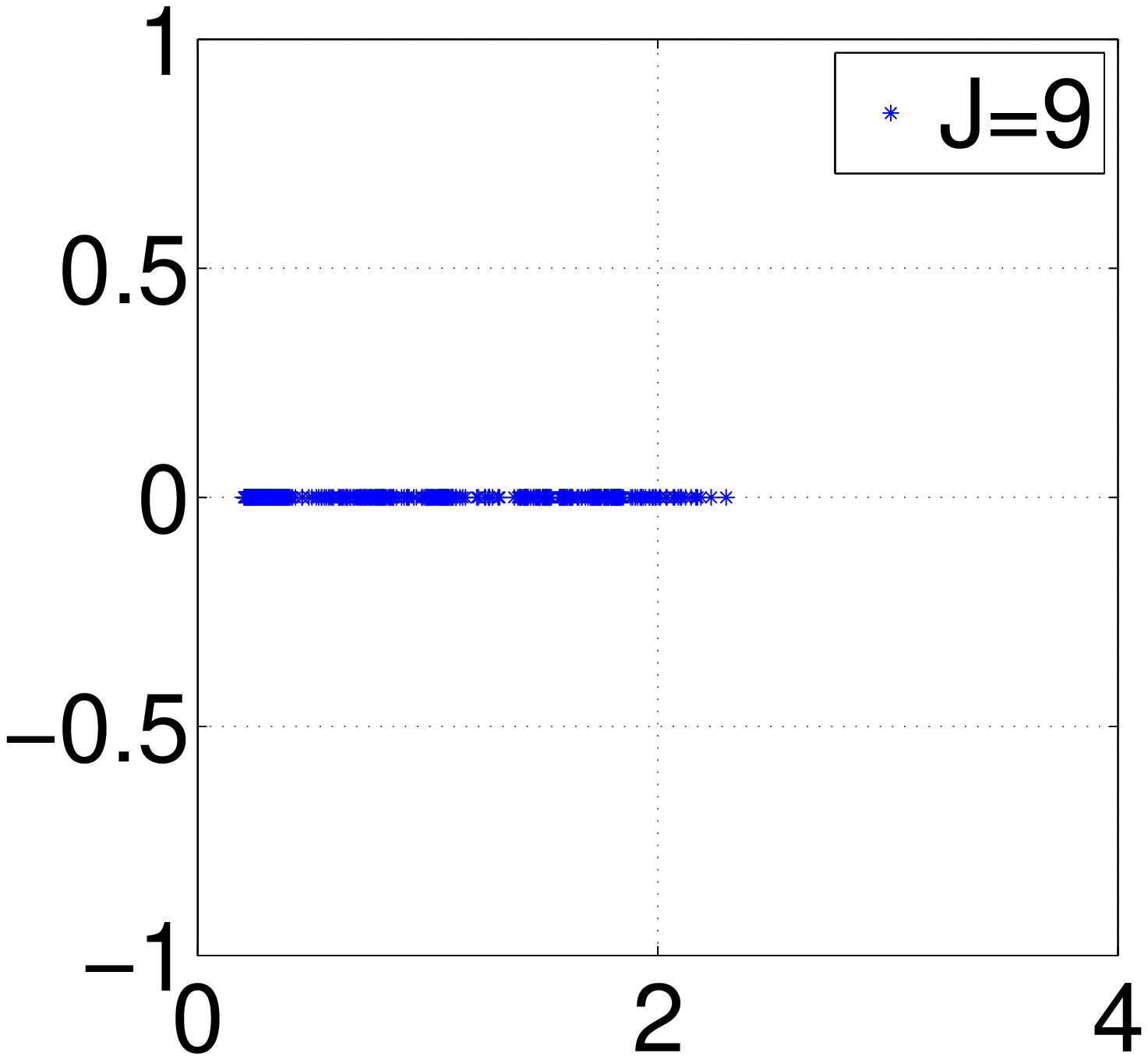}
\includegraphics[width=1.1in,height=1.1in,angle=0]{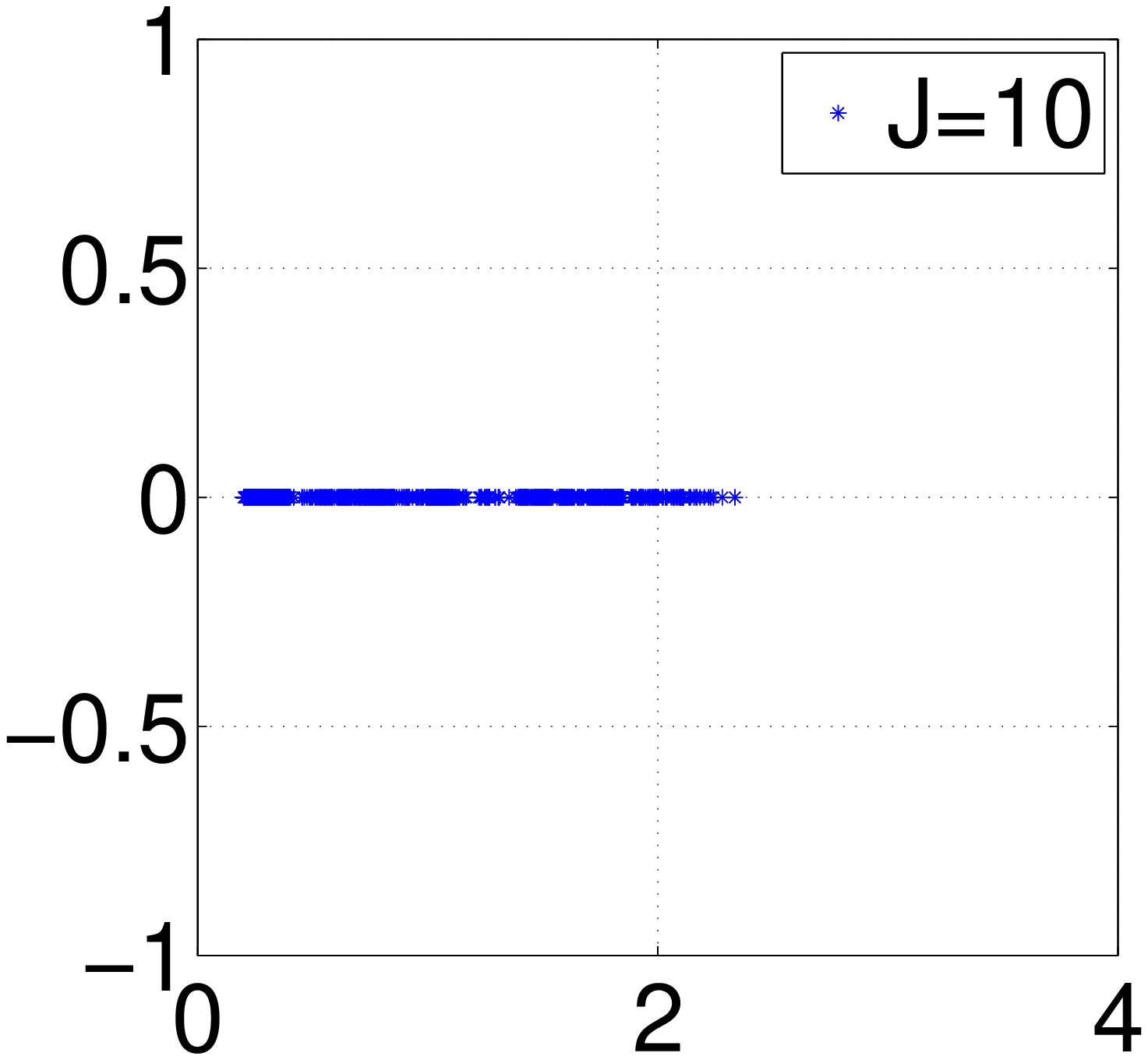}
\caption{Eigenvalue distribution of the preconditioned BVP systems with the interpolation wavelet (first line) and  the semiorthogonal wavelet (second line), respectively (the first two columns are for $p=1$ and the last two columns for  $p=1/2$).}\label{Fig-Precondtion2}
\end{center}
\end{figure}

To explore the effectiveness of this preconditioned system, we numerically solve BVP (\ref{eq:1.1}) with
\[ f=\frac{1}{\Gamma(1+\beta)}\left(-2x^{\beta}+\beta x^{\beta-1}\right),\]
and
\[f= \frac{1}{2\Gamma(1+\beta)}\left(-2x^{\beta}+\beta x^{\beta-1}-2(1-x)^{\beta}+\beta (1-x)^{\beta-1}\right),\]
for $p=1$ and $p=1/2$, respectively. We use  GMRES and Bi-CGSTAB to solve the algebraic system before and after preconditioning, and the numerical results are given in Tables \ref{tab:2_3} and \ref{tab:2_4}, respectively.
The comparisons for the two methods are made almost with the same $L_2$ approximation error, not listed in the tables. The stopping criterion for solving the linear systems is
 \[ \frac{\|r(k)\|_{l_2}}{\|r(0)\|_{l_2}}\le 1e-8,\]
 with $r(k)$ being the residual vector of linear systems after $k$ iterations. It should be noted that the GMRES method for $p=1$  without preconditioning stops before reaching this criterion. In fact, by the two-dimension FWT and the properties of tensor product, these proposed preconditioner can also easily apply to two dimension and it also works well for the algebraic systems generated by the finite difference methods, e.g., \cite{Meerschaert:04, Tian:15}.
 \begin{table}[!h t b p]\fontsize{6.0pt}{11pt}\selectfont
\begin{center}
 \caption{Numerical results of the BVP (\ref{eq:1.1}), solved by GMRES and Bi-CGSTAB, with $q=0,\,\kappa_{\beta}=1,\,\beta=1/5$, and $d=2$.}
\begin{tabular} {|c|cc|cc|cc|cc|}  \hline
$J$  &\multicolumn{2}{c|}{$p=1$, GMRES} &\multicolumn{2}{c|}{$p=1/2$, GMRES}  &\multicolumn{2}{c|}{$p=1$, Bi-CGSTAB}    &\multicolumn{2}{c|}{$p=1/2$, Bi-CGSTAB} \\\cline{2-9}
         &       Iter &   CPU(s)      & Iter & CPU(s)        & Iter & CPU(s)   &Iter & CPU(s)\\ \hline

    $8$  &      2.5500e+02   &  0.3443    &  1.1800e+02        & 0.0915    &  2.6350e+02    & 0.0672            & 1.1700e+02 & 0.0303 \\
    $9$  &     5.1100e+02  &  1.5217  & 2.2000e+02            &  0.3230        &  5.3550e+02      &0.2408        &2.0950e+02 &0.1012 \\
    $10$  &    1.0230e+03  & 7.4783   & 4.1200e+02              & 1.3219         & 1.1665e+03    & 0.6731           &3.9150e+02 &0.2280\\
    \hline
   \end{tabular}\label{tab:2_3}
\end{center}
\end{table}
\begin{table}[!h t b]\fontsize{7.0pt}{12pt}\selectfont
\begin{center}
 \caption{ Numerical results of the BVP (\ref{eq:1.1}), solved by the preconditioned GMRES and Bi-CGSTAB, with $q=0,\, \kappa_{\beta}=1,\,\beta=1/5$, $d=2,$ and $J_0=0$.}
\begin{tabular} {|cc|cc|cc|cc|cc|}  \hline
 $p$ & $J$    &\multicolumn{2}{c|}{GMRES, Inte-}  &\multicolumn{2}{c|}{GMRES, Semi-} &\multicolumn{2}{c|}{Bi-CGSTAB, Inte-}  &\multicolumn{2}{c|}{Bi-CGSTAB, Semi-}  \\\cline{3-10}
      &       & Iter&  CPU(s)      & Iter& CPU(s)      & Inter  & CPU(s)    & Iter &CPU(s)          \\
      \hline
      & $8$   &  13.0    &   0.0094    &27.0&  0.0203                    &8.0     & 0.0077   &19.0  &  0.0198     \\
 $p=1$& $9$   &  13.0    &   0.0115    &28.0&   0.0258                     &9.5     & 0.0133  &20.0  & 0.0272    \\
      & $10$  &  13.0    &   0.0209    &28.0&   0.0452                     &9.5      & 0.0149  &22.0  & 0.0376 \\
      \hline
      &$8$    &  9.0    & 0.0075     &25.0&   0.0227                     & 6.5 & 0.0064  &18.0  &0.0201       \\
$p=1/2$& $9$  &  9.0   &  0.0084     &26.0&   0.0248                    & 7.5 & 0.0087  &20.0    & 0.0267   \\
     & $10$   &  9.0   & 0.0163   &26.0&   0.0370                     & 8.0 &  0.0126  &21.0   & 0.0363    \\
    \hline
   \end{tabular}\label{tab:2_4}
\end{center}
\end{table}

Secondly, we use the MMG to solve the fractional IBVPs (\ref{eq:1.1}) with the exact solution $u(x,t)=\exp(-t)(x^{\nu}-x^2)$, $q=\kappa_{\beta}=1$, and the suitable source term and initial condition. It can be noted that because of the constant diagonal elements of the stiffness matrix, the Richardson and the Jacobi iterations used in MMG are actually equivalent. For $m_1(j)=m_2(j)=1, J_0=3$, the numerical results of CN-MMG are given in Tables \ref{tab:10} and \ref{tab:11},
where `Iter' denotes the average iteration times and `CPU(s)'  the computation time also including the time of the calculation of coefficient matrix $B_j^{n+1},j=J_0,\cdots J$ and the right term. The initial iteration vector at $t_{n+1}$ is chosen as the approximation at $t_n$, and the stopping criterion is
\[\left\|c_J^{n+1,l}-c_J^{n+1,l-1}\right\|_{\infty}\le 2^{-J/2}\times ( 1e-9), \]
where $c_J^{n+1,l}$ is the approximation vector after the $l$ iteration. Of course, the FFT and the FWT are used to accelerate the process. `Gauss(s)' denotes the computation time of the Gaussian elimination method; for fair comparison, the FFT is also used to get the matrix-vector product appeared in the right-hand term at the time $t_{n+1}$.
\begin{table}[!h t b p]\fontsize{5.0pt}{11pt}\selectfont
\begin{center}
 \caption{Numerical results of the IBVP (\ref{eq:1.1}), solved by CN-MMG, with $q=\kappa_{\beta}=1,\,p=1/2$, $\nu=1,\,T=1,$ and $\Delta t=1/2^J$.}
\begin{tabular}{|c c|c c c c|c c c c|}
 \hline
$\omega$           & $J$     &\multicolumn{4}{c|}{$\beta=7/10$}                       &\multicolumn{4}{c|}{$\beta=2/10$}    \\ \cline{3-10}
                    &        &$L_2$-Err    & Iter & CPU(s) & Gauss(s)              &$L_2$-Err    & iter & CPU(s)  &Gauss(s)\\
\hline
                   &$8$      &5.6512e-07    & 5.03  & 1.0782   & 0.2541                   & 7.7427e-07   & 7.97   & 1.6402     &0.2692          \\
$4/(5\lambda_{\max})$&$9$    &1.3673e-07    & 5.00  &3.2092    & 1.6474                   & 1.8554e-07   & 7.00   & 4.3035     &1.6994        \\
                   &$10$     &3.3486e-08    & 4.32  &9.2786   & 18.5861                   & 4.1703e-08   & 6.03   & 12.2416    &18.9016        \\
 \hline
                   &$8$      &5.6512e-07     & 4.00   & 0.8895    &0.2528                   & 7.7431e-07   & 6.00    & 1.2449   & 0.2493        \\
$6/(5\lambda_{\max})$&$9$    &1.3673e-07     & 4.00   & 2.6654   &1.6631                    & 1.8550e-07   & 5.01    & 3.2196   & 1.7100        \\
                   &$10$      &3.3488e-08    & 3.59   &8.0329    &18.7333                   & 4.1709e-08   & 4.90    & 10.3390  & 18.4800      \\
 \hline
\end{tabular}\label{tab:10}
\end{center}
\end{table}

For $p=1/2$, the coefficient matrix is symmetric. And if we choose $\omega<1/\lambda_{\max},\lambda_{\max}=\left(\sigma\left({\rm diag}(B_J^{n+1})B_J^{n+1}\right)\right)$,  then the MMG is convergent and the average iteration number is slightly affected by the choice of $\omega$.  It also seems that the restriction to $\omega$ can be relaxed to some extent in real computation. When $p\ne1/2$, even though there are no strict theoretical prediction, the numerical results show when $\omega\ge \lambda_{\max}$, the iteration may be divergent; see Table \ref{tab:11}. Here, we get the value of $\lambda_{\max}$ by the Matlab function {\em eigs}; it can also be estimated by the Gerschgorin Theorem or the Power method.
\begin{table}[!h t b]\fontsize{5.0pt}{11pt}\selectfont
\begin{center}
 \caption{Numerical results of the IBVP (\ref{eq:1.1}), solved by CN-MMG, with  $q=p=\kappa_{\beta}=1$, $\beta=7/10,\,T=1$, and $\Delta t=1/2^J$.}
\begin{tabular}{|c c|c c c c|c c c|c|}
  \hline
$\nu$         & $J$   &\multicolumn{4}{c|}{$\omega=2/(5\lambda_{\max})$}  &\multicolumn{3}{c|}{$\omega=4/(5\lambda_{\max})$ } &$\omega=6/(5\lambda_{\max})$\\
 \cline{3-10}
              &        &$L_2$-Err    & Iter & CPU(s)  &Gauss(s)                  & Iter &  CPU(s)    &Gauss(s)         &     \\
\hline
              &$8$    & 1.2500e-06   & 14.79     &2.9692   & 0.3550            & 10.95  &2.2017      &0.3667            & $no\  cvge.$     \\
$1$           &$9$    & 3.1242e-07  & 12.95      &7.6719     &3.0322            & 9.80    & 5.8060    &2.9866              & $no\  cvge.$     \\
              &$10$   & 7.9268e-08   &10.80      &20.3230    &31.3422           & 8.13    & 15.4748   &31.0381            & $no \ cvge.$     \\
 \hline
              &$8$   & 1.7059e-06   &13.47     & 2.7015   &0.4289                &10.44    & 2.1180   & 0.4300             & $no\  cvge.$      \\
$11/10$       &$9$   &5.0960e-07    &11.21     & 6.8185   &3.1898                &9.08     & 5.5673     & 3.1589           & $no\  cvge.$     \\
              &$10$  &1.5759e-07    &9.01      & 17.7328  &31.7698               &7.36     & 14.8017    & 31.9835            & $no\  cvge.$     \\
 \hline
\end{tabular}\label{tab:11}
\end{center}
\end{table}

\begin{example}\label{example3}\end{example} In this example, we focus on the previously proposed ad-hoc wavelet adaptive algorithms for the fractional PDEs. The BVP is solved by the biorthogonal wavelet bases produced by ${}^{3,3}\psi$ ($d=3,\tilde{d}=3$), and the IBVP by the semi-interpolation wavelet bases. We first consider the BVP (\ref{eq:1.1}),  the regularity of its exact solution, $u(x)=(1-x)^{11/10}-(1-x)$, is weak at the area close to the right boundary; and the parameters $\kappa_\beta=1, p=0$, and the source term
\[f(x)=-\frac{\Gamma(21/10)(1-x)^{\beta-9/10}}{\Gamma(\beta+1/10)}+\frac{(1-x)^{\beta-1}}{\Gamma(\beta)}.\]
In the algorithm, we take $J_0=3, \epsilon(j)=1e-5$. For every iteration step, the finally extended irregular indexes are obtained by firstly adding the children of all the significant indexes and then including two neighbors, i.e., the right and left neighbors, of each index of the extended irregular indexes.
When $\beta=1/2$, the sets of wavelet indices that corresponding to the adaptively chosen wavelets and the corresponding error $u-\widehat{u}_{J_0+m}$ are presented in Figures \ref{fig:G1}, where the blue bar denotes that we have used all the scaling bases in the coarest level $J_0$.  One can see that the algorithm in fact automatically recognizes the whereabouts of the boundary layer of the solution $u$, and adds wavelets locally to there. It also reveals that the newly added computational costs are spended  in the most needed place, and the large peaks of the errors are successively reduced.
Moreover, for different $\beta$, from the decreasing of the $L_2$ approximation error of the adaptive and uniform Galerkin schemes with the increasing of the freedom $N$ (the loglog coordinate) in Figure \ref{fig:G3}, one can see that the adaptive MGM is remarkably superior to the uniform MGM.
\begin{figure}[!h t b p]
\begin{center}
\includegraphics[width=2.3in,height=1.5in,angle=0]{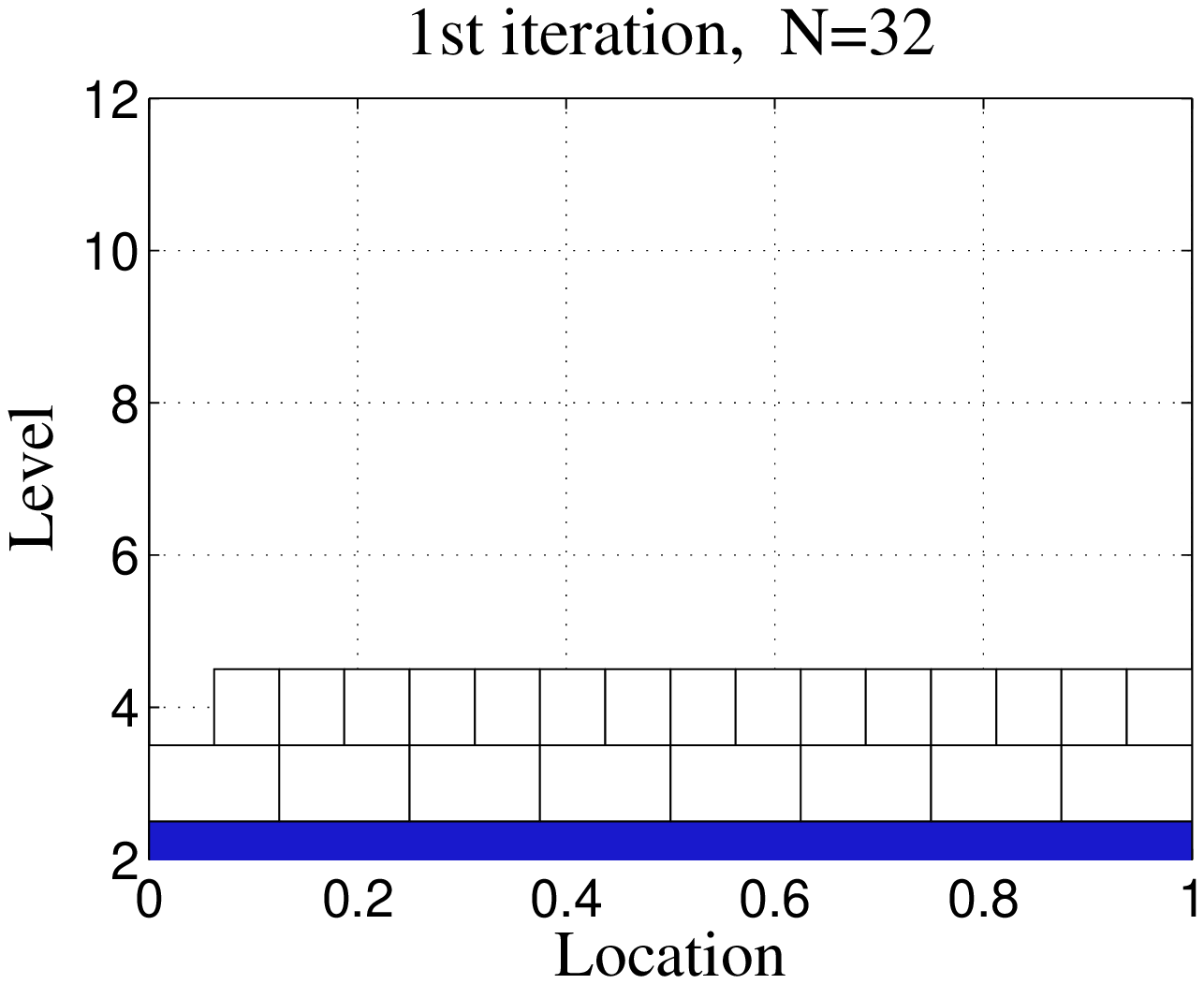}
\includegraphics[width=2.3in,height=1.5in,angle=0]{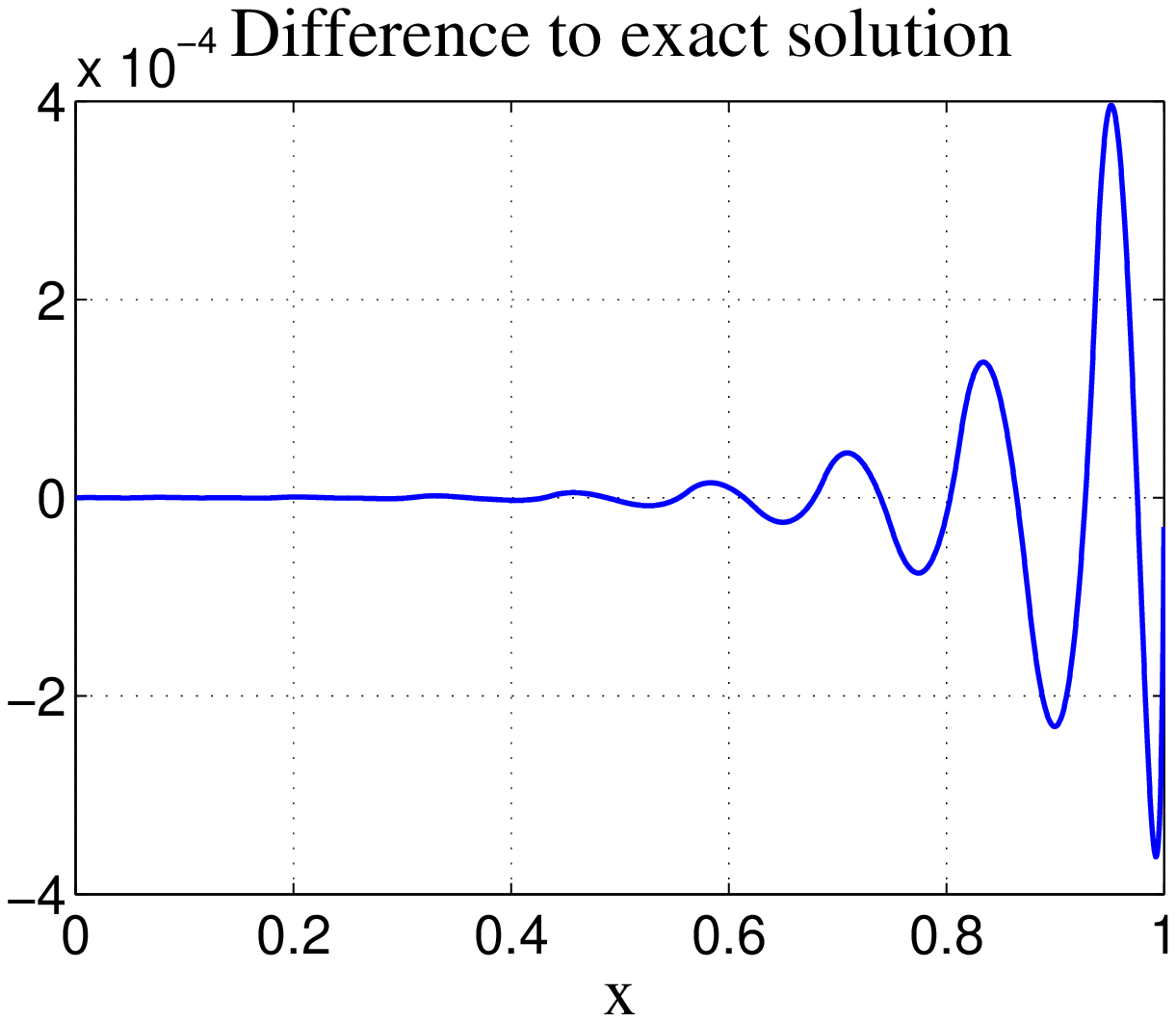}
\includegraphics[width=2.3in,height=1.5in,angle=0]{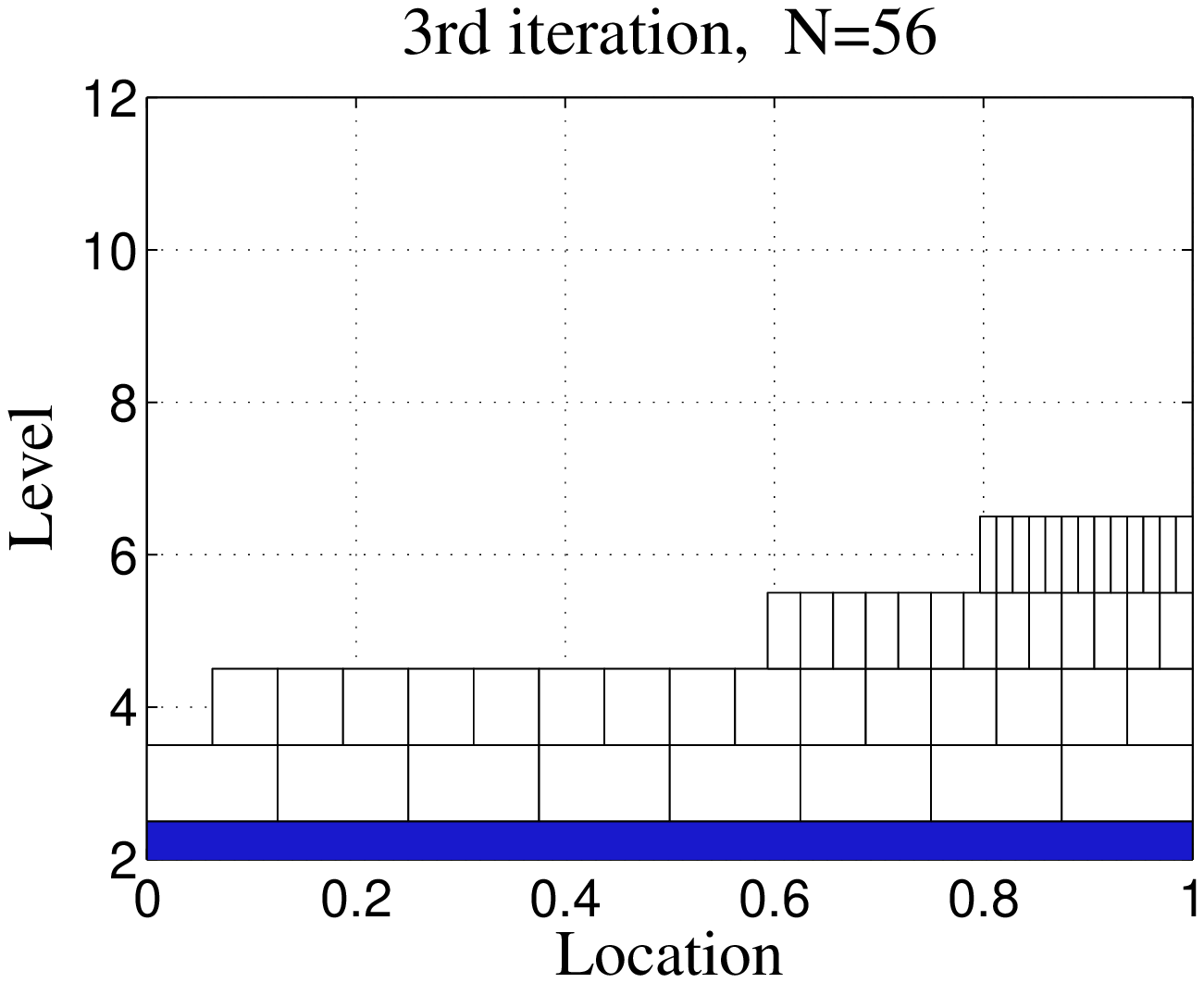}
\includegraphics[width=2.3in,height=1.5in,angle=0]{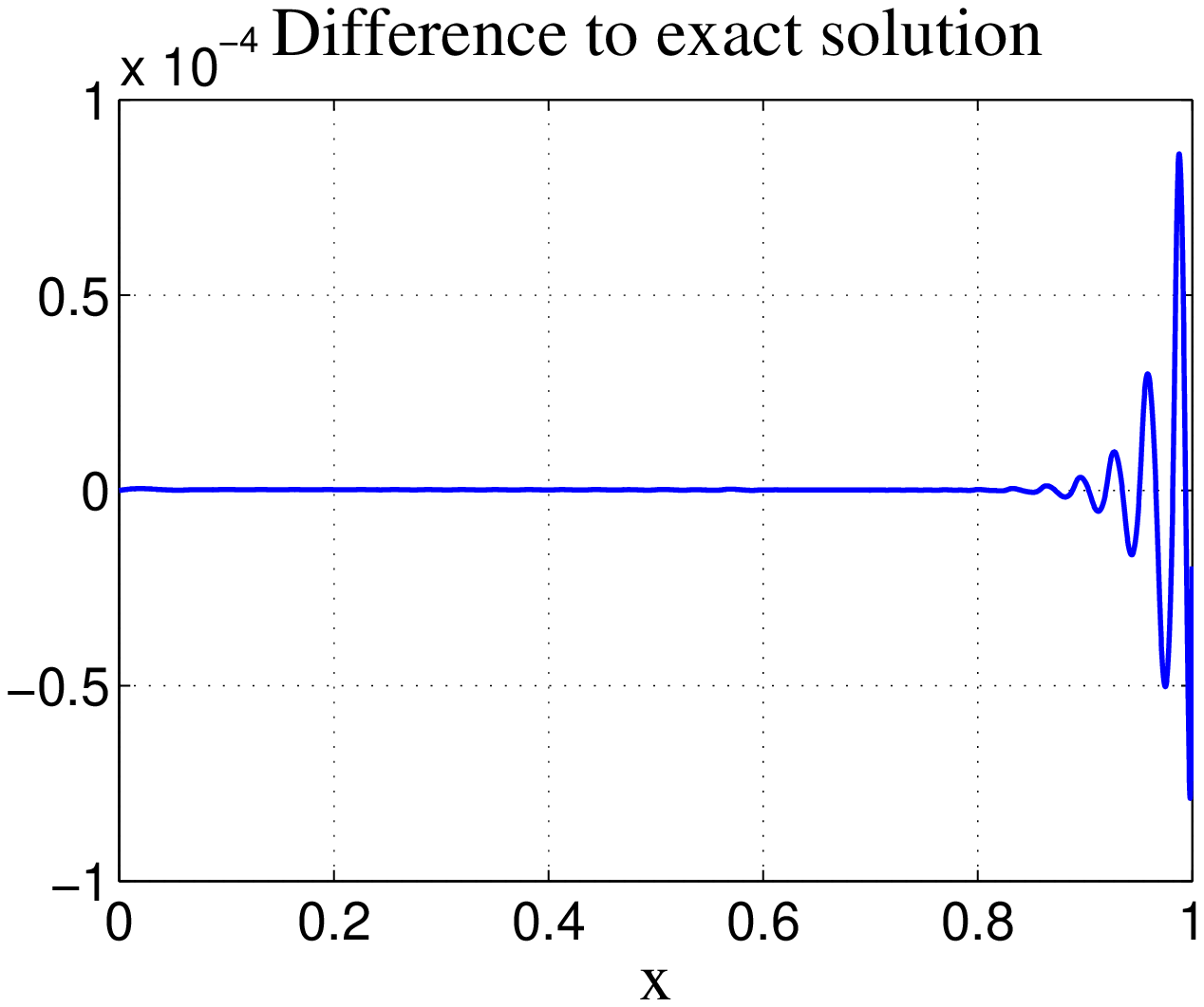}
\includegraphics[width=2.3in,height=1.5in,angle=0]{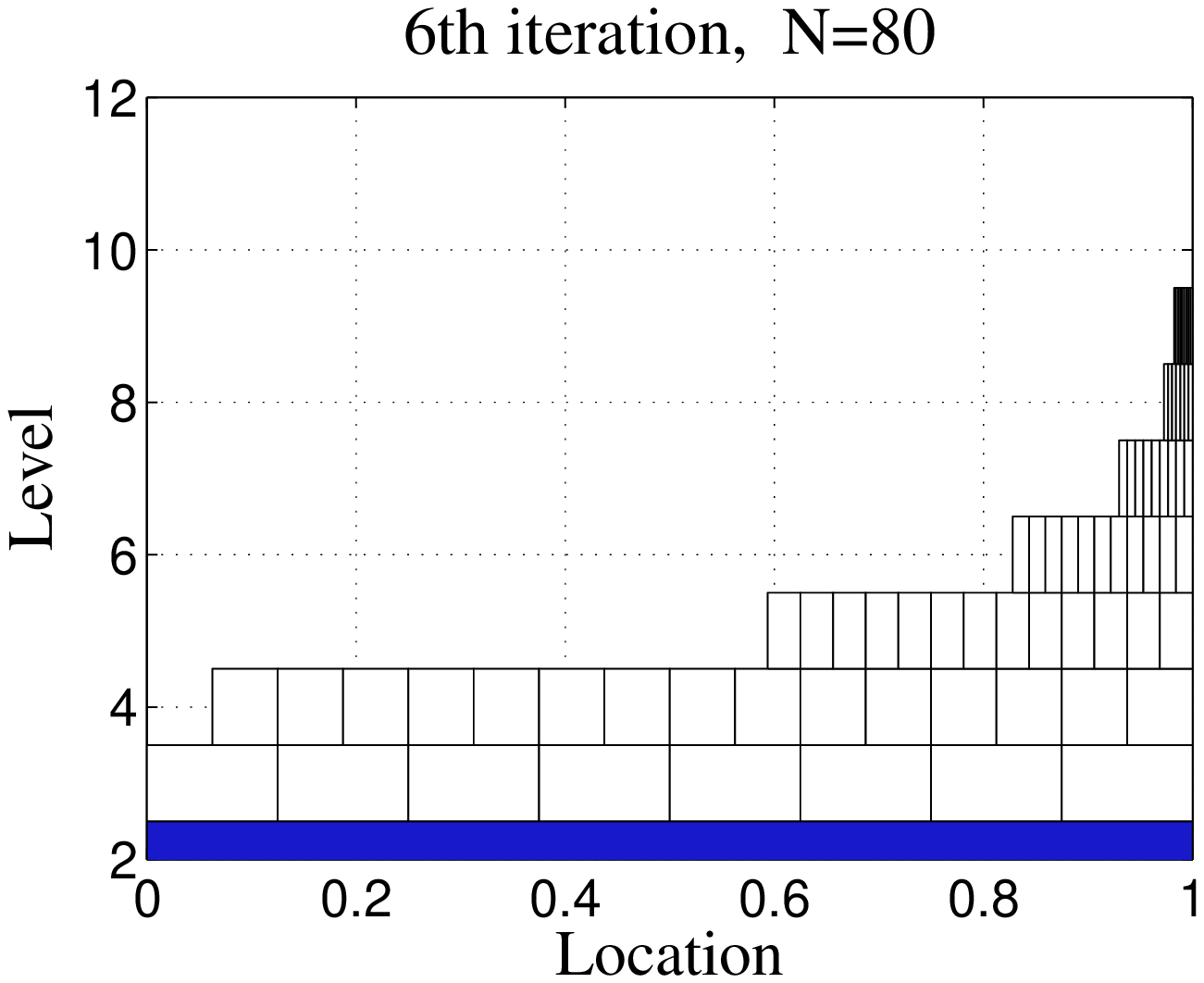}
\includegraphics[width=2.3in,height=1.5in,angle=0]{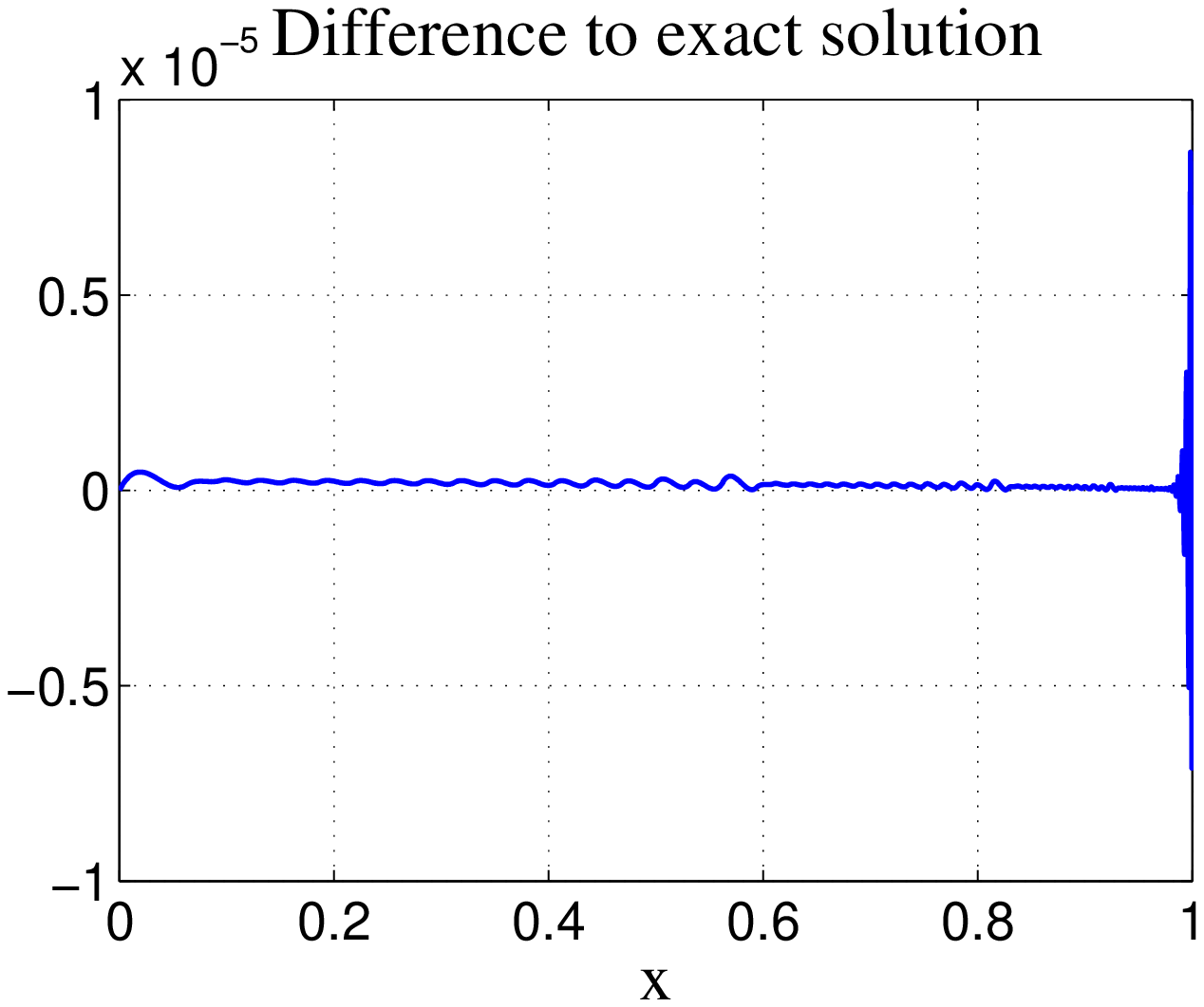}
\includegraphics[width=2.3in,height=1.5in,angle=0]{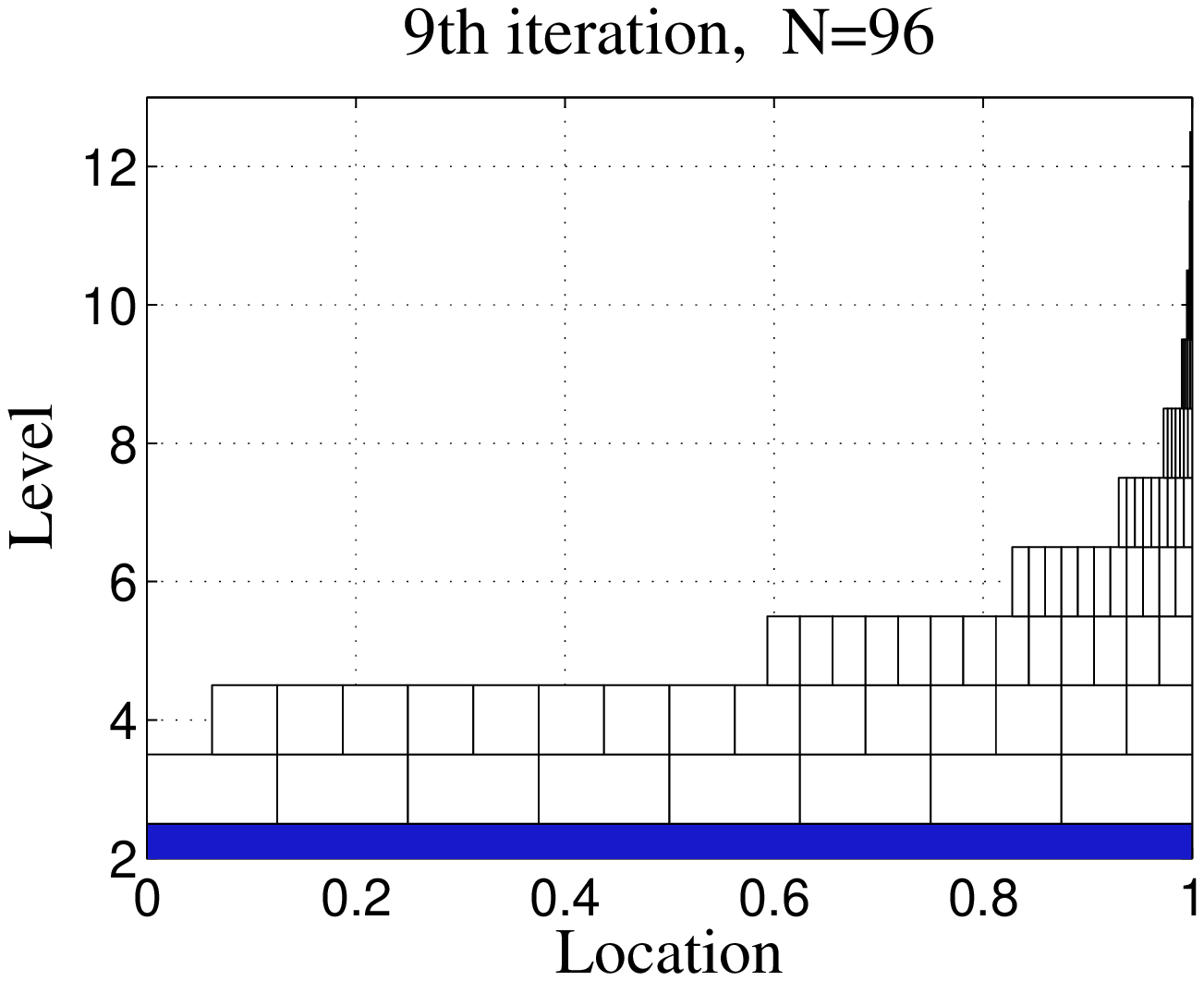}
\includegraphics[width=2.3in,height=1.5in,angle=0]{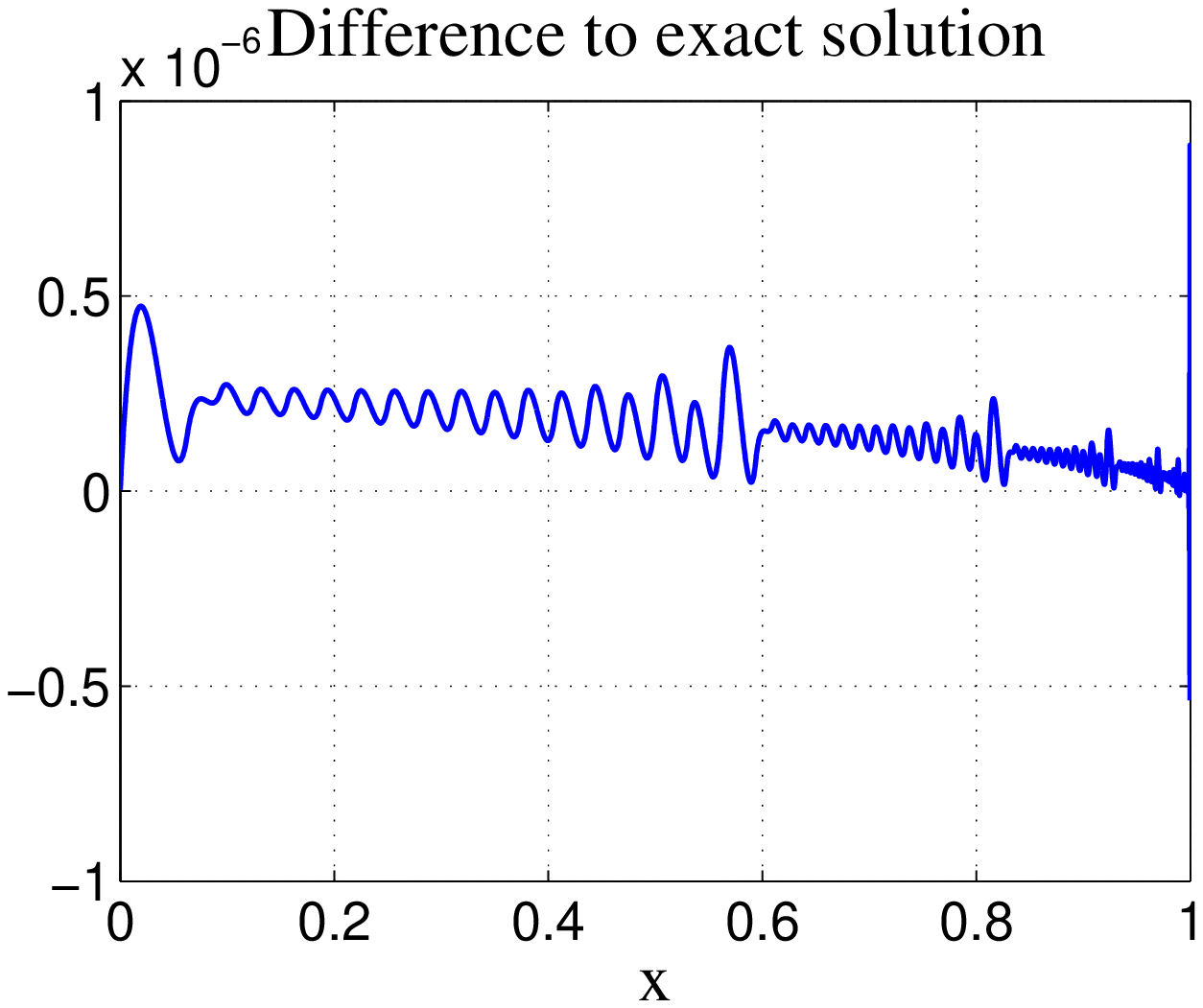}
\caption{Distribution of adaptive wavelet bases and curve of the approximation error gotten by Algorithm \ref{AD-STATIC}.} \label{fig:G1}
\end{center}
\end{figure}
\begin{figure}[!h t b p]
\begin{center}
\includegraphics[width=2.3in,height=1.5in,angle=0]{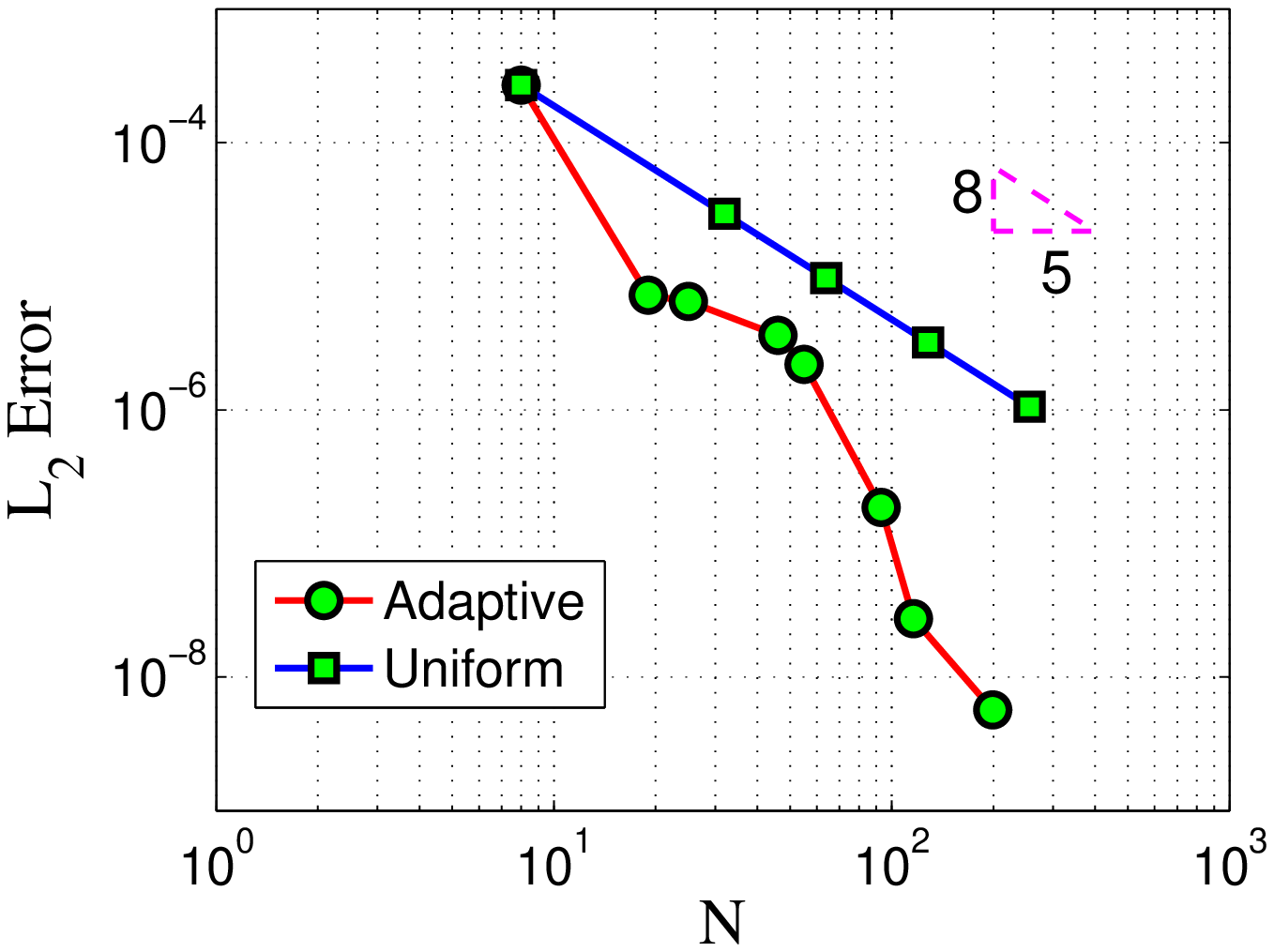}
\includegraphics[width=2.3in,height=1.5in,angle=0]{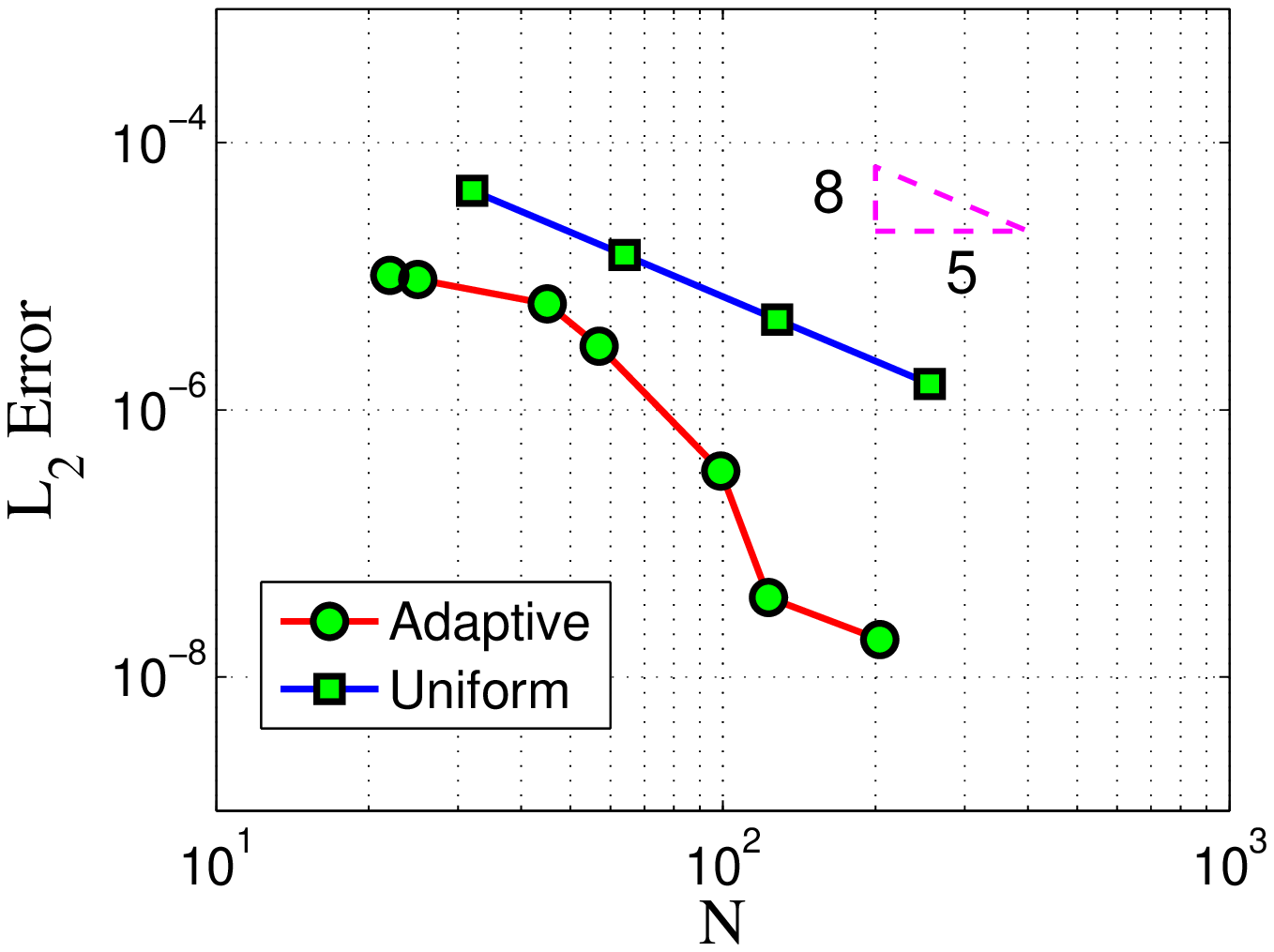}
\caption{$L_2$ errors versus freedom $N$ for the adaptive  and the uniform  Galerkin approximations with $\beta=1/2$ (left) and $\beta=4/5$ (right), respectively.}\label{fig:G3}
\end{center}
\end{figure}

Secondly, we consider the IBVP (\ref{eq:1.1}) with $\kappa_{\beta}=p=1,\,\beta=5/10$, the initial condition $u(x,0)=x^4(1-x)$, and the source term
\begin{eqnarray*}
f(x,t)=\exp(3t)x^{20t+4}(1-x)(3+20\ln x) +\exp(3t)\frac{\Gamma(20t+5)}{\Gamma(20t+3+\beta)}x^{20t+2+\beta}\left(\frac{20t+5}{20t+3+\beta}x-1\right).
\end{eqnarray*}
Its exact solution is $u=\exp(3t)x^{20t+4}(1-x)$, which has a strong gradient at somewhere as shown in Figure \ref{fig:C2} (left).
In the computation, based on Algorithm \ref{AD-TIME} the semi-interpolation adaptive wavelet collocation method is used;
and the time step  $\Delta t=1/2^{2J_{max}}$, $J_0=3, J_{max}=10, \epsilon(j)=1e-5$, and $T=1$.  In this adaptive algorithm, for every time step,
the index extension techniques being used are the same as the ones for the BVP, and the chosen collocation points are just the ones corresponding to
the reserved wavelet bases. For $t=1$, the adaptive solution and the distribution of semi-interpolation wavelets  are displayed
in Figure \ref{fig:C1}. Further seeing the global picture, Figure \ref{fig:C2} (right), one  can easily notice that the high level wavelets and collocation points mainly concentrate on the area with steep gradient, being exactly as what we have desired.
\begin{figure}[!h t b p]
\begin{center}
\includegraphics[width=2.4in,height=1.6in,angle=0]{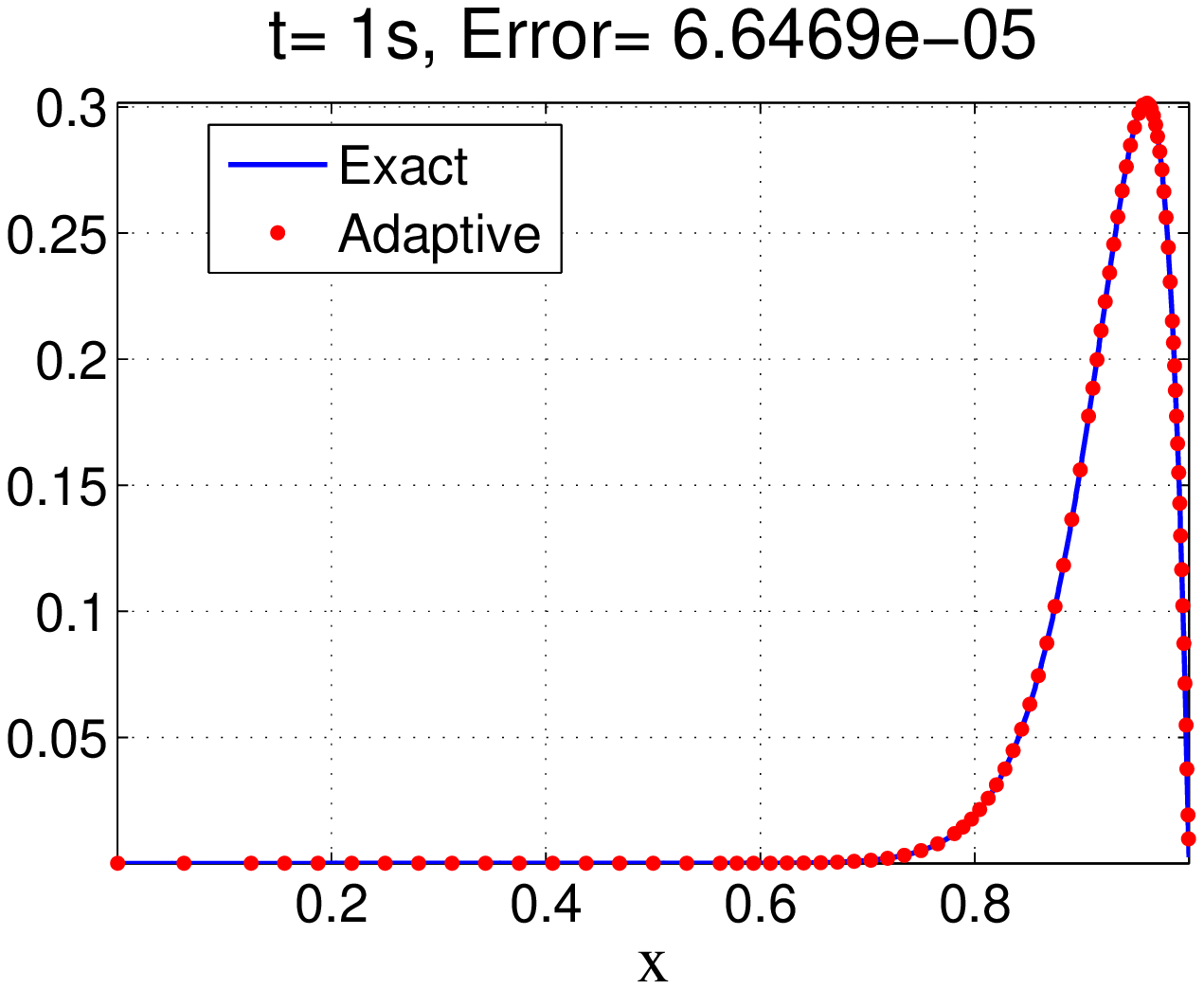}
\includegraphics[width=2.4in,height=1.6in,angle=0]{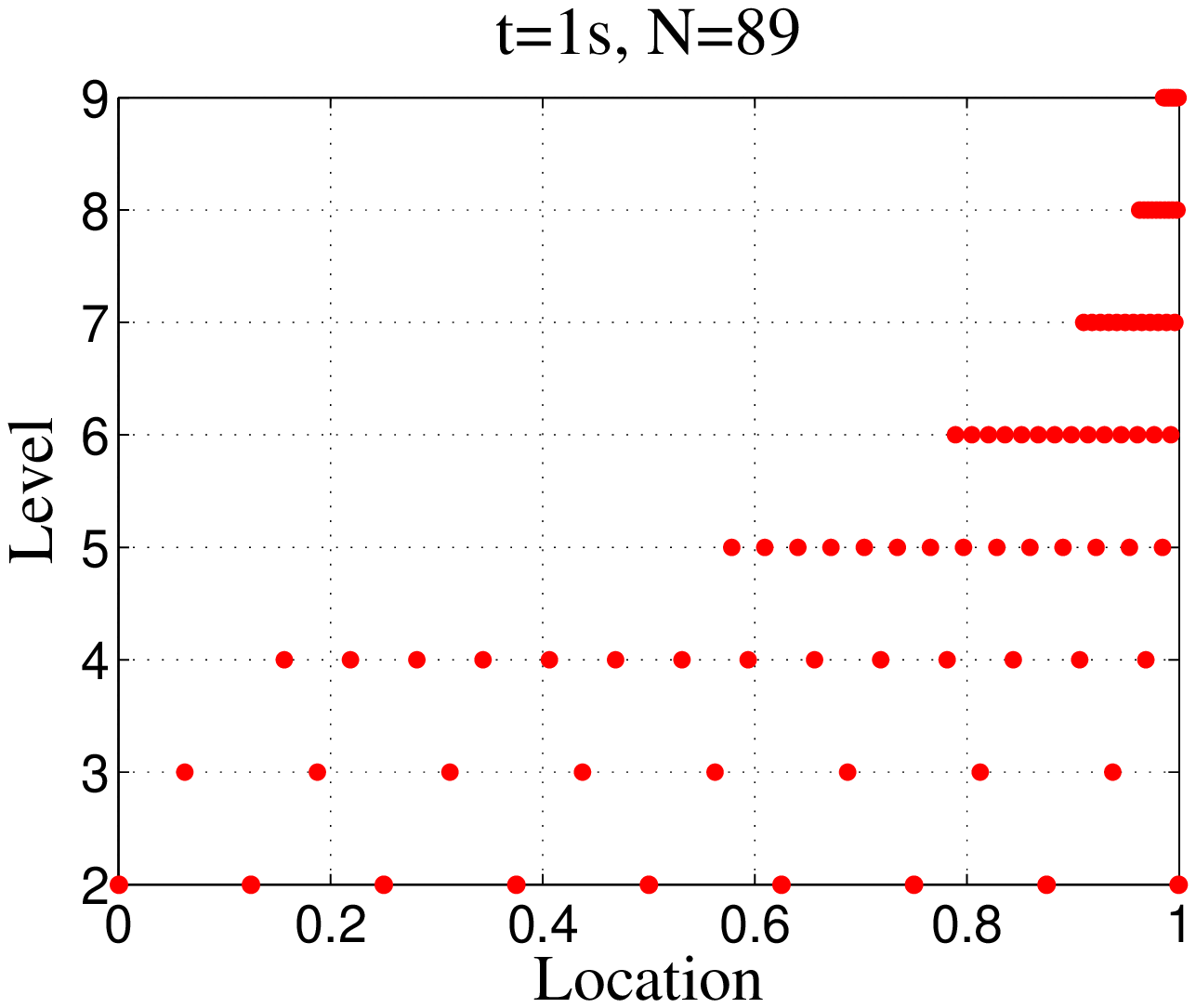}
\caption{The adaptive solution ($t=1$) and the corresponding distribution of the semi-interpolation wavelets  gotten by Algorithm \ref{AD-TIME}.} \label{fig:C1}
\end{center}
\end{figure}
\begin{figure}[!h t b p]
\begin{center}
\includegraphics[width=2.4in,height=1.6in,angle=0]{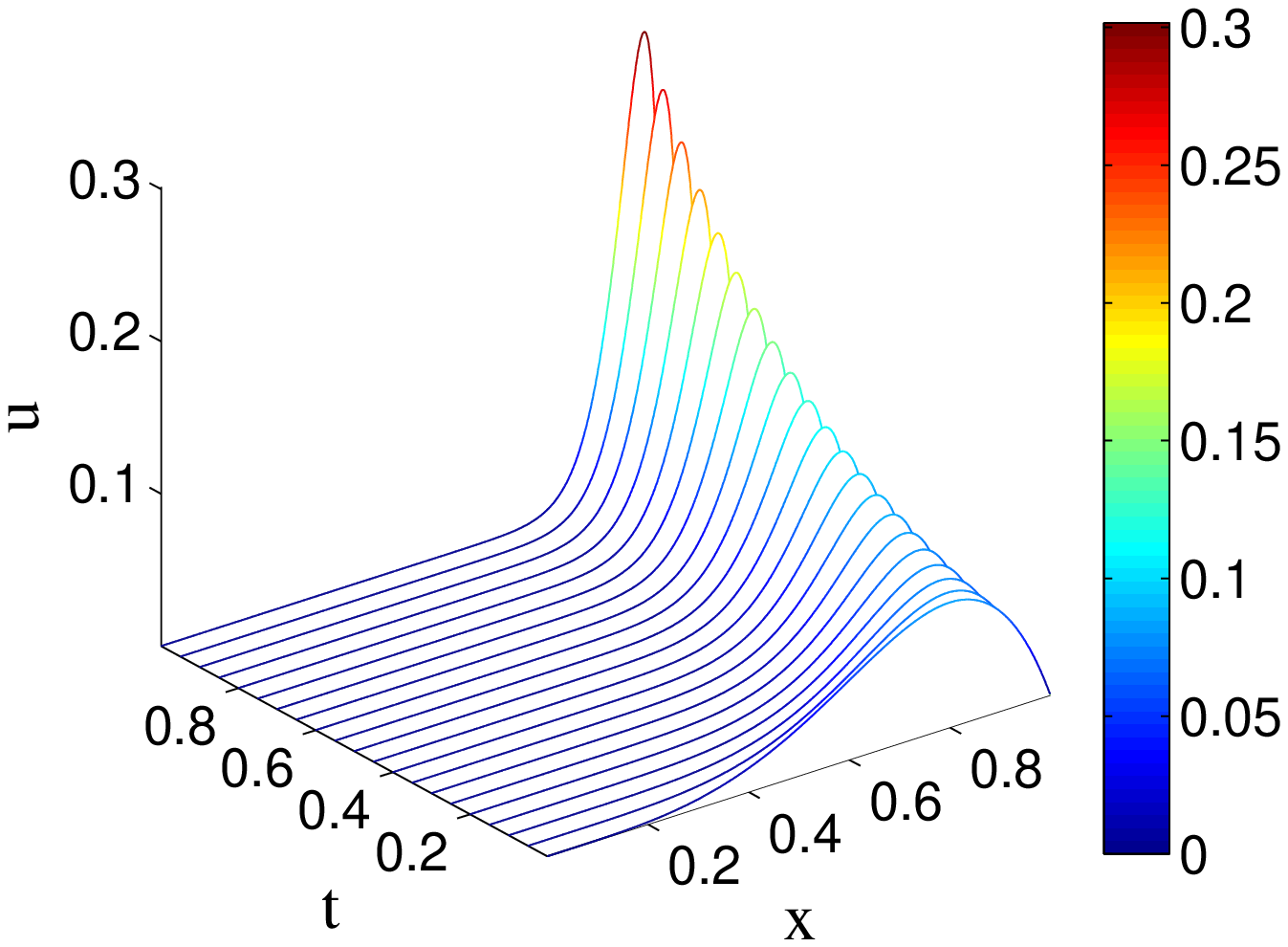}
\includegraphics[width=2.4in,height=1.6in,angle=0]{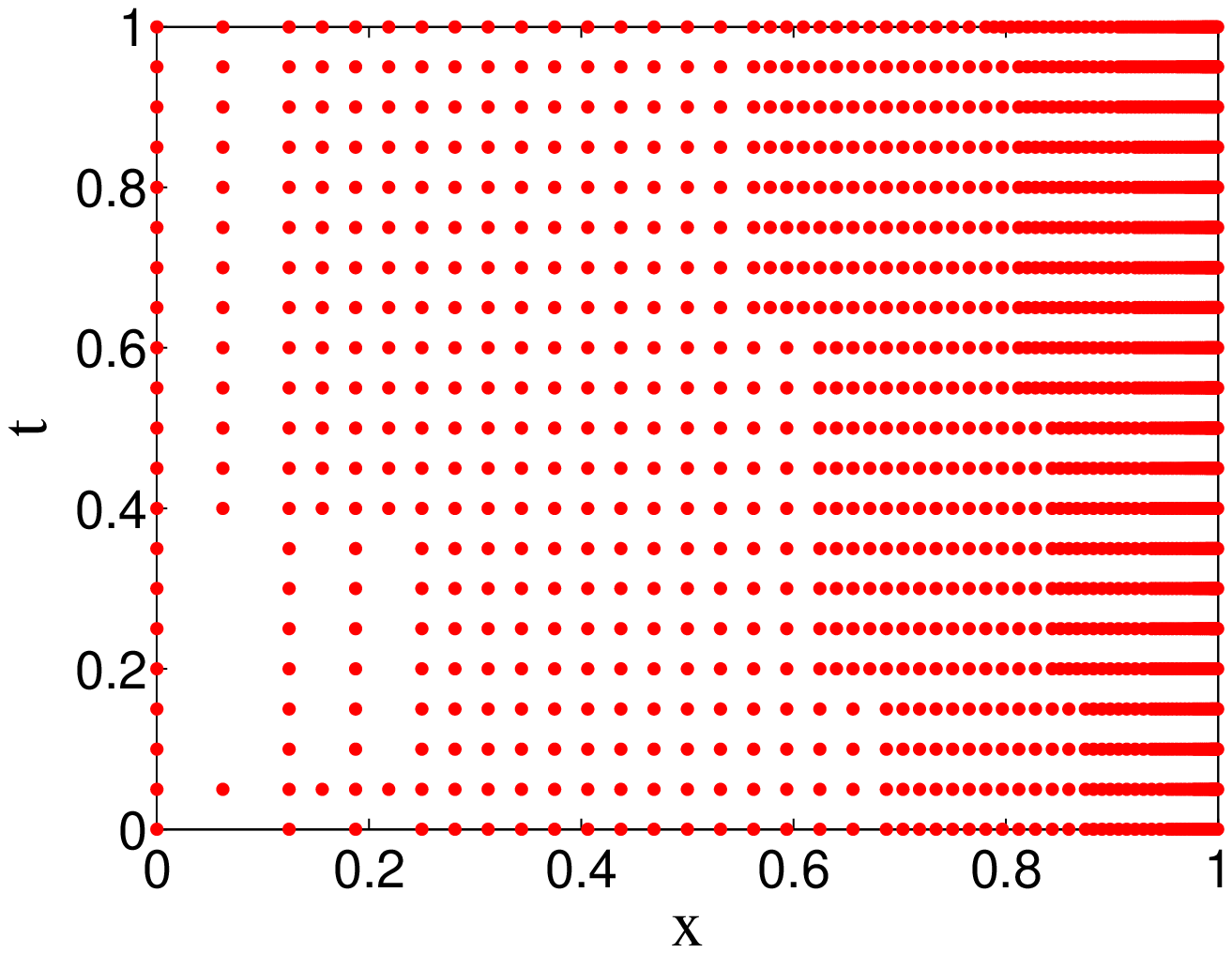}
\caption{Evolution of the solutions and corresponding distribution of the collocation points gotten by Algorithm \ref{AD-TIME}.} \label{fig:C2}
\end{center}
\end{figure}

\section{Conclusion and Discussion}
This paper focuses on digging out the potential benefits, providing the techniques, and performing the theoretical analysis and extensive numerical experiments in solving the fractional PDEs by wavelet numerical methods. The multiscale (wavelet) bases show their strong advantages in treating the fractional operators which essentially arise from the multiscale problem. Even the scaling bases also display their powerfulness in saving computational cost when generating stiffness matrix, i.e., by using the scaling bases, the stiffness matrix has the Toeplitz structure. The way of generating effective preconditioner is presented for time-independent problem and multigrid scheme for time dependent problem is detailedly discussed. We numerically show that the heuristic wavelet adaptive scheme works very well for fractional PDEs; in particular, it is still easy to get the local regularity indicator even for the fractional (nonlocal) problem; and the algorithm descriptions are provided.

After finishing this work, one of the directions of our further research appears, i.e., applying the wavelet compression property to fractional operator.
A key difference between the fractional and classical operators is that the former is non-local, and then both the matrixes generated by the scaling and the multiscale bases are no longer sparse. Fortunately, the wavelet compression not only allows one to obtain a sparse representation of functions, but it seems also effective for the fractional operators. Considering the discretization of the operator:
\[{\bf A}u=-D\left(\frac{2}{3}{}_0D_x^{-\beta} + \frac{1}{3} {}_xD_1^{-\beta}\right)Du\]
 in the approximation space $S_J$ with $J=10$, we first compute the matrix $A_J$ or $\hat{A}_J$ (here the multiscale wavelet bases also have been normalized with $D$, proposed in Section 3 ). Then we get the compressed matrix by setting all entries of $A_J$ or $\hat{A}_J$ with modulus less than $\epsilon=10^{-4}\times 2^{-J}$ to zeros. The comparison results are displayed in Table \ref{tab:115} and Figure \ref {fig:6.1} ,  where ($\cdot\,\%$) denotes the percentage of the non-zero entries of the compressed matrix. It can be seen that many entries in $\hat{A}_J$ are so small that they can be omitted to retrieve the famous finger structure, whereas essentially all entries in $A_J$ are significant. In the future, we will investigate the effective ways of using wavelet compression to get the paralleled sparse approximate inverse (SPAI) preconditioner and to perform the low-cost multiscale matrix-vector product.
\begin{figure}[!h t b]
\begin{center}
\includegraphics[width=2.4in,height=1.50in,angle=0]{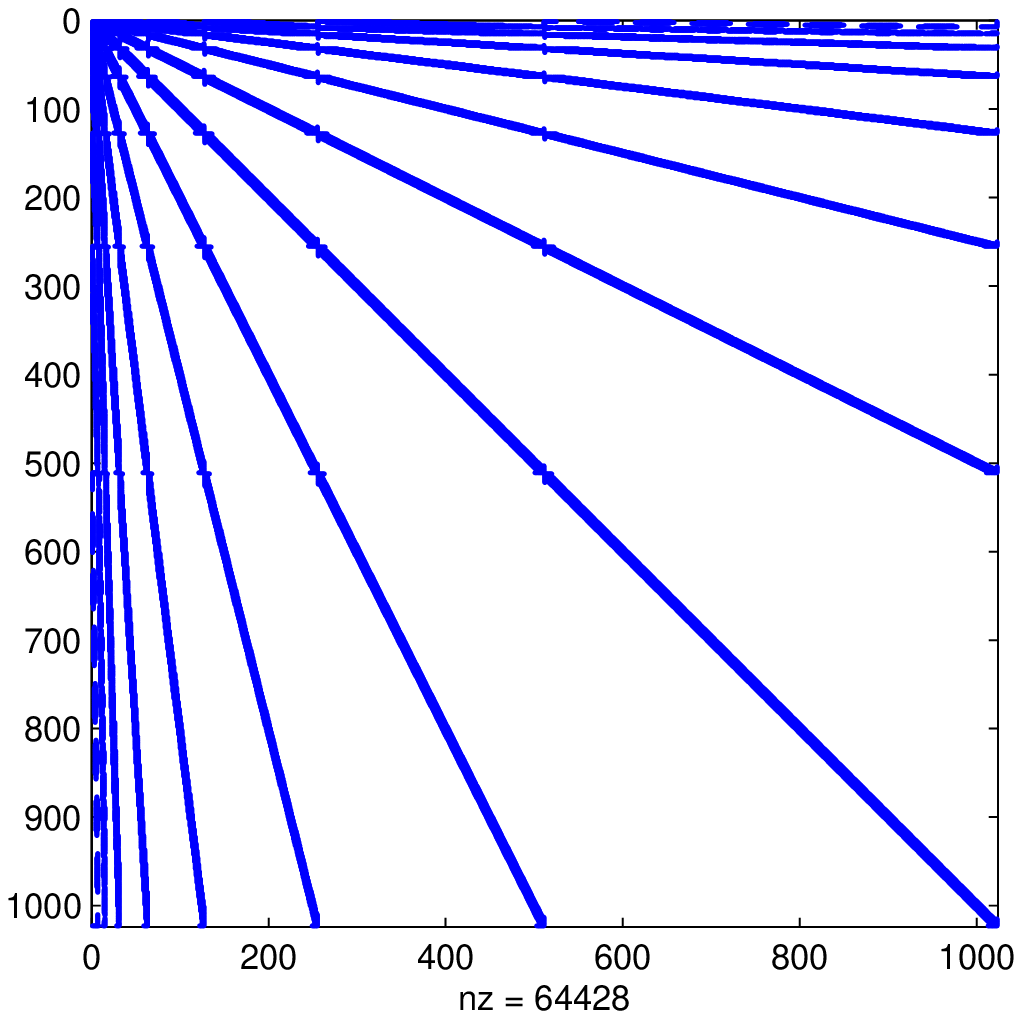}\hspace{-4pt}
\includegraphics[width=2.4in,height=1.50in,angle=0]{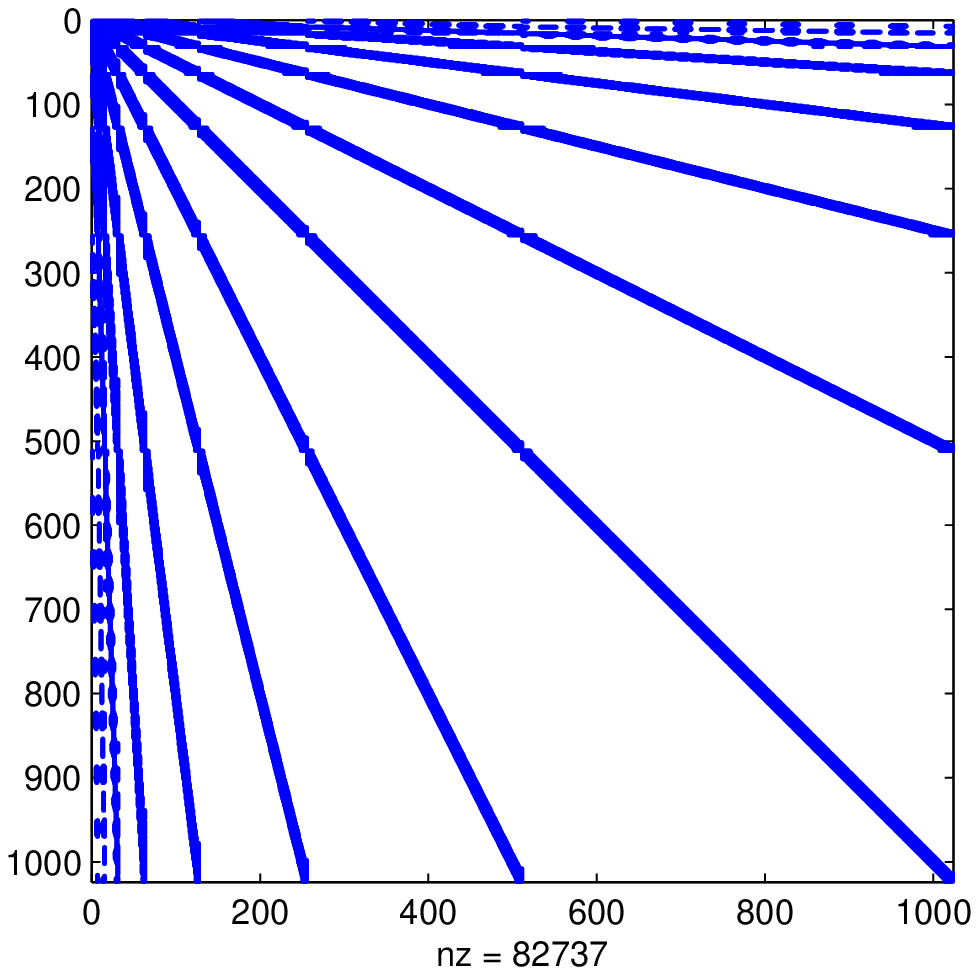}
\caption{Sparsity patterns of the compression matrix $\hat{A}_J$ by using the semiorthogonal with $\beta=5/10$ and $d=2$ (left) and the biorthogonal wavelet ${}^{2,4}\psi$ with $d=2$ and $\tilde{d}=4$ bases (right).} \label{fig:6.1}
\end{center}
\end{figure}

\begin{table}[h t b p]\fontsize{6.0pt}{11pt}\selectfont
\begin{center}
 \caption{Compression capacity of the different bases for operator $\bf{A}$.} \label{tab:115}
\begin{tabular}{|l|c c |c c|c c|lc|}
  \hline
  $\qquad\quad\psi$      &\multicolumn{2}{c|}{$\beta=8/10$}  &\multicolumn{2}{c|}{$\beta=5/10$} & \multicolumn{2}{|c|}{$\beta=2/10$} &\multicolumn{2}{c|}{Note}\\
  \cline{2-9}
                     & $A_J$& $\hat{A}_J $     &$A_J$   &$ \hat{A}_J $   & $A_J$& $ \hat{A}_J $         &\quad wavelet &compression\\
  \hline
  $d=2, Inte-$        & 99.80\%   & 99.07\%     &99.80\%     & 80.38     & 99.80\%      & 48.04\%        & interpolation    &No/Yes\\
  $d=2, Semi- $       & 99.80\%   & 6.66\%      & 99.80\%    & 6.14\%    & 99.80\%      &5.31\%          & semiorthogonal     &Yes\\
  $d=2,\tilde{d}=4$   &99.80\%    & 8.15\%      & 99.80\%    & 7.89\%    & 99.80\%      &7.15\%          & biorthogonal       &Yes\\
  $d=3,\tilde{d}=3$   & 100\%     & 10.18\%     & 100\%      & 9.52\%     & 100\%      &8.30\%           & biorthogonal       &Yes \\
 \hline
\end{tabular}
\end{center}
\end{table}
\section*{Appendix} Here we present the techniques for computing the left fractional derivative of the base function $\Phi_J=\{\phi_{J,k},k\in\triangle_J, d=4\}$. Besides $\phi(x)$ and $\phi_b(x)$ given by (\ref{Cubic:1}) and (\ref{Cubic:2}), define
$$
  \phi_{a}(x)=3x_+-\frac{9}{2}x_+^2+\frac{7}{4}x_+^3-2(x-1)_+^3+\frac{1}{4}(x-2)_+^3.
$$
Then $\Phi_J=2^{J/2}\big\{\phi_a(2^Jx),\phi_b(2^Jx), \phi(2^Jx-k) \big| _{k=0}^{2^J-4}, \phi_b\left(2^J(1-x)\right), \phi_a\left(2^J(1-x)\right)\big\}$ is a Riesz bases of $S_J$. For $x_0\ge 0$, it is easy to check that
     \[{}_0 D_x^{1-\beta}\left(H(x-x_0)v(x)\right)=H(x-x_0){}_{x_0}D_x^{1-\beta}v(x), \]
  where $H(x)$ denotes the Heaviside function. By the well-known formulae
  \begin{eqnarray*}
   &\ &{}_aD_x^{1-\beta}(x-a)^{\nu}=\frac{\Gamma(\nu+1)(x-a)^{\nu+\beta-1}}{\Gamma(\nu+\beta)}, \quad \nu\in \mathbb{N}, \\
   &\ &(b-ax)^k_+=(b-ax)^k+(-1)^{k-1}(ax-b)_+^k,\quad k\in \mathbb{N}^+,
   \end{eqnarray*}
for $b/a\ge0, \, k\in \mathbb{N}^+$, there exist
  \begin{eqnarray*}
  && {}_0 D_x^{1-\beta}(ax-b)_+^k=a^{2-\beta}\frac{\Gamma(k+1)}{\Gamma(k+\beta-1)}(ax-b)_+^{k+\beta-1},\\
   &&{}_0 D_x^{1-\beta}(b-ax)_+^k=(-1)^{k-1}{}_0 D_x^{1-\beta}(ax-b)_+^k+\sum_{m=0}^k{k\choose m} a^kb^{k-m}\frac{m!(-1)^m}{\Gamma(m+\beta)}x_+^{m+\beta-1},
   \end{eqnarray*}
 Define
\begin{eqnarray*}
 &&M_1(x):={}_0 D_x^{1-\beta}\phi_{a}(x)=\frac{3}{\Gamma(\beta+2)}\left((\beta+1) x_+^{\beta}-3x_+^{\beta+1}+\frac{7}{2(\beta+2)}x_+^{\beta+2}\right)\\
 &&\qquad\qquad\qquad\qquad\qquad+\frac{3}{2\Gamma(\beta+3)}\left(-8(x-1)_+^{\beta+2}+(x-2)_+^{\beta+2}\right),\\
 && M_2(x):={}_0 D_x^{1-\beta}\phi_{b}(x)=\frac{1}{\Gamma(\beta+2)}\left(3x_+^{\beta+1}-\frac{11}{2(\beta+2)}x_+^{\beta+2}\right)\\
 &&\qquad\qquad\qquad\qquad\qquad+\frac{1}{2\Gamma(\beta+3)}\left(18(x-1)^{\beta+2}_+-9(x-2)_+^{\beta+2}+2(x-3)_+^{\beta+2}\right),\\
 && M_3(x):={}_0 D_x^{1-\beta}\phi(x)=\frac{1}{\Gamma(\beta+3)}\sum_{i=0}^4{4\choose i}(-1)^i(x-i)_+^{\beta+2},\\
  &&M_4(x,l):={}_0 D_x^{1-\beta}\phi_{b}(l-x)=\frac{-1}{\Gamma(\beta+2)}\left(3(x-l)_+^{\beta+1}+\frac{11}{2(\beta+2)}(x-l)_+^{\beta+2}\right)\\
 &&\qquad\qquad\qquad+\,\frac{1}{2\Gamma(\beta+3)}\!\Big(18(x-l+1)^{\beta+2}_+-9(x-l+2)_+^{\beta+2}+2(x-l+3)_+^{\beta+2}\Big).
\end{eqnarray*}
Then we have
\begin{eqnarray*}
&\,&{}_0 D_x^{1-\beta}\left(2^{J/2}\phi_{a_i}(2^Jx)\right)=2^{J(3/2-\beta)}M_i\left(2^Jx\right),\quad (i,a_i)=(1,a)\, {\rm or}\, (2,b),\\[4pt]
&\,&{}_0 D_x^{1-\beta}\left(2^{J/2}\phi(2^Jx-k)\right)=2^{J(3/2-\beta)}M_3\left(2^Jx-k\right),\\
&\,&{}_0 D_x^{1-\beta}\left(2^{J/2}\phi_{b}\left(2^J(1-x)\right)\right)=2^{J(3/2-\beta)}M_4\left(2^Jx,2^J\right).
\end{eqnarray*}
The similar formulae can also be derived for $2^{J/2}\phi_a(2^J(x-1))$.

\end{document}